\UseRawInputEncoding
\documentclass{article}

\usepackage{stmaryrd,graphicx}
\usepackage{amssymb,mathrsfs,amsmath,amscd,amsthm,color}
\usepackage{tikz-cd}
\usepackage{float}
\usepackage{caption}
\usepackage{mathabx}
\usepackage[all,cmtip]{xy}
\DeclareMathAlphabet{\mathpzc}{OT1}{pzc}{m}{it}
\usepackage{amsfonts,latexsym,wasysym}
\usepackage[shortlabels]{enumitem}
\setlist{nosep}
\setenumerate{label=(\arabic*)}

\renewcommand{\int}{\operatorname{int}}

\newcommand{\im}{\operatorname{Im}}
\newcommand{\lpc}{\operatorname{lpc}}
\newcommand{\bft}{\mathbf{T}}

\newcommand{\nt}{\mathbf{nt}}
\newcommand{\ntij}{\mathbf{nt}_{\infty,j}}
\newcommand{\ntkj}{\mathbf{nt}_{k,j}}
\newcommand{\bfa}{\mathbf{a}}

\newcommand{\bfe}{\mathbf{e}}

\newcommand{\sw}{\bigcurlyvee}

\newcommand{\tY}{\wt{Y}}
\newcommand{\ty}{\wt{y}}
\newcommand{\tZ}{Z}
\newcommand{\tz}{z}

\newcommand{\ilim}{\varprojlim}

\newcommand{\hx}{\widehat{x}}

\newcommand{\hX}{\widehat{X}}

\newcommand{\ui}{I}

\newcommand{\xleqk}{X_{\leq k}}

\newcommand{\wh}{\widehat}

\newcommand{\bfx}{\mathbf{x}}
\newcommand{\tx}{\tilde{x}}
\newcommand{\tX}{\widetilde{X}}
\newcommand{\wt}{\widetilde}

\newcommand{\mca}{\wt{\mathcal{X}}}
\newcommand{\mcb}{\wt{\mathcal{Y}}}
\newcommand{\mcc}{\mathcal{C}}

\newcommand{\mcd}{\mathcal{D}}

\newcommand{\mcp}{\mathcal{P}}

\newcommand{\mcu}{\mathcal{U}}
\newcommand{\mcv}{\mathcal{V}}

\newcommand{\scru}{\mathscr{U}}

\newcommand{\scrt}{\mathscr{T}}
\newcommand{\scrv}{\mathscr{V}}

\newcommand{\tf}{\wt{f}}
\newcommand{\tg}{\wt{g}}

\newcommand{\bbe}{\mathbb{E}}

\newcommand{\bbn}{\mathbb{N}}

\newcommand{\bbr}{\mathbb{R}}

\newcommand{\bbt}{\mathbb{T}}
\newcommand{\bbz}{\mathbb{Z}}

\newcommand{\bfee}{\mathbf{E}}
\newcommand{\mct}{\mathcal{T}}

\newcommand{\ov}{\overline}

\newtheorem{theorem}{Theorem}[section]
\newtheorem{lemma}[theorem]{Lemma}
\newtheorem{proposition}[theorem]{Proposition}
\newtheorem{corollary}[theorem]{Corollary}
\theoremstyle{definition}\newtheorem{definition}[theorem]{Definition}
\newtheorem{example}[theorem]{Example}

\newtheorem{remark}[theorem]{Remark}
\newtheorem{problem}[theorem]{Problem}

\begin{document}
\title{Homotopy groups of shrinking wedges of non-simply connected CW-complexes}
\author{Jeremy Brazas}
\date{\today}

\maketitle

\begin{abstract}
In this paper, we study the homotopy groups of a shrinking wedge $X$ of a sequence $\{X_j\}$ of non-simply connected CW-complexes. Using a combination of generalized covering space theory and shape theory, we construct a canonical homomorphism $$\Theta:\pi_n(X)\to\prod_{j\in\mathbb{N}}\bigoplus_{\pi_1(X)/\pi_1(X_j)}\pi_n(X_j),$$ characterize its image, and prove that $\Theta$ is injective whenever each universal cover $\widetilde{X}_j$ is $(n-1)$-connected. These results (1) provide a characterization of the $n$-th homotopy group of the shrinking wedge of copies of $\mathbb{RP}^n$, (2) provide a characterization of $\pi_2$ of an arbitrary shrinking wedge, and (3) imply that a shrinking wedge of aspherical CW-complexes is aspherical.
\end{abstract}

\tableofcontents

\section{Introduction}

The shrinking wedge of a sequence $X_1,X_2,X_3,\dots$ of based spaces, which we will denote as $\sw_{j\in\bbn}X_j$, is the usual one-point union $\bigvee_{j\in\bbn}X_j$ but equipped with a topology coarser than the weak topology. In particular, every neighborhood of the wedgepoint $x_0$ contains $X_j$ for all but finitely many $j\in\bbn$. For example, $\bbe_m=\sw_{j\in\bbn}S^m$ is the \textit{$m$-dimensional earring space}, which embeds in $\bbr^{m+1}$. While fundamental groups of shrinking wedges of connected CW-complexes are well-understood \cite{MM,Edafreesigmaproducts}, general methods for characterizing higher homotopy groups remain elusive. It remains an open problem to establish an ``infinite Hilton-Milnor Theorem" that would provide a characterization of $\pi_n(\bbe_m)$, $n>m$.

In 1962, Barratt and Milnor proved that the rational homology groups of $\bbe_2$ are non-trivial (and even uncountable) in arbitrarily high dimension \cite{BarrattMilnor}. Compare this with the fact that the reduced homology groups (with any coefficients) of $\bigvee_{j\in\bbn}S^2$ are only non-trivial in dimension $2$. The apparently ``anomalous" behavior of $\bbe_2$ is due to the effect of natural, non-trivial, infinitary operations in the higher homotopy groups $\pi_n(\bbe_2)$, $n>2$ and the fact that standard homology groups are only ``finitely commutative." 

When local structures in a space allow one to form geometrically represented infinite products in homotopy groups, standard methods in homotopy theory fail to apply. Thus other methods, e.g. shape theory \cite{MS82}, generalizations of covering space theory \cite{Brazcat,FZ07}, and infinite word theory \cite{Edafreesigmaproducts} are often required. Since infinite products in $\pi_n$ are formed ``at a point," shrinking wedges present an important case that informs more general scenarios.

In the past two decades some progress has been made toward an understanding of the higher homotopy groups of shrinking wedges. In \cite{EK00higher}, Eda and Kawamura show that $\bbe_m$ is $(m-1)$-connected and $\pi_m(\bbe_m)\cong \bbz^{\bbn}$. In \cite{Kawamurasuspensions} it is shown that $\pi_{n}(\bbe_m)$ splits as $\pi_{n+1}((S^m)^{\bbn},\bbe_m)\oplus \pi_{n}(S^m)^{\bbn}$ for $n>m$. However, new methods will be need to characterize the elements of $\pi_{n+1}((S^m)^{\bbn},\bbe_m)$. Some ad-hoc approaches have also appeared, e.g. to show the second homotopy group of the shrinking wedge of tori (see Figure \ref{fig1}) is trivial \cite{EKRZSnake}. More recently, the results of Eda-Kawamura were extended in \cite{Braznsequential} to other kinds of attachment spaces constructed by attaching a shrinking sequence of spaces to a fixed one-dimensional ``core" space, e.g. attaching a shrinking sequence of spheres to an arc, dendrite, Sierpinski carpet, etc. In this paper, we continue the effort to better understand the higher homotopy groups of shrinking wedges. Using a combination of shape theory and generalized covering space theory and the results of \cite{Braznsequential}, we establish methods that characterize the effect of the fundamental group $\pi_1(\sw_{j\in\bbn}X_j)$ on $\pi_n(\sw_{j\in\bbn}X_j)$, $n\geq 2$.

To provide more context for the statement of our main result, we briefly recall a standard argument for the usual one-point union. If $X=\bigvee_{j\in\bbn}X_j$ is a wedge of non-simply connected CW-complexes, the universal covering space $\tX$, consists of copies of $\tX_j$ (indexed by the coset space $\pi_1(X)/\pi_1(X_j)$) that are attached to each other in a tree-like fashion that matches the reduced-word structure of the free product $\pi_1(X)=\ast_{j\in\bbn}\pi_1(X_j)$. If $T$ is a maximal tree in the 1-skeleton of $\tX$, then $\tX/T$ is homotopy equivalent to $\bigvee_{j\in\bbn}\bigvee_{\pi_1(X)/\pi_1(X_j)}\tX_j$. Since $\pi_{n}(X)\cong \pi_n(\tX)$, we have a surjective homomorphism $\Theta:\pi_n(X)\to \bigoplus_{j\in\bbn}\bigoplus_{\pi_1(X)/\pi_1(X_j)}\pi_n(X_j)$, which can be defined independent of the choice of $T$. Moreover, when each covering space $\tX_j$ is $(n-1)$-connected, $\Theta$ is an isomorphism. In the case that some $\tX_j$ are not $(n-1)$ connected, other methods for computing homotopy groups of wedges may become relevant; however, the indexing of the wedge summand fully incorporates the effect of $\pi_1$ on $\pi_n$.

There are a few places where standard methods break down for a shrinking wedge $X=\sw_{j\in\bbn}X_j$ of connected, non-simply connected CW-complexes. First and foremost, $X$ does not have a universal covering space. However, it does have a generalized universal covering space $\tX$ (and map $p:\tX\to X$) in the sense of Fischer and Zastrow \cite{FZ07}. The structure of $\tX$ is an ``infinite version" of the classical situation. In particular, $\tX$ also consists of copies of the universal covering spaces $\tX_j$ arranged in a tree-like fashion (in the sense that simple closed curves only exist in individual copies of $\tX_j$). However, these arrangements will now mimic the reduced infinite-word description of $\pi_1(X)$ \cite{Edafreesigmaproducts}. For example, an infinite product $\ell_1\ell_2\ell_3\cdots\in\pi_1(X)$ where $\ell_j\in\pi_1(X_j)$ will lift to a path in $\tX$ that proceeds (in order) through copies of $\tX_1,\tX_2,\tX_3,...$ in $\tX$. Thus, when $j\to\infty$, one should consider copies of $\tX_j$ in $\tX$ as being shrinking in size. Since infinite words in $\pi_1(X)$ may be indexed by countable, dense linear orders, there will be corresponding dense arrangements of the spaces $\tX_j$ within $\tX$ too. With this description of $\tX$, it is possible to choose a uniquely arcwise connected subspace $T\subseteq\tX$ that is analogous to a maximal tree. However, the collapsing map $\tX\to \tX/T$ will rarely be a homotopy equivalence. Moreover, care is required if one wishes to choose $T$ to be coherent with a choice of trees in the universal covers over the approximating finite wedge $\bigvee_{j=1}^{k}X_k$. Finally, while $X$ is ``wild" at only a single point, $\tX$ will be wild at uncountably many points, namely those in the wedgepoint fiber $p^{-1}(x_0)$.

To overcome the many obstacles laid out in the previous paragraph, we first attach an arc to each space $X_j$ to form $Y_j$ and take the endpoint of the added ``whisker" to be the basepoint of $Y_j$. The universal cover $\tY_j$ consists of $\tX_j$ with arcs attached to each point in the basepoint fiber of the covering map $\tX_j\to X_j$. Now the generalized universal covering $\tY$ is comprised of copies of $\tY_j$ arranged in the same way copies of $\tX_j$ are arranged in $\tX$. However, the added arcs will provide ``extra space" around which we can perform desired deformations. Most of the technical work in this paper goes toward understanding both the direct construction of $\tY$ (Definition \ref{whiskertopologydef}) and its relationship to the inverse limit of ordinary universal covering spaces. A key insight is that we must coherently choose a maximal tree in each copy of $\tX_j$ appearing within $\tY$. We also use this relationship to prove that the map collapsing each of these (uncountably many) trees to a point is a homotopy equivalence. The resulting quotient $Z$ consists of a uniquely arcwise-connected space with copies of a homotopy equivalent quotient of $\tX_j$ attached along points. This puts us precisely in a situation to apply the main result of \cite{Braznsequential}. The main result of the current paper is the following theorem.

\begin{theorem}\label{mainthm}
Let $n\geq 2$ and $X=\sw_{j\in\bbn}X_j$ be a shrinking wedge of connected CW-complexes $X_j$. Then there is a canonical homomorphism
\[\Theta:\pi_n(X)\to \prod_{j\in\bbn}\bigoplus_{\pi_1(X)/\pi_1(X_j)}\pi_n(X_j),\]
which is injective if each $X_j$ has an $(n-1)$-connected universal covering space.
\end{theorem}

The injectivity of $\Theta$ in Theorem \ref{mainthm} is the isomorphism from \cite{EK00higher} if each $X_j$ is simply connected and thus $(n-1)$-connected. In the arbitrary case, we are still able to characterize the image of $\Theta$ in terms of a natural topology on $\pi_1(X)$ (see Remark \ref{thetaimageremark}, which follows from Theorem \ref{imagetheorem}). We remark on some immediate applications and cases of interest.

\begin{example}
Consider the shrinking wedge $X=\sw_{j\in\bbn}\mathbb{RP}^n$ of copies of real projective $n$-space. The universal cover $S^n$ of $\mathbb{RP}^n$ is $(n-1)$-connected and so $\pi_n\left(X\right)$ embeds as a subgroup of \[\prod_{j\in\bbn}\bigoplus_{\pi_1(X)/\pi_1(\mathbb{RP}^n)}\pi_n(\mathbb{RP}^n)\cong \prod_{j\in\bbn}\bigoplus_{\pi_1(X)/\pi_1(\mathbb{RP}^n)}\bbz\cong \prod_{j\in\bbn}\bigoplus_{\mathfrak{c}}\bbz.\]
It is possible construct the generalized universal covering space $\tX$ similar to how one might describe the universal cover of $\bigvee_{j=1}^{k}\mathbb{RP}^n$ as a tree-like arrangement of $n$-spheres. Explicitly, we could start with the generalized universal covering space $\wt{\bbe}_1$ of the $1$-dimensional earring space $\bbe_1$, which is a topological $\bbr$-tree and acts as a generalized Caley graph \cite{FZ013caley}. Every lift of a loop parameterizing the $j$-th circle of $\bbe_1$ parameterizes an ``edge" in $\wt{\bbe}_1$. Replacing each of these edges with a copy of $S^n$ (replacing endpoints with a choice of antipodal points) and topologizing in a suitable fashion yields $\tX$ (see Figure \ref{figdyad}). This is an instructive case to consider when reading the remainder of the paper as we understand $\pi_n(\sw_{j\in\bbn}\mathbb{RP}^n)$ by using inverse limits to characterize and deform the structure of $\tX$.
\end{example}

\begin{figure}[H]
\centering \includegraphics[height=2in]{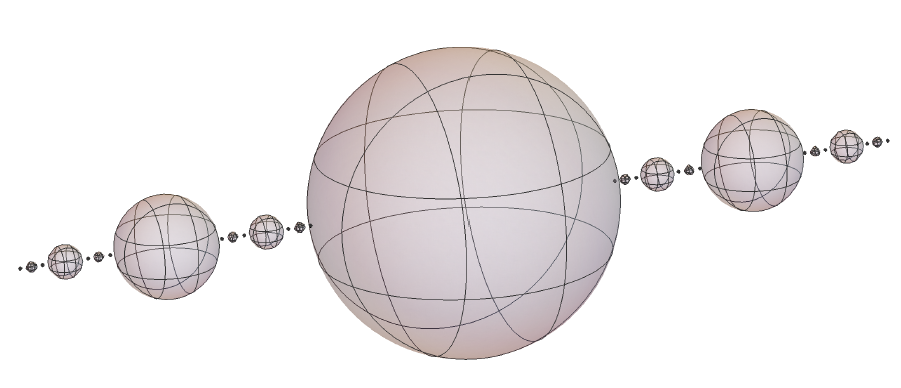}
\caption{\label{figdyad}The generalized universal cover $\tX$ of $\sw_{j\in\bbn}\mathbb{RP}^2$ seems impossible to visualize as a whole but it will contain homeomorphic copies of the space illustrated here, namely, an arc where the closure of each component of the complement of the ternary Cantor set in that arc has been replaced by a $2$-sphere and such that the diameters of the spheres approach $0$. $\tX$ will also contain arrangements of $2$-spheres indexed by every other countable linear order type. Each point in the Cantor set shown here, will be a ``branch point" of uncountable valence; every possible linear arrangement of spheres being attached at every branch point multiple times. }
\end{figure}

Theorem \ref{mainthm} also provides a characterization of $\pi_2$ for an arbitrary shrinking wedge since the universal covering spaces $\tX_j$ are always $1$-connected.

\begin{corollary}\label{secondhomotopygroupcor}
If $X=\sw_{j\in\bbn}X_j$ is a shrinking wedge of connected CW-complexes, then there is a canonical injective homomorphism
\[\Theta:\pi_2(X)\to \prod_{j\in\bbn}\bigoplus_{\pi_1(X)/\pi_1(X_j)}\pi_2(X_j).\]
\end{corollary}

There are many algebraic statements, which are immediate consequences of embedding statements like Theorem \ref{mainthm} and Corollary \ref{secondhomotopygroupcor}, e.g. $\pi_2(\sw_{j\in\bbn}X_j)$ is torsion-free if and only if $\pi_2(X_j)$ is torsion-free for all $j\in\bbn$. Recall that a path-connected space $Y$ is \textit{aspherical} if $\pi_n(Y)=0$ for all $n\geq 2$. Theorem \ref{mainthm} also implies the first part of the following theorem; the second part must be proved separately (see Section \ref{sectionaspherical}).

\begin{theorem}\label{asphericalcorollary}
If $X_j$ is an aspherical CW-complex for all $j\in\bbn$, then $\sw_{j\in\bbn}X_j$ is aspherical. Moreover, if each $X_j$ is locally finite, then the generalized universal covering space $\tX$ is contractible.
\end{theorem}

\begin{example}
Corollary \ref{secondhomotopygroupcor} implies that the shrinking wedge of tori $\sw_{j\in\bbn}\bbt$ (see Figure \ref{fig1}) is aspherical. Previously, it was only known that $\pi_2(\sw_{j\in\bbn}\bbt)=0$ \cite{EKRZSnake}. Similarly, a shrinking wedge of any sequence of orientable surfaces with positive (or infinite) genus is aspherical.
\end{example}

\begin{figure}[H]
\centering \includegraphics[height=2.8in]{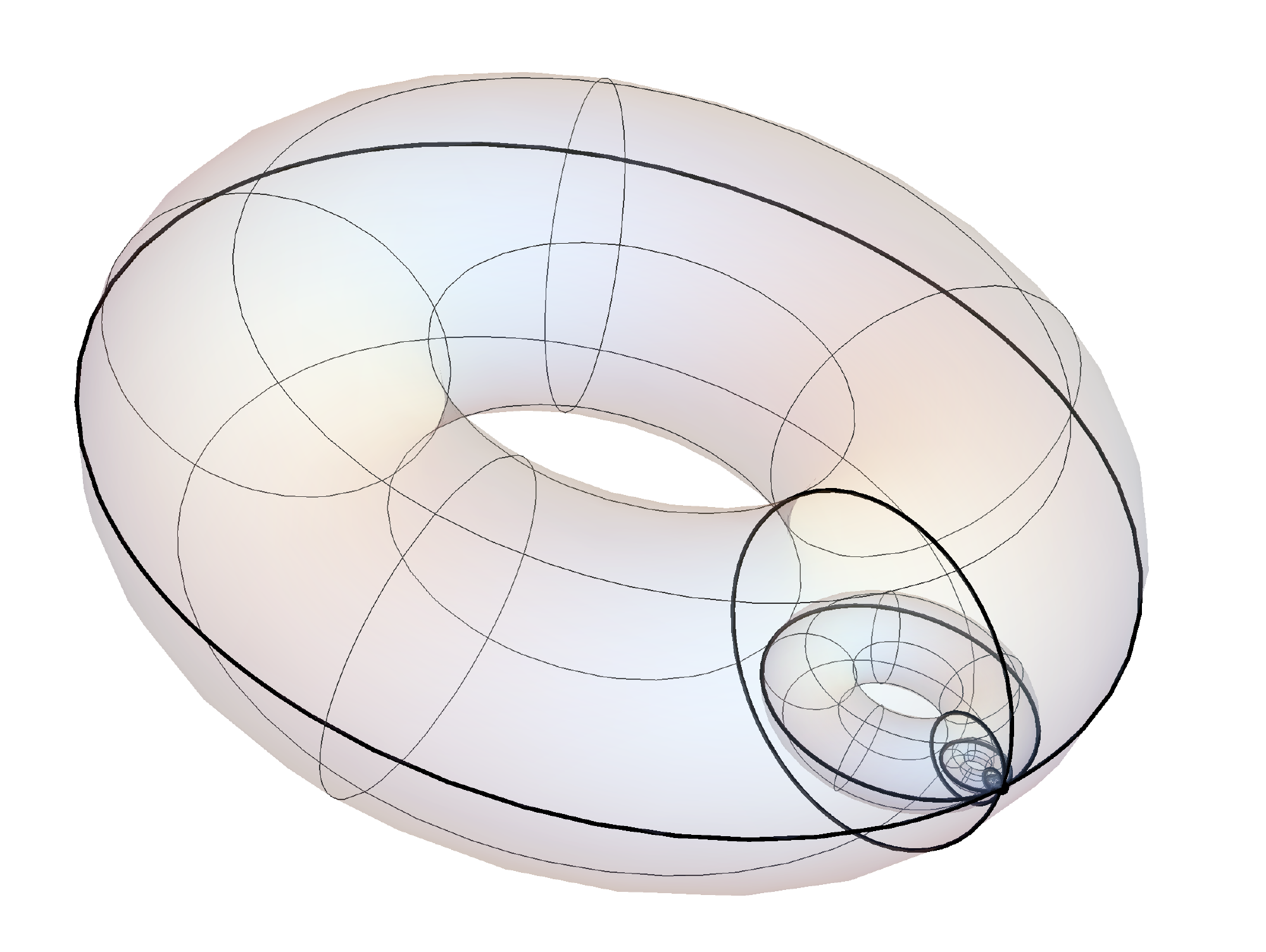}
\caption{\label{fig1}The shrinking wedge of tori is aspherical and has a contractible generalized universal covering space.}
\end{figure}

\section{Preliminaries and Notation}

All topological spaces in this paper are assumed to be Hausdorff. Throughout, $I$ denotes the unit interval $[0,1]$ and a \textit{path} is a map $\alpha:\ui \to X$. We write $\alpha\cdot\beta$ for the concatenation of paths when $\alpha(1)=\beta(0)$ and $\alpha^{-}$ for the reverse path $\alpha^{-}(t)=\alpha(1-t)$. If $[a,b]\subseteq \ui$ and $\alpha:\ui\to X$ is a path, we may simply write $[\alpha|_{[a,b]}]$ to denote the path-homotopy class $[\alpha|_{[a,b]}\circ h]$ where $h:[0,1]\to [a,b]$ is the unique increasing linear homeomorphism. 

We will generally represent elements of the $n$-th homotopy group $\pi_n(X,x)$, $n\geq 1$ by relative maps $(I^n,\partial I^n)\to (X,x)$. When the basepoint $x$ is clear from context, we will suppress it from our notation and simply write $\pi_n(X)$.

We say that a homotopy $H:X\times I\to Y$ is \textit{constant on} $A\subseteq X$ (or is \textit{relative to }$A$) if for all $\bfx\in A$, $H(\bfx,t)$ is constant as $t$ varies. If $H$ is constant on the basepoint $x_0$, then we call $H$ a based homotopy. A \textit{based homotopy equivalence} is based map $f:(X,x)\to (Y,y)$ where there is a based homotopy inverse $g:(Y,y)\to (X,x)$ and based homotopies $id_{X}\simeq g\circ f$ and $id_{Y}\simeq f\circ g$.

A \textit{Peano continuum} is a connected locally path-connected compact metric space. The Hahn-Mazurkiewicz Theorem \cite[Theorem 8.14]{Nadler} implies that a Hausdorff space is a Peano continuum if and only if there exists a continuous surjection $\ui\to X$. A Peano continuum which is uniquely arcwise connected is a \textit{dendrite}.

\subsection{Shrinking wedges and their fundamental groups}\label{sectionshrinkingwedgeintro}

Given a collection $(X_j,x_j)$, $j\in S$ of spaces, let $\bigvee_{j\in S}(X_j,x_j)$ (or $\bigvee_{j\in S}X_j$ when basepoints are clear from context) denote the usual one point union with the weak topology. We will refer to the natural basepoint $x_0$ as the \textit{wedgepoint}.

\begin{definition}
The \textit{shrinking wedge} of an infinite sequence $(X_j,x_j)$, $j\in \bbn$ of based spaces is the space $\sw_{j\in \bbn}(X_j,x_j)$ with the underlying set of $\bigvee_{j\in \bbn}X_j$ but with the following topology: $U\subseteq X$ is open if and only if $U\cap X_j$ is open in $X_j$ for all $j\in \bbn$ and if $x_0\in U$ implies $X_j\subseteq U$ for all but finitely many $j\in \bbn$.

For both standard and shrinking wedges, we will refer to each space $X_j$ as a \textit{wedge summand}.
\end{definition}

In the remainder of this section, we will assume that, for each $j\in\bbn$, the space $X_j$ is a connected CW-complex basepoint $x_j$ that serves as the basepoint of $X_j$. Let $X=\sw_{j\in\bbn}X_j$ be the shrinking wedge and for each $k\in\bbn$, let $X_{\leq k}=\bigvee_{j=1}^{k}X_j$ be finite wedge of the first $k$ spaces. Define
\begin{itemize}
\item $R_{k+1,k}:X_{\leq k+1}\to X_{\leq k}$ to be the retraction that collapses $X_{k+1}$ to $x_0$ ,
\item $R_{k}:X\to X_{\leq k}$ to be the retraction that collapses $\bigcup_{j>k}X_j$ to $x_0$.
\end{itemize}
The canonical induced map $X\to \varprojlim_{k}X_{\leq k}$, $x\mapsto (R_k(x))$ is a homeomorphism; we will sometimes identify $X$ with this inverse limit representation. 

We identify $\pi_1(X_{\leq k})$ with the free product $\ast_{j=1}^{k}\pi_1(X_j)$. If $\pi_1(X_j)=1$ for all but finitely many $j$, then we arrive at the finitely generated case $\pi_1(X)\cong\pi_1(X_{\leq k})$ for some $k$. To avoid this situation we will assume that $\pi_1(X_j)\neq 1$ for infinitely many $j$. By grouping and rearranging some of the $\pi_1(X_j)$, we may assume that $\pi_1(X_j)\neq 1$ for all $j$. In this case, $\pi_1(X)$ will be uncountable and not isomorphic to the infinite free product of the groups $\pi_1(X_j)$. We recall the two main approaches to characterizing the elements of $\pi_1(X)$: (1) the inverse limit/shape theoretic approach and (2) infinite reduced words.

The idea of the shape theoretic approach is to embed $\pi_1(X)$ into an inverse limit of the fundamental groups of the approximating projections. It is well-known that shrinking wedges of CW-complexes are $\pi_1$-shape injective in the following sense.

\begin{theorem}\cite{MM}\label{injectivetheorem}
If $X=\sw_{j\in\bbn}X_j$ is a shrinking wedge of CW-complexes, then the canonical homeomorphism $\phi_{X}:\pi_1(X)\to \varprojlim_{k}\pi_1(X_{\leq k})$, $\phi_{X}(\alpha)=((R_k)_{\#}(\alpha))$ is injective.
\end{theorem}

Thus $\alpha\in \pi_1(X)$ is non-trivial if and only if there exists $k\in\bbn$ such that $(R_k)_{\#}(\alpha)\neq 1$ in the free product $\ast_{j=1}^{k}\pi_1(X_j)$.

The second approach assigns a unique infinite word to each element of $\pi_1(X)$. A \textit{word} is a function $w$ from a countable linearly ordered set $\ov{w}$ to $\bigcup_{j\in\bbn}\pi_1(X_j)$ (assuming $\pi_1(X_j)\cap \pi_1(X_{j'})=\{1\}$ when $j\neq j '$) such that $w^{-1}(\pi_1(X_j))$ is finite for all $j\in\bbn$. If $v$ is another word and there is an order isomorphism $\kappa:\ov{w}\to \ov{v}$ such that $v\circ \kappa=w$, then we consider $w$ and $v$ isomorphic (and write $w\cong v$). The collection of all isomorphism classes words $\mathscr{W}$ is a set. 

Given a word $w$ and finite set $F\subseteq \bbn$, we define the projection word $w_F:\ov{w_F}\to \bigcup_{j\in\bbn}\pi_1(X_j)$ to be the finite word obtained by deleting all letters in $\pi_1(X_j)$, $j\notin F$. More precisely, $\ov{w_F}=\{\ell\in \ov{w}\mid w(\ell)\in \bigcup_{j\in F}\pi_1(X_j)\}$ and $\ov{w_F}(\ell)=\ov{w}(\ell)$ whenever $\ell\in\ov{w_F}$. We may regard $w_F$ as an unreduced word representing an element of the free product $\ast_{j\in F}\pi_1(X_j)$.

Given $w,v\in \mathscr{W}$, we write $w\sim v$ if for every finite subset $F\subseteq \bbn$, the reduced representatives of $w_F$ and $v_F$ in $\ast_{j\in F}\pi_1(X_j)$ are equal. Since $\sim$ is an equivalence relation on $\mathscr{W}$, we let $[w]$ denote the equivalence class of $w$. The set $\varoast_{j\in\bbn}\pi_1(X_j)=\mathscr{W}/\mathord{\sim}$ becomes a group with the operation $[w][v]=[wv]$ where $wv$ is the concatenation of the reduced words with $\ov{wv}$ defined as the linear order sum $\ov{w}+\ov{v}$. The identity $e$ or ``empty word" is the equivalence class of the identity on the 1-point ordered set $\{1\}\to \{1\}$.

A word $w\in\mathscr{W}$ is \textit{reduced} if (1) whenever $w=avb$, we have $[v]\neq e$ and (2) whenever $\ell,\ell '$ are consecutive elements in $\ov{w}$, $w(\ell)$ and $w(\ell ')$ lie in distinct groups $\pi_1(X_j)$. Intuitively, $w$ is reduced if it has no trivial subwords (including $w$ itself) and if it is not possible to combine any existing consecutive letters. It is known that for every word $w\in\mathscr{W}$, there exists a reduced word $v$, unique up to isomorphism, such that $[w]=[v]$ (see \cite[Theorem 1.4]{Edafreesigmaproducts}).

The projection maps $\varoast_{j}\pi_1(X_j)\to \pi_1(X_{\leq k})$, $[w]\mapsto [w_F]$ where $F=\{1,2,\dots,k\}$ agree with the bonding maps $(R_{k+1,k})_{\#}:\pi_1(X_{\leq k+1})\to \pi_1(X_{\leq k})$ and induce a homomorphism $\psi:\varoast_{j}\pi_1(X_j)\to \varprojlim_{k}\pi_1(X_{\leq k})$ such that $\im(\psi)=\im(\phi_{X})$. 

Given a non-constant loop $\beta:\ui \to X$ based at $x_0$, let $\ov{\beta}$ be the set of connected components of $\beta^{-1}(X\backslash\{x_0\})$ with the linear ordering inherited from $\ui$. There is a well-defined word $w_{\beta}:\ov{\beta}\to \bigcup_{j\in \bbn}\pi_1(X_j)$ given by $w_{\beta}((a,b))=[\alpha|_{[a,b]}]$. Now $\chi([\beta])=[w_{\beta}]$ defines a group isomorphism satisfying $\psi\circ \chi=\phi_{X}$.
\[\xymatrix{
G \ar[dr]_-{\chi} \ar@{^{(}->}[r]^-{\phi_{X}} & \varprojlim_{k}\pi_1(X_{\leq k})\\
& \varoast_{j}\pi_1(X_j) \ar@{^{(}->}[u]_-{\psi}
}\]

\begin{definition}
We say that a loop $\beta:\ui \to X$ based at $x_0$ is \textit{reduced} if $\beta$ is constant or if $w_{\beta}$ is a reduced word in $\mathscr{W}$.
\end{definition}

Considering the above diagram, it follows that every loop $\alpha:\ui\to X$ based at $x_0$ is path-homotopic to a reduced loop $\beta$. Moreover, reduced loop representatives of homotopy classes are unique in the following sense: if $\beta$ and $\gamma$ are path-homotopic reduced loops, then there is an order-isomorphism $\kappa:\ov{\beta}\to \ov{\gamma}$, such that if $(a,b)\in \ov{\beta}$ and $\kappa((a,b))=(c,d)\in\ov{\gamma}$, then $\beta|_{[a,b]}\simeq \gamma|_{[c,d]}$ as loops in one of the spaces $X_j$. Therefore, if $\alpha\in\pi_1(X)$, we may also use the symbol $\alpha$ to denote a choice of reduced loop in $\alpha$.

\subsection{The locally path connected coreflection}

Because inverse limits of locally path connected spaces are not always locally path connected, we require the following construction.

\begin{definition}
The \textit{locally path-connected coreflection} of a space $X$ is the space $\lpc(X)$ with the same underlying set as $X$ but with topology generated by the basis consisting of all path components of the open sets in $X$.
\end{definition}

The topology of $\lpc(X)$ is finer than that of $X$ thus the identity function $id:\lpc(X)\to X$ is continuous. It is well-known that $\lpc(X)$ is locally path connected and that $\lpc(X)=X$ if and only if $X$ is already locally path connected. The construction of $\lpc(X)$ defines a functor $\lpc:\mathbf{Top}\to\mathbf{Lpc}$ from the category of topological spaces to the full subcategory of locally path connected spaces. This functor is a coreflection in the sense that $\lpc$ is right adjoint to the inclusion functor $\mathbf{Lpc}\to\mathbf{Top}$. In other words, if $Z$ is locally path connected, then a function $f:Z\to X$ is continuous if and only if $f:Z\to \lpc(X)$ is continuous. In particular, $X$ and $\lpc(X)$ share the same set of continuous functions from $I^n$. It follows that $id:\lpc(X)\to X$ is a bijective weak homotopy equivalence.

Since the direct product of locally path-connected spaces is locally path connected, $\lpc(\prod_{j}X_j)\cong \prod_{j}\lpc(X_j)$ in $\mathbf{Top}$. In particular, $\lpc(X\times \ui)=\lpc(X)\times \ui$ allows one to prove the following proposition.

\begin{proposition}\label{lpcheprop}
If $f:X\to Y$ and $g:Y\to X$ are (based or unbased) homotopy inverses, then so are $\lpc(f):\lpc(X)\to \lpc(Y)$ and $\lpc(g):\lpc(Y)\to \lpc(X)$.
\end{proposition}

One should be wary of limits of inverse systems in $\mathbf{Lpc}$ because inverse limits of locally path connected spaces in $\mathbf{Lpc}$ and $\mathbf{Top}$ do not always agree. If $\varprojlim_{j}X_j$ is an inverse limit in $\mathbf{Top}$ of locally path connected spaces (viewed as a subspace of $\prod_{j}X_j$), then $\lpc(\varprojlim_{j}X_j)$ is the space that gives the limit of the same inverse system in $\mathbf{Lpc}$.

\subsection{Generalized universal covering maps}

When each $X_j$ is a connected, non-simply connected CW-complex, $\sw_{j}X_j$ will not have a universal covering space. However $\sw_{j}X_j$ always admits a generalized universal covering space in the sense of Fischer-Zastrow \cite{FZ07}. The idea behind this notion of  ``generalized (universal) covering map" is to use the lifting properties of covering maps as the definition and work internal to the category of path-connected, locally path-connected spaces.

\begin{definition}\label{gencovdef}
A map $q:E\to X$ is a \textit{generalized covering map} if $E$ is non-empty, path connected, and locally path connected and if for any map $f:(Y,y)\to (X,x)$ from a path-connected, locally path-connected space $Y$ and point $e\in q^{-1}(x)$ such that $f_{\#}(\pi_1(Y,y))\leq q_{\#}(\pi_1(E,e))$, there is a unique map $\wt{f}:(Y,y)\to (E,e)$ such that $q\circ \wt{f}=f$. Moreover, if $E$ is simply connected, we call $q$ a \textit{generalized universal covering map} and $E$ a \textit{generalized universal covering space}.
\end{definition}

Unlike ordinary covering maps, based generalized covering maps are closed under composition and form a complete category \cite{Brazcat}. The following proposition follows immediately from the definition and standard covering space theory arguments.

\begin{proposition}\label{coveringisomorphism}
If $q:(E,e_0)\to (X,x_0)$ satisfies all properties of being a generalized covering map except for the assumption that $E$ is locally path connected, then the induced homomorphism $q_{\#}:\pi_n(E,e_0)\to \pi_n(X,x_0)$ is an injection for $n=1$ and an isomorphism for all $n\geq 2$.
\end{proposition}

Every (universal) covering map (in the usual sense) $p:E\to X$ where $E$ is path connected and $X$ is locally path connected is a generalized (universal) covering map. If $p:E\to X$ is a generalized universal covering map, then $p$ is an ordinary covering map if and only if $X$ is semilocally simply connected. We recall the following standard construction from covering space theory \cite{Spanier66}.

\begin{definition}\label{whiskertopologydef}
Whenever $X$ is a space with given basepoint $x_0\in X$, let $\wt{X}$ be the space of path-homotopy classes $[\alpha]$ of paths $\alpha:(\ui,0)\to (X,x_0)$. An open neighborhood of $[\alpha]$ is a set of the form $N([\alpha],U)=\{[\alpha\cdot\epsilon]\mid \epsilon(\ui)\subseteq U\}$ where $U$ is an open neighborhood of $\alpha(1)$ in $X$. This topology is the so-called \textit{whisker topology} on $\wt{X}$. The homotopy class of the constant path at $x_0$, which we denote as $\wt{x}_0$, is the basepoint of $\wt{X}$.
\end{definition}

The endpoint projection map $p:\tX\to X$, $p([\alpha])=\alpha(1)$ is a continuous surjection, which is open if and only if $X$ is locally path connected and provides a candidate for a generalized universal covering map. 

\begin{remark}[Standard Lifts of Paths]\label{standardliftremark}
Every path in $X$ lifts uniquely to $\wt{X}$ relative to a chosen starting point. Suppose $[\beta]\in\wt{X}$ and $\alpha:(I,0)\to (X,\beta(1))$ is a path. Define paths $\alpha_s:\ui\to X$, $s\in\ui$ by $\alpha_s(t)=\alpha(st)$. The function $\wt{\alpha}:(I,0)\to (\wt{X},[\beta])$, $\wt{\alpha}(s)=[\beta\cdot\alpha_s]$ defines a continuous lift of $\alpha$ starting at $[\beta]$ (c.f. \cite[Lemma 2.4]{FZ07}), which we refer to as a \textit{standard lift} of $\alpha$.
\end{remark}

Remark \ref{standardliftremark} ensures that $p:\tX\to X$ always has path-lifting. According to \cite[Prop. 2.14]{FZ07}, $p:\tX\to X$ is a generalized universal covering map if and only if $p$ has the unique path-lifting property, that is, if the lift described in Remark \ref{standardliftremark} is the \textit{only} lift of $\alpha$ starting at $[\beta]$. In general, this does not have to happen \cite[Example 2.7]{FZ07}. However, many sufficient conditions are known.

\begin{theorem}\cite{FZ07}
If $X$ is metrizable and path-connected, $x_0\in X$, and the canonical homomorphism $\phi:\pi_1(X,x_0)\to \check{\pi}_1(X,x_0)$ to the first shape homotopy group is injective, then $p:\tX\to X$ is a generalized universal covering map.
\end{theorem}

In the case of a shrinking wedge $X=\sw_{j\in\bbn}X_j$ of CW-complexes $X_j$, the homomorphism $\phi:\pi_1(X,x_0)\to \check{\pi}_1(X,x_0)$ is precisely that from Theorem \ref{injectivetheorem}.

\begin{corollary}
Every shrinking wedge of CW-complexes admits a generalized universal covering space.
\end{corollary}

The next theorem guarantees ensures that whenever a generalized universal covering map exists, it may be constructed as in Definition \ref{whiskertopologydef}.

\begin{theorem}\cite[Section 5]{Brazcat}\label{structurethm}
If there exists a generalized universal covering map $q:(E,e_0)\to (X,x_0)$, then there exists a homeomorphism $h:(\tX,\tx_0)\to (E,e_0)$ such that $q\circ h=p$.
\end{theorem}

We will also have need of the following separation axiom.

\begin{lemma}\cite[Lemmas 2.10 and 2.11]{FZ07}
If $p:\tX\to X$ is a generalized universal covering map where $X$ is Hausdorff, then $\tX$ is Hausdorff.
\end{lemma}

\begin{definition}[A topology on the fundamental group]
When $p:\tX\to X$ is a generalized universal covering map with respect to a basepoint $x_0\in X$, the fiber $p^{-1}(x_0)$ is precisely the fundamental group $\pi_1(X,x_0)$. In particular, $\pi_1(X,x_0)$ naturally inherits a topology as a subspace of $\tX$, which we also refer to as the \textit{whisker topology}. Since $\tX$ is Hausdorff by the previous lemma, $\pi_1(X,x_0)$ is Hausdorff with this topology.
\end{definition}

If one has a map $q:E\to X$, which has all of the properties of a generalized universal covering map except for $E$ being locally path connected, the locally path connected coreflection provides a ``quick fix." Indeed, for any path connected space $X$, the identity function $id:\lpc(X)\to X$ is a generalized covering map.

\begin{proposition}\label{lpcprop}
If $q:E\to X$ has all of the properties of a generalized universal covering map except for the assumption that $E$ is locally path connected, then $q:\lpc(E)\to X$ is a generalized universal covering map.
\end{proposition}

Based generalized covering maps are closed under pullback using $\lpc$: If $q:(\tX,\tx_0)\to (X,x)$ is a based generalized covering map and $f:(Y,y_0)\to (X,x_0)$ is a based map, then there is a pullback generalized covering map (of $p$ over $f$) $p:(\wt{Y},\ty_0)\to (Y,y_0)$ and a map $\wt{f}:(\wt{Y},\ty_0)\to (\wt{X},\tx_0)$ such that $q\circ \wt{f}=f\circ p$. In particular, we let $C$ be the path component of $(\tx_0,y_0)$ in the ordinary topological pullback $\wt{X}\times_{X}Y=\{([\alpha],y)\in\tX\times Y\mid f(y)=\alpha(1)\}$, set $\wt{Y}=\lpc(C)$ and let $p$ and $\wt{f}:\wt{Y}\to \tX$ be the restrictions of the projection maps.

\begin{proposition}\label{pullbackprop}
If $q:(\tX,\tx_0)\to (X,x)$ is a based generalized universal covering map and $f:(Y,y_0)\to (X,x_0)$ induces an injection on fundamental groups, then the pullback $p:\wt{Y}\to Y$ of $q$ over $f$ is also a generalized universal covering.
\end{proposition}

The next corollary follows from straightforward lifting arguments so we omit the proof. 

\begin{corollary}\label{gencovhomequivcorollary}
If $f:(Y,y_0)\to (X,x_0)$ is a based homotopy equivalence and $q:(\tX,\tx_0)\to (X,x_0)$ is a generalized universal covering map, then there exists a generalized universal covering map $p:(\tY,\ty_0)\to (Y,y_0)$ and a based homotopy equivalence $\wt{f}:(\tY,\ty_0)\to (\tX,\tx_0)$ such that $q\circ \wt{f}=f\circ p$.
\end{corollary}


Because Theorem \ref{injectivetheorem} holds for shrinking wedges, there is another way to construct their generalized universal covering maps, which we detail in the next remark.

\begin{remark}[Inverse Limits of Coverings]\label{ilimremark}
Let $(X,x_0)=\sw_{j\in\bbn}(X_j,x_j)$ be a shrinking wedge of connected CW-complexes with inverse limit presentation $\varprojlim_{k}(X_{\leq k},R_{k+1,k})$. If $q_{\leq k}:(\tX_{\leq k},\tx_0)\to (X_{\leq k},x_0)$, $k\in\bbn$ are the universal covering maps, we have the following situation.
\[\xymatrix{
&  \cdots  & (\wt{X}_{\leq 3},\tx_0) \ar[d]_-{q_{\leq 3}}  & (\wt{X}_{\leq 2},\wt{x}_0) \ar[d]_-{q_{\leq 2}}  & (\wt{X}_{\leq 1},\wt{x}_0) \ar[d]_-{q_{\leq 1}} \\
X & \cdots \ar[r] & (X_{\leq 3},x_0) \ar[r]_-{R_{3,2}} & (X_{\leq 2},x_0) \ar[r]_-{R_{2,1}} & (X_{\leq 1},x_0)
}\]
Since $\tX_{\leq k}$ is locally path connected and simply connected, the lifting property of the maps $q_{\leq k}$ ensures that we have maps $\wt{R}_{k+1,k}:(\tX_{\leq k+1},\tx_{0})\to (\tX_{\leq k},\tx_0)$ making the diagram below commute. The result is an inverse sequence of based generalized universal covering maps. Let $ \wh{X}=\varprojlim_{k}(\tX_{\leq k},\wt{R}_{k+1,k})$ and $\wh{q}=\ilim_{k}q_{\leq k}$ denote the respective inverse limits. Let $\wt{R}_k:\wh{X}\to \tX_{\leq k}$ be the projection map for $k\in\bbn$.
\[\xymatrix{
 \wh{X}=\varprojlim_{k}(\tX_{\leq k},\wt{x}_0) \ar[d]_-{\wh{q}} & \cdots \ar[r] & (\wt{X}_{\leq 3},\wt{x}_0) \ar[d]_-{q_{\leq 3}} \ar[r]^-{\wt{R}_{3,2}} & (\wt{X}_{\leq 2},\wt{x}_0) \ar[d]_-{q_{\leq 2}} \ar[r]^-{\wt{R}_{2,1}} & (\wt{X}_{\leq 1},\wt{x}_0) \ar[d]_-{q_{\leq 1}} \\
  X &\cdots \ar[r] & (X_{\leq 3},x_0) \ar[r]_-{R_{3,2}} & (X_{\leq 2},x_0) \ar[r]_-{R_{2,1}} & (X_{\leq 1},x_0)
}\]
Since $\wh{X}$ need not be path-connected, we let $\wh{X}_0$ be the path component of $\hx_0=(\tx_0,\tx_0,\tx_0,\dots)\in\wh{X}$. Taking $\wh{q}_0:(\wh{X}_0,\wh{x}_0)\to (X,x_0)$ to be the restriction of $\wh{q}$, a direct argument shows that $\wh{q}_0$ has all the properties of a generalized covering map except that $\wh{X}_0$ need not be locally path connected. In particular, $\wh{q}_{0\#}:\pi_1(\wh{X}_0,\wh{x}_0)\to \pi_1(X,x_0)$ is injective. By Proposition \ref{lpcprop} and Theorem \ref{structurethm}, $\wh{q}_0:\lpc(\wh{X}_0)\to X$ is a generalized covering map. 

Moreover, $\wh{X}_0$ is simply connected: Given a loop $\wt{\alpha}:(\ui,\{0,1\})\to (\hX_0,\hx_0)$, $\wt{R}_k\circ \wt{\alpha}$ is null-homotopic since $\wt{X}_{\leq k}$ is simply connected. Therefore, $R_{k}\circ (\wh{q}\circ \wt{\alpha})=q_{\leq k}\circ \wt{R}_k\circ \wt{\alpha}$ is null-homotopic for all $k$. Since $\phi_{X}:\pi_1(X,x_0)\to \varprojlim_{k}\pi_1(X_{\leq k},x_0)$ is injective (Theorem \ref{injectivetheorem}) and $\phi_{X}([\wh{q}\circ \wt{\alpha}])=1$, we have $[\wh{q}_0\circ \wt{\alpha}]=1$. Since $\wh{q}_0$ is $\pi_1$-injective (Proposition \ref{coveringisomorphism}), $\wt{\alpha}$ is null-homotopic.

Since $\wh{X}_0$ is simply connected and $id:\lpc(\wh{X}_0)\to \wh{X}_0$ is a weak homotopy equivalence, $\lpc(\wh{X}_0)$ is simply connected. According to Proposition \ref{lpcprop}, $\wh{q}_0:\lpc(\wh{X}_0)\to X$ satisfies all criteria to be a generalized \textit{universal} covering map. Theorem \ref{structurethm} now implies that there is a unique homeomorphism $\phi_{X}:(\tX,\tx_0)\to (\lpc(\wh{X}_0),\wh{x}_0)$ such that $\wh{q}_0\circ \phi=q$. We use ``$\phi_{X}$" to denote this map because its restriction to the basepoint-fiber: $p^{-1}(x_0)\to \varprojlim_{k}q_{k}^{-1}(x_0)$ is precisely the homomorphism $\phi_{X}$ from Theorem \ref{injectivetheorem}.
\[\xymatrix{
\tX \ar[r]^-{\phi_{X}} \ar[dr]_-{q} & \lpc(\wh{X}_0) \ar[d]^-{\wh{q}_0} \ar[r]^-{id}  & \wh{X}_0 \ar[dl]^-{\wh{q}_0}\\
& X
}\]
\end{remark}

\begin{corollary}\label{continuitycor1}
Let $X=\sw_{j\in\bbn}X_j$ be a shrinking wedge and $\phi_{X}:\tX\to \wh{X}_0$ and $\wt{R}_k:\wh{X}\to \wt{X}_{\leq k}$ be defined as in Remark \ref{ilimremark}. If $W$ is locally path connected and $f:W\to \tX$ is a function, then the following are equivalent:
\begin{enumerate}
\item $f:W\to \tX$ is continuous,
\item $\phi_{X}\circ f:W\to \wh{X}_0$ is continuous,
\item $\wt{R}_k\circ\phi_{X}\circ f:W\to \tX_{\leq k}$ is continuous for all $k\in\bbn$.
\end{enumerate}
\end{corollary}

\section{Attaching whiskers and choosing maximal trees}

At this point, we being to fix spaces and establish notation that will be used throughout the remainder of the paper. We assume that, for each $j\in\bbn$, the space $X_j$ is a fixed connected, non-simply connected, CW-complex with a single $0$-cell $x_j$ that serves as the basepoint of $X_j$. Let $q_j:\tX_j\to X_j$ be the universal cover of the individual wedge summands. We use the notation consistent with that in Section \ref{sectionshrinkingwedgeintro}, namely, $X=\sw_{j\in\bbn}X_j$ is the shrinking wedge with wedgepoint $x_0$ and for each $k\in\bbn$, let $X_{\leq k}=\bigvee_{j=1}^{k}X_j\subseteq X$. Additionally, $R_{k+1,k}:X_{\leq k+1}\to X_{\leq k}$ and $R_{k}:X\to X_{\leq k}$ are the canonical retractions and $q_{\leq k}:\wt{X}_{\leq k}\to X_{\leq k}$ is the universal covering map of the finite wedge. We apply the construction in Remark \ref{ilimremark} to the retractions $R_{k+1,k}$ and covering maps $q_{\leq k}$ and fix the notation used there. Since the canonical homomorphism $\phi_{X}:\pi_1(X)\to \varprojlim_{k}\pi_1(\xleqk)$ is injective (Theorem \ref{injectivetheorem}) $X$ admits a generalized universal covering map $q:\tX\to X$ and there is a homeomorphism $\phi_{X}:\tX\to\lpc(\wh{X}_0)$ such that $\wh{q}_0\circ \phi_{X}=q$.

\subsection{Attaching whiskers: replacing $X$ with $Y$} \label{sectionwellpointed}

Let $Y_j=X_j\times\{0\}\cup \{x_j\}\times I$ be the subspace of $X_{j}\times \ui$ with basepoint $y_j=(x_j,1)$. By identifying $X_j$ with $X_j\times\{0\}$, we may treat $Y_j$ as the CW-complex consisting of $X_j$ and a ``whisker" $\bfe_j=\{x_j\}\times I$ attached at $x_j$. Let $p_j:\tY_j\to Y_j$ be the universal covering map. The homotopy extension property of the pair $(X_j,x_j)$ allows us to choose a retraction $\mu_j:X_j\times I\to Y_j$ so that $\mu_{j,1}(x)=\mu_j(x,1)$ is a based homotopy inverse of the quotient map $\zeta_j:Y_j\to X_j$ that collapses the whisker $\{x_j\}\times I$ to $x_j$ (see Figure \ref{fig2}). In particular, $\zeta_j\circ \mu_j$ is a based homotopy from $id_{X_j}$ to $\zeta_j\circ \mu_{j,1}$. A based homotopy from $id_{Y_j}$ to $\mu_{j,1}\circ \zeta_j$ is illustrated in Figure \ref{fig4}.

\begin{figure}[H]
\centering \includegraphics[height=1in]{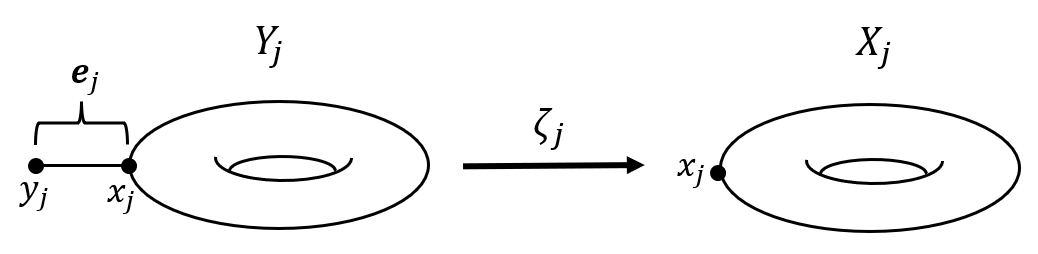}
\caption{\label{fig2} The based homotopy equivalence $\zeta_j:Y_j\to X_j$, which collapses the arc $\mathbf{e}_j$.}
\end{figure}

\begin{figure}[H]
\centering \includegraphics[height=.7in]{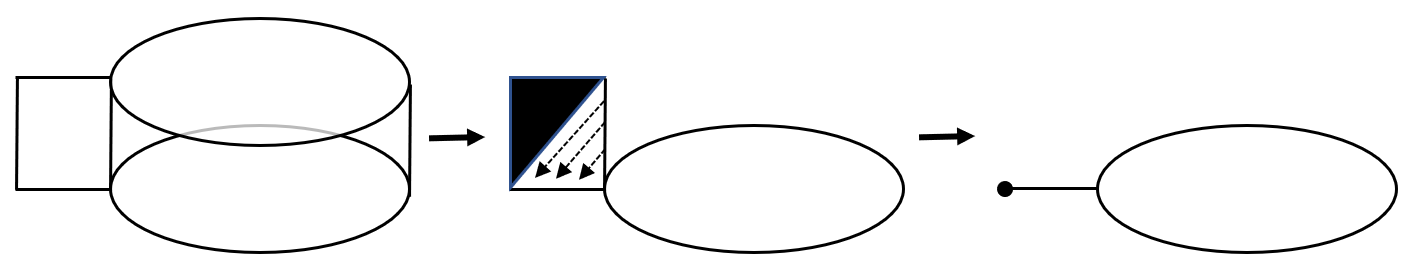}
\caption{\label{fig4} The based homotopy $Y_j\times \ui\to Y_j$ from $id_{Y_j}$ to $\mu_{j,1}\circ \zeta_j$ is a composition which first applies $\mu_{j}$ to the subspace $X_j\times \ui$, which is illustrated as a cylinder. The square $\mathbf{e}_j\times\ui$ is mapped to $\mathbf{e}_j$ so that the upper left triangle maps to $y_j$ and the lower right triangle is projected linearly.}
\end{figure}

Let $Y=\sw_{j\in\bbn}(Y_j,y_j)$ be the shrinking wedge with wedgepoint $y_0$ and for each $k\in\bbn$, let $Y_{\leq k}=\bigvee_{j=1}^{k}Y_j$ be the finite wedge viewed as retractions of $Y$. The respective canonical retractions will be denoted by $r_{k+1,k}:Y_{\leq k+1}\to Y_{\leq k}$ and $r_{k}:Y\to Y_{\leq k}$ and the universal covering maps by $p_j:\tY_j\to Y_j$ and $p_{\leq k}:\tY_{\leq k}\to Y_{\leq k}$.

\begin{lemma}\label{zetalemma}
The quotient map $\zeta:Y\to X$ that collapses $\bigcup_{j}\bfe_j$ to $y_0$ is a based-homotopy equivalence 
\end{lemma}

\begin{proof}
Let $\mu:X\to Y$ be the map whose restriction to $X_j$ is $\mu_{j,1}$. Let $K_j=\zeta_j\circ \mu_j:X_j\times \ui\to X_j$ and define $K:X\times \ui\to X$ so the restriction to $X_j\times \ui$ is $K_j$. Since each $K_j$ is the constant homotopy at the basepoint, $K$ is well-defined. Since the projection $R_k\circ K=\bigvee_{j=1}^{k}K_j:X\times \ui\to X_{\leq k}$ is continuous for every $k\in\bbn$, $K$ is continuous. By construction, $K$ is a homotopy from $id_{X}$ to $\zeta\circ \mu$. Next, let $L_j:Y_j\times \ui\to Y_j$ be the based homotopy from $id_{Y_j}$ to $\mu_{j,1}\circ \zeta_j$ illustrated in Figure \ref{fig4}. The analogous construction shows that one can construct a homotopy $L$ from $id_{Y_j}$ to $\mu_{j,1}\circ\zeta$ using the maps $L_j$.
\end{proof}

\begin{figure}[H]
\centering \includegraphics[height=1in]{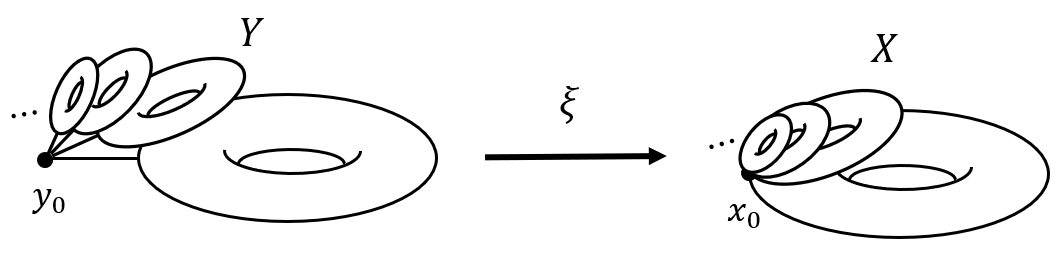}
\caption{\label{fig3} The based homotopy equivalence $\zeta:Y\to X$, which collapses the attached arcs.}
\end{figure}

\begin{remark}
The construction of the homotopy equivalence in the proof of Lemma \ref{zetalemma} shows that a sequence of based homotopy equivalences $(A_j,a_j)\simeq (B_j,b_j)$, $j\in\bbn$ induced a based homotopy equivalence $\sw_{j\in\bbn}(A_j,a_j)\simeq \sw_{j\in\bbn}(B_j,b_j)$ of the shrinking wedges. Thus, if desired one may replace each $X_j$ with any representative of its homotopy type.
\end{remark}

The space $Y$ is a shrinking wedge of CW-complexes and so the content of Remark \ref{ilimremark} applies. In particular, there are lifted maps $\wt{r}_{k+1,k}:\wt{Y}_{\leq k+1}\to \wt{Y}_{\leq k}$ satisfying $\wt{r}_{k+1,k}\circ p_{\leq k+1}=p_{\leq k}\circ \wt{r}_{k+1,k}$.
\[\xymatrix{
\tY_{\leq k+1} \ar@{-->}[r]^-{\wt{r}_{k+1,k}} \ar[d]_-{p_{k+1}} & \tY_{\leq k} \ar[d]^-{p_k}\\
Y_{\leq k+1} \ar[r]_-{r_{k+1,k}} & Y_{\leq k}
}\]
We have $\wh{Y}=\varprojlim_{k}(\wt{Y}_{\leq k},\wt{r}_{k+1,k})$ with projection maps $\wt{r}_k:\wh{Y}\to \wt{Y}_{\leq k}$ and $\wh{Y}_0$ is the path component of $\wh{y}_0=(\wt{y}_0)$ in $\wh{Y}$. Set $\wh{p}=\varprojlim_{k}p_{\leq k}$ and let $\wh{p}_0:\wh{Y}_0\to Y$ be the restriction of $\wh{p}$ to $\wh{Y}_0$. The canonical homomorphism $\phi_{Y}:\pi_1(Y)\to \varprojlim_{k}\pi_1(Y_{\leq k})$, $\phi(\alpha)=((r_{k})_{\#}(\alpha))$ is injective.

There also exists a generalized universal covering map $p:\tY\to Y$ where $\wt{Y}$ has the standard construction (Definition \ref{whiskertopologydef}). The lifting property of $p_{\leq k}$ gives an induced map $\varrho_k:\tY\to \tY_{\leq k}$, $\varrho_k([\ell])=[r_k\circ \ell]$. The canonical map $\phi_{Y}:\tY\to \wh{Y}_0$, $\phi_{Y}(\alpha)=(\varrho_k(\alpha))$ is a continuous bijection, which satisfies $\wt{r}_k\circ \phi_{Y}=\varrho_k$. The coreflection $\phi_{Y}:\tY\to \lpc(\wh{Y}_0)$ is a homeomorphism. 
\[\xymatrix{
\tY \ar@/^2pc/[rr]^-{\varrho_k} \ar[r]^{\phi_{Y}} & \wh{Y}_0 \ar[r]^-{\wt{r}_k} & \tY_{\leq k}
}\]
By Corollary \ref{gencovhomequivcorollary}, the based homotopy equivalence $\zeta:Y\to X$ lifts to a based homotopy equivalence $\wt{\zeta}:\wt{Y}\to \wt{X}$ satisfying $q\circ \wt{\zeta}=\zeta\circ p$. Similarly, the homotopy equivalence $\zeta_{\leq k}=\bigvee_{j=1}^{k}\zeta_j:Y_{\leq k}\to X_{\leq k}$ lifts to based homotopy equivalence $\wt{\zeta}_{\leq k}:\wt{Y}_{\leq k}\to \wt{X}_{\leq k}$ such that $q_{\leq k}\circ \wt{\zeta}_{\leq k}=p_{\leq k}\circ \zeta_{\leq k}$. Since, $\wt{R}_{k+1,k}\circ \wt{\zeta}_{\leq k+1}=\wt{\zeta}_{\leq k}\circ \wt{r}_{k+1,k}$, we may take the limit $\wh{\zeta}=\varprojlim_{k}\wt{\zeta}_{\leq k}:\wh{Y}\to \wh{X}$. By lifting the homotopy inverse for $\zeta_{\leq k}$ and the relevant homotopies, it is straightforward from taking inverse limits that $\wh{\zeta}$ is a homotopy equivalence (even though the domain and codomain are neither path connected nor locally path connected). Moreover, the restriction $\wh{\zeta}_0:\wh{Y}_0\to \wh{X}_0$ to the path component of the basepoint is also a homotopy equivalence. In summary, we have the following commutative diagram where the vertical maps are homotopy equivalences.
\[\xymatrix{
Y \ar[d]_-{\zeta}  & \wt{Y} \ar[d]_-{\wt{\zeta}} \ar[l]_-{p}  \ar[r]^-{\phi_Y} & \wh{Y}_0 \ar[d]^-{\wh{\zeta}_0}\\
X   & \wt{X} \ar[l]^-{q}  \ar[r]_-{\phi_X} & \wh{X}_0 
}\]

The above shows that we may replace $X$ with $Y$, $\tX$ with $\tY$, and $\hX_0$ with $\wh{Y}_0$ without any loss of homotopical or shape-theoretic information. The arcs in $\tY$ will provide ``extra space" for performing suitable deformations of $\wt{Y}$ that are not possible in $\tX$. 

\begin{remark}\label{metrizableremark}[Metrizability]
In general, CW-complexes are not metrizable. Consequently, $\tX$ and $\tY$ will not always be metrizable. However, when each $X_j$ is locally finite, each $X_j$ is metrizable. Consequently, $Y_j$ and the universal covers $\tX_j$ and $\tY_j$ are metrizable. Since limits of inverse sequences of metrizable spaces are metrizable, $\hX$, $\wh{Y}$ and the subspaces $\hX_0$, $\wh{Y}_0$ will be metrizable. Finally, it is known that $\lpc$ preserves metrizability. Therefore, $\tX$ and $\tY$ are metrizable whenever each $X_j$ is locally finite.

Even if some $X_j$ are not metrizable, $\tX$ and $\tY$ are still highly structured. Indeed, every compact subset of a CW-complex is metrizable. Combining this with the arguments used in the previous paragraph, it follows that all compact subspaces of $\tX$, $\tY$, $\wh{X}$, and $\wh{Y}$ are metrizable.
\end{remark}

\subsection{Collapsing maximal trees $T_j\subseteq \tX_j$}\label{sectionfixinghe}

Since $\tX_j$ is a CW-complex, we may fix a maximal tree $T_j$ in the $1$-skeleton of $\tX_j$. Since $(\tX_j,T_j)$ has the homotopy extension property, the map $\tX_j\to C_j=\tX_j/T_j$ that collapses $T_j$ to a point is a homotopy equivalence. Let $c_j$ be the image of $T_j$ in $C_j$. Now, $C_j$ is a simply connected CW-complex with a single 0-cell $c_j$ and $\pi_n(C_j)\cong \pi_n(X_j)$ for $n\geq 2$.

The inclusion $X_j\to Y_j$ induces an embedding $\tX_j\to \tY_j$ in the following way. Recall that $\tY_j$ is defined to be the space of path-homotopy classes of paths in $Y_j$ starting at $y_j$. Let $\tau_j:\ui\to Y_j$, $\tau_j(t)=(x_j,1-t)$ be the path from $y_j$ to $x_j$ that parameterizes the arc $\bf{e}_j$ and define $\tau_{j,s}:\ui\to Y_j$ by $\tau_{j,s}(t)=\tau_j(st)$ for each $s\in\ui$. We will also use the symbols $\tau_j$ and $\tau_{j,s}$ to denote the path-homotopy classes so that $\tau_j=[\tau_j]$ and $\tau_{j,s}=[\tau_{j,s}]$.
\begin{itemize}
\item We identify $\tX_j$ with the subspace $\{\tau_j\delta\in\tY_j\mid \delta\in\tX_j\}$ of $\tY_j$ by the closed embedding $\tX_j\to \tY_j$, $\delta\mapsto \tau_j\delta$. Under this identification, $\wt{x}_j=\tau_j$ is the basepoint of $\tX_j$. 
\item for each $\beta\in \pi_1(Y_j)$, $\bfe_{j,\beta}=\{\beta \tau_{j,s}\in\tY_j\mid s\in\ui\}$ is an arc.
\item for $\beta\in\pi_1(Y_j)$, $\wt{x}_{j,\beta}=\beta\tau_{j}$ is the point where the arc $\bfe_{j,\beta}$ meets $\tX_j$ and $\wt{y}_{j,\beta}=\beta$ is the free endpoint of $\bfe_{j,\beta}$. Note that $\wt{y}_j=\wt{y}_{j,1}$ is the basepoint of $\tY_j$. We refer to the points $\wt{y}_{j,\beta}$ as the \textit{arc-endpoints} of $\wt{Y}_j$.
\item $\bft_{j}=T_j\cup \bigcup_{\beta\in\pi_1(Y_j)}\bfe_{j,\beta}$ is a maximal tree in $\tY_j$.
\end{itemize}
In summary, the space $\tY_j$ consists of the subspace $\tX_j$ with an arc $\bfe_{j,\beta}$ attached at $\beta\tau_j\in p_{j}^{-1}(x_j)$ for each $\beta\in\pi_1(Y_j)$.

\begin{figure}[H]
\centering \includegraphics[height=1.9in]{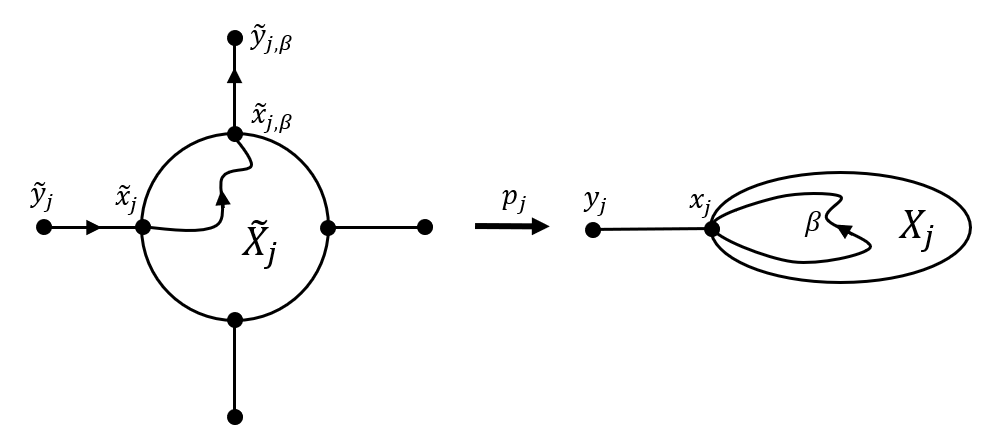}
\caption{\label{fig5} The structure of the universal covering map $p_j:\tY_j\to Y_j$ where the subspace $\tX_j$ is illustrated as a disk. The arc-endpoints form the fiber $p_{j}^{-1}(y_j)$ and the attachment points form the fiber $p_{j}^{-1}(x_j)$. The path illustrated in $\tY_j$ is the lift of a given $\beta\in\pi_1(Y_j)$, which can be factored as $\tau_j\delta\tau_{j}^{-1}$ for $\delta\in\pi_1(X_j)$.
}
\end{figure}

We may identify the quotient space $D_j=\tY_j/T_j$ with the one-point union $D_j=(C_j,c_j)\vee (E_j,c_j)$ where $E_j=f_j(\bft_j)$ is a wedge of arcs with the weak topology. The quotient map $f_j:\tY_j\to D_j$ is also a homotopy equivalence (see Figure \ref{fig7}). We will also write $\bfe_{j,\beta}$ to denote the arc $f_j(\bfe_{j,\beta})$ in $E_j$ and $\wt{y}_{j,\beta}$ to denote its endpoint. We give a specific construction of a homotopy inverse $g_j:\tY_j/T_j\to \tY_j$ of $f_j$ since we will need for it to have special features.
\begin{figure}[H]
\centering \includegraphics[height=1.3in]{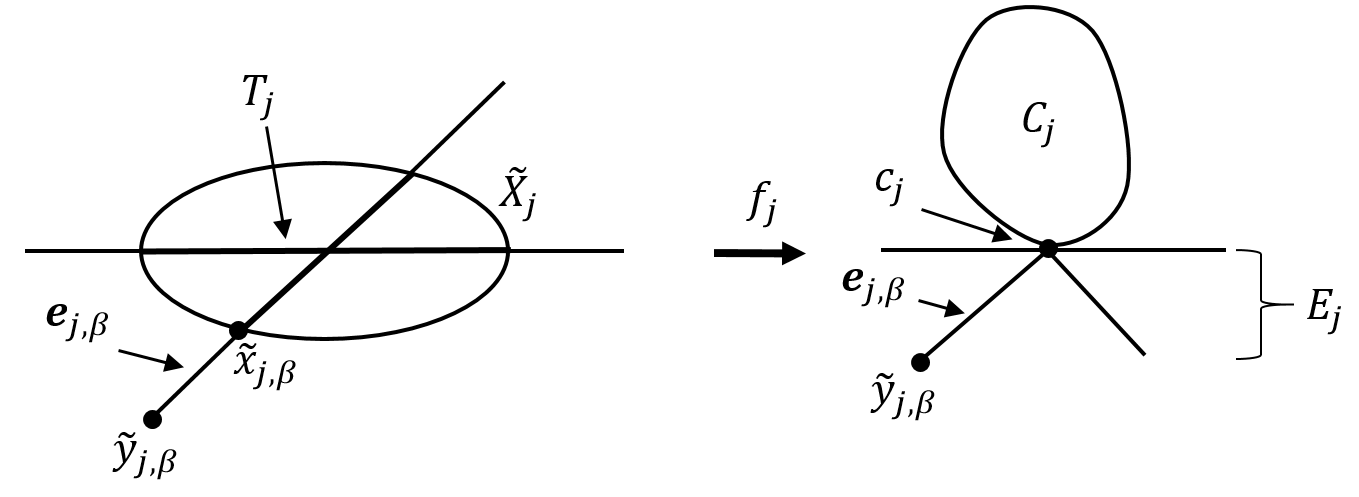}
\caption{\label{fig7} The quotient map $f_j:\wt{Y}_j\to D_j$ where $\wt{X}_j$ is illustrated as a disk in the domain.}
\end{figure}
\begin{itemize}
\item Since $(\tY_j,\tX_j)$ has the homotopy extension property, there is a retraction $h_1:\tY_j\times\ui\to \tY_j\times\{0\}\cup \tX_j\times \ui$. Instead of using an arbitrary retraction, we choose $h_1$ so that
\begin{itemize}
\item if $\bfa_{j,\beta}=\{\gamma\tau_{j,s}\mid s\in [0,1/3]\}$ is the third of $\bfe_{j,\beta}$ containing the arc-endpoint $\ty_{j,\beta}$, then $h_1$ projects $\bfa_{j,\beta}\times\ui$ vertically onto $\bfa_{j,\beta}\times \{0\}$,
\item $h_1(\bfe_{j,\beta}\times \ui)\subseteq (\bfe_{j,\beta}\times \{0\})\cup (\wt{x}_{j,\beta}\times \ui)$.
\end{itemize}
\item Since $(\tX_j,T_j)$ has the homotopy extension property, there is a retraction
$h_2:\tY_j\times\{0\}\cup \tX_j\times \ui\to \tY_j\times\{0\}\cup T_j\times \ui$ such that $h_2(\tX_j\times \ui)\subseteq \tX_j\times \{0\}\cup T_j\times \ui$.
\item Define $h_3:\tY_j\times\{0\}\cup T_j\times \ui\to \tY_j$ to be the identity on $\tY_j\times\{0\}$ and, on $T_j\times \ui$, to be a choice of contraction $T_j\times \ui\to T_j$ for $T_j$ (there will be no benefit to choosing this to be a based contraction of $T_j$).
\end{itemize}
Let $H_j=h_3\circ h_2\circ h_1:\tY_j\times \ui\to\tY_j$ (see Figure \ref{fig6}). The map $H_j(y,1):\tY_j\to \tY_j$ is constant on $T_j$ and thus induces a unique map $g_j:D_j\to \tY_j$ satisfying $g_j\circ f_j(y)=H_j(y,1)$. By construction, $H_j$ is a homotopy from $id_{\tY_j}$ to $g_j\circ f_j$. Because we require $H_j(T_j\times I)\subseteq T_j$, $H_j$ cannot be constructed as a lift of the homotopy $Y_j\times \ui\to Y_j$ used in Section \ref{sectionwellpointed} (recall Figure \ref{fig4}).

\begin{figure}[H]
\centering \includegraphics[height=2.8in]{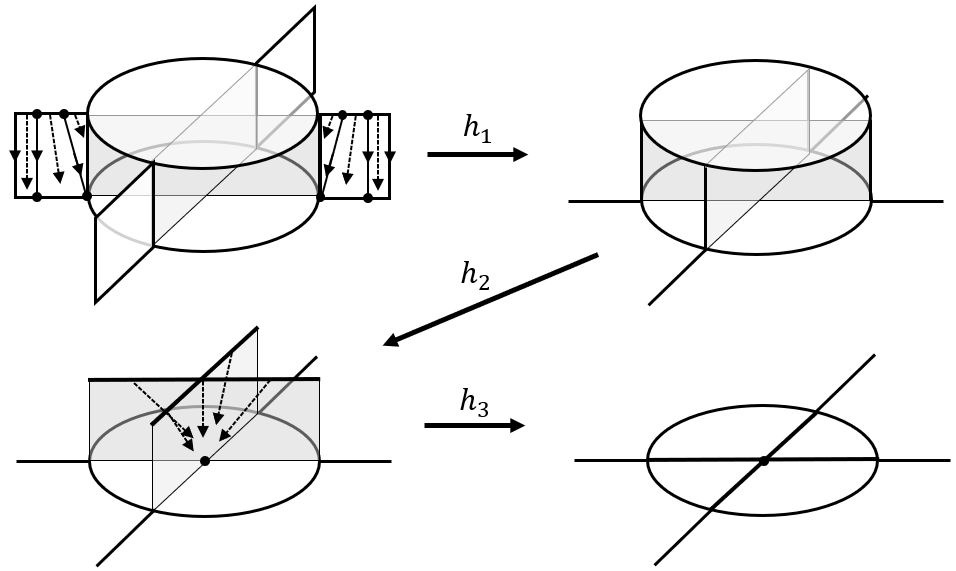}
\caption{\label{fig6}The homotopy $H_j:\tY_j\times\ui\to\tY_j$ constructed as a composition. There is a square attached to $\wt{X}_j\times \ui$ (represented by the cylinder), at each vertex of $T_j$; however, this is not reflected in this illustration for the sake of clarity. $T_j\times \ui$ is represented by the shaded gray surface.}
\end{figure}

The map $f_j\circ H_j:\tY_j\times \ui\to D_j$ sends $T_j\times\ui$ to the point $c_j$ and so there is a unique map $G_j:D_j \times \ui\to D_j$ making the following diagram commute.
\[\xymatrix{
\tY_j\times \ui \ar[d]_-{f_j\times id} \ar[r]^-{H_j} & \tY_j \ar[d]^-{f_j} \\
D_j \times \ui \ar[r]_-{G_j} & D_j
}\]
Since $G_j(f_j(y),1)=f_j\circ H_j(y,1)=f_j\circ g_j\circ f_j(y)$ where $f_j$ is surjective, it follows that $G_j(d,1)=f_j\circ g_j(d)$. Thus $G_j$ is a homotopy from $id_{D_j}$ to $f_j\circ g_j$.

While the construction of $f_j$, $g_j$, $H_j$, and $G_j$ is mostly standard (\cite[Prop. 0.17]{Hatcher}), our choice of $h_1$ and $h_2$ ensure the following important features.
\begin{enumerate}
\item $H_j$ and $G_j$ are the constant homotopies on a uniform neighborhood of every arc-endpoint,
\item $f_{j}(\tX_j)\subseteq C_{j}$ and $g_{j}(C_{j})\subseteq \tX_j$,
\item $H_j(\tX_j\times\ui)\subseteq \tX_j$ and $G_j(C_j\times \ui)\subseteq C_j$,
\item $f_{j}(\bft_j)\subseteq E_{j}$ and $g_{j}(E_{j})\subseteq \bft_j$,
\item $H_j(\bft_j\times\ui)\subseteq \bft_j$ and $G_j(E_j\times\ui)\subseteq E_j$.
\end{enumerate}
Note that (2) and (3) imply that the restricted maps $(f_j)|_{\tX_j}:\tX_j\to C_j$ and $(g_j)|_{C_j}:C_j\to \tX_j$ are homotopy inverses. Similarly, (4) and (5) imply that $(f_j)|_{\bft_j}:\bft_j\to E_j$ and $(g_j)|_{E_j}:E_j\to \bft_j$ are homotopy inverses.

\subsection{Quotients of $\tY_j$ by translates of $T_j$}\label{sectiontranslatesoftj}

The fundamental group $\pi_1(Y_j)$ acts on $\tY_j$ by deck transformation: $\Delta_{\beta}:\tY_j\to \tY_j$, $\Delta_{\beta}(\alpha)=\beta\alpha$, $\beta\in\pi_1(Y_j)$. Note that $\Delta_{\beta}$ maps $m$-cells to $m$-cells, $\Delta_{\beta}(\tX_j)=\tX_j$, and that $\Delta_{\beta}$ permutes the discrete set of arc-endpoints $\{\wt{y}_{j,\gamma}\mid \gamma\in \pi_1(Y_j)\}$. In particular, the translated trees $\beta T_{j}=\Delta_{\beta}(T_j)$ and $\beta \bft_j=\Delta_{\beta}(\bft_j)$ are maximal trees in the 1-skeleton of $\tX_j$ and $\tY_j$ respectively. Later on, we will need to consider quotients of $\tY_j$ by arbitrary translates of the tree $\beta T_j$. For this purpose, we establish notation for the corresponding homotopy equivalences. 

Let $D_{j,\beta}= \tY_j/\beta T_{j}$ be the quotient of $\tY_j$ obtained by collapsing $\beta T_{j}$ to a point. If $f_{j,\beta}:\tY_j\to D_{j,\beta}$ is the quotient map, then there is a unique homomorphism $\delta_{\beta}:D_j\to D_{j,\beta}$ such that the left square in the diagram below commutes. Define $g_{j,\beta}:D_{j,\beta}\to \tY_j$ by $g_{j,\beta}=\Delta_{\beta}\circ g_j\circ \delta_{\beta}^{-1}$ so the square on the right commutes. 
\[\xymatrix{
\tY_j \ar[r]^-{f_j} \ar[d]_-{\Delta_{\beta}} & D_j \ar[r]^-{g_j} \ar[d]_-{\delta_{\beta}} & \tY_j \ar[d]^-{\Delta_{\beta}} \\
\tY_j \ar[r]_-{f_{j,\beta}} & D_{j,\beta} \ar[r]_-{g_{j,\beta}} & \tY_j
}\]
Since the vertical maps are homeomorphisms, $f_{j,\beta}$ and $g_{j,\beta}$ are homotopy equivalences. To verify that these are homotopy inverses of each other, we construct the $\beta$-translates of $H_j$ and $G_j$. Set
\begin{enumerate}
\item $H_{j,\beta}=\Delta_{\beta}\circ H_j\circ(\Delta_{\beta}^{-1}\times id_{\ui})$,
\item $G_{j,\beta}=\delta_{\beta}\circ G_j\circ(\delta_{\beta}^{-1}\times id_{\ui})$.
\end{enumerate}
We now have maps $H_{j,\beta}:\tY_j\times \ui\to \tY_j$ and $G_{j,\beta}:D_j\times \ui\to D_j$ that make the left and right faces of the following cube commute. The top face commutes by the definition of $G_j$. The front and back faces commute by the definition of $\wt{f}_{j,\beta}$. Since the vertical maps are homeomorphisms, the bottom face commutes.
\begin{equation}\label{C1}
\begin{tikzcd}[row sep=2.5em]
\tY_j\times I \arrow[rr,"{f_j\times id}"] \arrow[dr,swap,"{H_{j}}"] \arrow[dd,swap,"{\Delta_{\beta}\times id}"] &&
  D_{j}\times I \arrow[dd,swap,"{\delta_{\beta}\times id}" near start] \arrow[dr,"G_j"] \\
& \tY_j \arrow[rr,swap,crossing over,"f_j" near start] &&
  D_j \arrow[dd,"{\delta_{\beta}}"] \\
\tY_j\times I \arrow[rr,"{f_{j,\beta}\times id}" near end] \arrow[dr,swap,"{H_{j,\beta}}"] && D_{j,\beta}\times I \arrow[dr,swap,"{G_{j,\beta}}"] \\
& \tY_j \arrow[rr,swap,"f_{j,\beta}"] \arrow[uu,<-,crossing over,"\Delta_{\beta}" near end]&& D_{j,\beta}
\end{tikzcd}\tag{C1}
\end{equation}
A straightforward check shows that $H_{j,\beta}$ is a homotopy from $id_{\tY_j}$ to $g_{j,\beta}\circ f_{j,\beta}$ and $G_{j,\beta}$ is a homotopy from $id_{D_j}$ to $f_{j,\beta}\circ g_{j,\beta}$. 

Note that we have an analogous wedge point $c_{j,\beta}=f_{j,\beta}(\beta T_j)=\delta_{\beta}(c_j)$ and subspaces $C_{j,\beta}=\delta_{\beta}(C_j)$ and $E_{j,\beta}=\delta_{\beta}(E_j)$ from which we have the decomposition $D_{j,\beta}=(C_{j,\beta},c_{j,\beta})\vee (E_{j,\beta},c_{j,\beta})$. Finally, we point out that the $\beta$-translated homotopies $H_{j,\beta}$ and $G_{j,\beta}$ enjoy the following properties just like the original maps $H_j$ and $G_j$.
\begin{enumerate}
\item $H_{j,\beta}$ and $G_{j,\beta}$ are the constant homotopies on a uniform neighborhood of every arc-endpoint,
\item $f_{j,\beta}(\tX_j)\subseteq C_{j,\beta}$ and $g_{j,\beta}(C_{j,\beta})\subseteq \tX_j$,
\item $H_{j,\beta}(\tX_j\times\ui)\subseteq \tX_j$ and $G_{j,\beta}(C_{j,\beta}\times \ui)\subseteq C_{j,\beta}$,
\item $f_{j,\beta}(\beta\bft_j)\subseteq E_{j,\beta}$ and $g_{j,\beta}(E_{j,\beta})\subseteq \beta\bft_j$,
\item $H_{j,\beta}(\beta\bft_j\times\ui)\subseteq \beta\bft_j$ and $G_{j,\beta}(E_{j,\beta}\times\ui)\subseteq E_{j,\beta}$.
\end{enumerate}
Note that (2) and (3) imply that the restricted maps $(f_{j,\beta})|_{\tX_j}:\tX_j\to C_{j,\beta}$ and $(g_{j,\beta})|_{C_{j,\beta}}:C_{j,\beta}\to \tX_j$ are homotopy inverses. Similarly, (4) and (5) imply that $(f_{j,\beta})|_{\beta\bft_j}:\beta\bft_j\to E_{j,\beta}$ and $(g_{j,\beta})|_{E_{j,\beta}}:E_{j,\beta}\to \beta\bft_j$ are homotopy inverses.

\begin{remark}\label{deltagammaremark}
For $\beta,\gamma\in\pi_1(Y_j)$, we also have a canonical homeomorphism $D_{j,\beta}\to D_{j,\gamma\beta}$. Formally, this map is $\delta_{\gamma\beta}\circ\delta_{\beta}^{-1}$, however, since it is determined by left multiplication by $\gamma$, we will also denote it by $\delta_{\gamma}$. With this definition, the following diagram commutes for all $j\in\bbn$ and $\beta,\gamma\in\pi_1(Y_j)$.
\[\xymatrix{
\tY_j \ar[r]^-{f_{j,\beta}} \ar[d]_-{\Delta_{\gamma}} & D_{j,\beta} \ar[r]^-{g_{j,\beta}} \ar[d]_-{\delta_{\gamma}} & \tY_j \ar[d]^-{\Delta_{\gamma}} \\
\tY_j \ar[r]_-{f_{j,\gamma\beta}} & D_{j,\gamma\beta} \ar[r]_-{g_{j,\gamma\beta}} & \tY_j
}\]
\end{remark}

\section{The inverse limits $\wh{Y}$ and $\wh{Z}$}

Since $\pi_n(X)\cong \pi_n(\tX)\cong\pi_n(\tY)$ for all $n\geq 2$, we wish to analyze the homotopical structure of $\tY$. However, directly verifying the continuity of the desired deformations of $\tY$ appears to be exceptionally tedious. Thus we seek a detailed description of the inverse system $(\tY_{\leq k},\wt{r}_{k+1,k})$. 

\subsection{The behavior of $\wt{r}_{k+1,k}:\tY_{\leq k+1}\to \tY_{\leq k}$}\label{sectionstructureofyk}

For each $k\in\bbn$, the universal covering space $\wt{Y}_{\leq k}$ is a CW-complex which is the union of homeomorphic copies of $\tY_j$ attached to each other in a tree-like fashion. We establish the following notation to keep track of the exact location of such subspaces. Recall that $\wt{Y}_{\leq k}$ is the set of path-homotopy classes of paths in $\wt{Y}_{\leq k}$ starting at $\wt{y}_0$ and the covering map $p_{\leq k}:\tY_{\leq k}\to Y_{\leq k}$ is the endpoint projection so if $\alpha=[a]$, then $p_k(\alpha)=a(1)$.

\begin{definition}
Fix $1\leq j\leq k$ and let $\alpha:\ui\to Y_{\leq k}$ be a reduced loop based at $y_0$. We say that $\alpha$ is \textit{non-$Y_{j}$-terminal} if either $\alpha$ is constant or if for the maximal element $(a,b)\in \ov{\alpha}$, the loop $\alpha|_{[a,b]}$ has image in $\bigcup_{i\neq j}Y_i$. Let $\ntkj\subseteq \pi_1(Y_{\leq k})$ denote the subset of homotopy classes of non-$Y_{j}$-terminal reduced loops.
\end{definition}

An element $\alpha\in\ntkj$ corresponds to a uniquely to a reduced word $w_{\alpha}$ in the free product $\pi_1(Y_{\leq k})=\ast_{j=1}^{k}\pi_1(Y_j)$, which does not terminate in a letter from $\pi_1(Y_j)$. Since for every $\alpha\in \pi_1(Y_{\leq k})$, there exists some $j\in\bbn$ for which $w_{\alpha}$ does not end in a letter from $\pi_1(Y_j)$, we have $p_{\leq k}^{-1}(y_0)=\pi_1(Y_{\leq k})=\bigcup_{j=1}^{k}\ntkj$.

When $j\leq k$, the coset projection $\pi_1(Y_{\leq k})\to \pi_1(Y_{\leq k})/\pi_1(Y_{j})$ restricts to a bijection $\ntkj\to \pi_1(Y_{\leq k})/\pi_1(Y_j)$, so the elements of $\ntkj$ are simply a canonical choice of representatives of the elements of $\pi_1(Y_{\leq k})/\pi_1(Y_{j})$. Moreover, $\ntkj$ indexes the set of connected components of $p_{\leq k}^{-1}(Y_j)$ in the following way: each connected component of $p_{\leq k}^{-1}(Y_j)$ is a set of the form $\mcb_{k,j,\alpha}=\{\alpha \beta\in\tY_{\leq k}\mid \beta\in\tY_j\}$ for some unique $\alpha\in \ntkj$ and is homeomorphic to $\tY_j$. In particular, there is a canonical homeomorphism $\Lambda_{k,j,\alpha}:\mcb_{k,j,\alpha}\to \tY_j$ given by $\Lambda_{k,j,\alpha}(\alpha\beta)=\beta$, which makes the following triangle commute.
\[\xymatrix{
\mcb_{k,j,\alpha} \ar[rr]^-{\Lambda_{k,j,\alpha}} \ar[dr]_-{p_{\leq k}} && \tY_j \ar[dl]^-{p_{j}}\\
& Y_j
}\]

\begin{proposition}
$\tY_{\leq k}$ is a CW complex, which has the weak topology with respect to the subcomplexes $\mcb_{k,j,\alpha}$, $1\leq j\leq k$, $\alpha\in \ntkj$.
\end{proposition}

\begin{remark}
Generally we use ``tilde" notation $\tX$ to indicate that Definition \ref{whiskertopologydef} is being applied and so we create a slight inconsistency with our notation for $\mcb_{k,j,\alpha}$ and $\mca_{k,j,\alpha}$. However, this notation most suitably reflects the fact that these two spaces are homeomorphic copies of $\tY_j$ and $\tX_j$ respectively.
\end{remark}

Allowing $k$ to vary, note that the map $\wt{r}_{k+1,k}$ collapses each subspace $\mcb_{k+1,k+1,\alpha}$ of $\tY_{\leq k+1}$ to a point. Additionally, $\wt{r}_{k+1,k}$ folds the subspaces $\mcb_{k+1,j,\alpha}$ (for fixed $j\leq k$) onto each other homeomorphically in a way that reflects word reduction in $\pi_1(Y_{\leq k})$. This folding is non-trivial because the homomorphism $(r_{k+1,k})_{\#}:\pi_1(Y_{\leq k+1})\to \pi_1(Y_{\leq k})$, which deletes letters from $\pi_1(Y_{\leq k+1})$ need not map $\nt_{k+1,j}$ into $\nt_{k,j}$. Indeed, if $1\neq \beta_{k+1}\in \pi_1(Y_{\leq k+1})$ and $1\neq \beta_j\in \pi_1(Y_{j})$ for $j<k+1$, then $\beta_j\beta_{k+1}\in \nt_{k+1,j}$ but $(r_{k+1,k})_{\#}(\beta_j\beta_{k+1})=\beta_j\notin  \ntkj$. We formalize this in the next remark.

\begin{remark}[Behavior of $\wt{r}_{k+1,k}$]
Fix $1\leq j\leq k+1$ and $\alpha\in\nt_{k+1,j}$.
\begin{itemize}
\item If $j=k+1$, then $\wt{r}_{k+1,k}$ maps $\mcb_{k+1,j,\alpha}$ to the single point $\wt{r}_{k+1,k}(\alpha)$.
\item If $1\leq j\leq k$, we write $\wt{r}_{k+1,k}(\alpha)=\alpha '\gamma$ for unique $\alpha '\in \ntkj$ and $\gamma\in \pi_1(Y_j)$. In this case, $\wt{r}_{k+1,k}$ maps $\mcb_{k+1,j,\alpha}$ homeomorphically onto $\mcb_{k,j,\alpha '}$ by $\alpha\beta\mapsto \alpha '\gamma \beta$ for $\beta\in\tY_j$. In other words, if $\Delta_{\gamma}:\tY_j\to \tY_j$ is the deck transformation $\Delta_{\gamma}(\beta)=\gamma\beta$, then the following diagram of homeomorphisms commutes.
\[\xymatrix{
\mcb_{k+1,j,\alpha} \ar[d]_-{\Lambda_{k+1,j,\alpha}} \ar[r]^-{\wt{r}_{k+1,k}} & \mcb_{k,j,\alpha '} \ar[d]^-{\Lambda_{k,j,\alpha '}}\\
\tY_j \ar[r]_-{\Delta_{\gamma}} & \tY_j
}\]
Indeed, because $\alpha ' \in\ntkj$ and $\gamma\beta\in\pi_1(Y_j)$, we have $\Lambda_{k,j,\alpha '}(\alpha '\gamma\beta)=\gamma\beta$.
Note that $\wt{r}_{k+1,k}(\alpha)\in\ntkj$ $\Leftrightarrow$ $\gamma=1$ $\Leftrightarrow$ $\Delta_{\gamma}=id_{\tY_j}$.
\end{itemize}
\end{remark}

\subsection{The quotient maps $\wt{f}_k:\tY_{\leq k}\to\tZ_k$}\label{sectionspacesqk}

For $1\leq j\leq k$ and $\alpha\in\ntkj$, let \[\mca_{k,j,\alpha}=\Lambda_{k,j,\alpha}^{-1}(\tX_j)=\{\alpha\tau_{j}\delta\in\mcb_{k,j,\alpha}\mid \delta\in\wt{X}_j\}\]be the copy of $\wt{X}_j$ in $\mcb_{k,j,\alpha}$. We will use the fixed tree $T_j\subseteq \tY_j$ and its translates $\beta T_j$, $\beta\in \pi_1(Y_j)$ to define a tree $\mct_{k,j,\alpha}$ in the 1-skeleton of $\mca_{k,j,\alpha}$. Naively, one could attempt to define such a tree as $\Lambda_{k,j,\alpha}^{-1}(T_j)$. However, we wish for these trees to be coherent with bonding maps $\wt{r}_{k+1,k}$ and the folding behavior of this map encountered in the previous section suggests that such a choice would ultimately fail. Although our choices will be entirely determined by our initial choices of $T_j$ in $\tX_j$, we must construct them by induction on $k$.

For $k=j=1$, we have $\tY_1=\mcb_{1,1,1}$ and $\Lambda_{1,1,1}=id_{\tY_1}$. Thus, we define $\mct_{1,1,1}=T_1$ and set $\scrt_1=\{\mct_{1,1,1}\}$.

Suppose we have defined every element of $\scrt_k=\{\mct_{k,j,\alpha}\mid 1\leq j\leq k,\alpha\in\ntkj\}$ so that each $\mct_{k,j,\alpha}$ is a maximal tree in the 1-skeleton of $\mca_{k,j,\alpha}$.
\begin{itemize}
\item[]\textbf{Case I:} If $j=k+1$, define $\mct_{k+1,j,\alpha}=\Lambda_{k+1,j,\alpha}^{-1}(T_{j})$. 
\item[] \textbf{Case II:} If $1\leq j\leq k$, write $(r_{k+1,k})_{\#}(\alpha)=\alpha '\gamma$ for $\alpha '\in \ntkj$ and $\gamma\in \pi_1(Y_j)$. The tree $\mct_{k,j,\alpha '}$ is defined in $\mcb_{k,j,\alpha '}$ by our induction hypothesis. Therefore, we set 
\[\mct_{k+1,j,\alpha}=\Lambda_{k+1,j,\alpha}^{-1}\circ\Delta_{\gamma}^{-1}\circ\Lambda_{k,j,\alpha '}(\mct_{k,j,\alpha '})\]
\end{itemize}
Set $\scrt_{k+1}=\{\mct_{k+1,j,\alpha}\mid 1\leq j\leq k+1,\alpha\in\nt_{k+1,j}\}$. This completes the induction.

\begin{remark}[Coherence of trees]\label{stabilizationremark}
The inductive construction of $\mct_{k+1,j,\alpha}$ was given precisely to match with the bonding maps $\wt{r}_{k+1,k}$.
\begin{enumerate}
\item[] \textbf{Case I:} If $j=k+1$, then $\wt{r}_{k+1,k}(\mct_{k+1,j,\alpha})=\wt{r}_{k+1,k}(\alpha)$ and $\Lambda_{k+1,j,\alpha}(\mct_{k+1,j,\alpha})=T_j$.
\item[] \textbf{Case II:} If $1\leq j\leq k$ and $\alpha '$ and $\gamma$ are as above then $\wt{r}_{k+1,k}$ maps $\mct_{k+1,j,\alpha}$ homeomorphically onto $\mct_{k,j,\alpha '}$. Moreover, if $\Lambda_{k,j,\alpha '}(\mct_{k,j,\alpha '})=\beta T_j$, then the above definition ensures \[\Lambda_{k+1,j,\alpha}(\mct_{k+1,j,\alpha})=\gamma^{-1}\beta T_j.\]
\end{enumerate}
\end{remark}

Our inductive construction of $\mct_{k,j,\alpha}$ implies that this tree will always correspond to some translation $\beta T_j$ under the ``bookkeeping" homeomorphism $\Lambda_{k,j,\alpha}$ (this is stated formally in the next proposition). We will see later on that the homomorphisms $\Lambda_{k,j,\alpha}$ are precisely the bridge required to witness the eventual stabilization of certain sequences of trees $\{\mct_{k,j,\alpha_{k}}\}$ where $j$ is fixed and $k\to\infty$.

\begin{proposition}\label{transprop}
For every $1\leq j\leq k<\infty$ and $\alpha\in\ntkj$, $\Lambda_{k,j,\alpha}(\mct_{k,j,\alpha})=\beta  T_j$ in $\tY_j$ for some $\beta\in \pi_1(Y_j)$.
\end{proposition}

Let $\tZ_k$ be the quotient space of $\wt{Y}_{\leq k}$ obtained by identifying each tree $\mct_{k,j,\alpha}\in\scrt_{k}$ to a point and $\wt{f}_{k}:\wt{Y}_{\leq k}\to \tZ_k$ be the quotient map. Let $z_k=\wt{f}_k(\wt{y}_0)$ be the basepoint of $Z_k$. Since $\scrt_{k}$ consists of a collection of disjoint, contractible subcomplexes in $\wt{Y}_{\leq k}$, it is clear that quotient map $\wt{f}_{k}$ is a homotopy equivalence of CW-complexes. However, we wish to choose homotopy inverses for $\wt{f}_k$ in a coherent way. Toward this end, we first show that the spaces $\tZ_k$ are part of a uniquely determined inverse system.

\begin{lemma}\label{smaplemma}
For each $k\in\bbn$, there is a unique map $\wt{s}_{k+1,k}:\tZ_{k+1}\to \tZ_{k}$ such that the following diagram commutes.
\[\xymatrix{
\wt{Y}_{\leq k+1} \ar[r]^-{\wt{r}_{k+1}} \ar[d]_-{\wt{f}_{k+1}} & \wt{Y}_{\leq k} \ar[d]^-{\wt{f}_k} \\ \tZ_{k+1} \ar@{-->}[r]_-{\wt{s}_{k+1}} & \tZ_k
}\]
\end{lemma}

\begin{proof}
Since $\wt{f}_{k+1}$ collapses each tree $\mct_{k+1,j,\alpha}$ to a point, it suffices to show that $\wt{f}_k\circ \wt{r}_{k+1,k}$ also collapses $\mct_{k+1,j,\alpha}$ to a point. If $j=k+1$, then $\wt{r}_{k+1,k}$ maps $\mct_{k+1,j,\alpha}$ to $\wt{r}_{k+1,k}(\alpha)$ and the conclusion is clear. Suppose $1\leq j\leq k$ and $(r_{k+1,k})_{\#}( \alpha)=\alpha ' \gamma$ for $\alpha '\in\ntkj$ and $\gamma\in\pi_1(Y_j)$. By construction, $\wt{r}_{k+1,k}$ maps $\mct_{k+1,j,\alpha}$ homeomorphically onto $\mct_{k,j,\alpha '}$. Since $\wt{f}_k$ maps $\mct_{k,j,\alpha '}$ to a point, the conclusion follows.
\end{proof}

\begin{definition}
Let $\wh{Z}=\varprojlim_{k}(Z_k,\wt{s}_{k+1,k})$ be the inverse limit with basepoint $\wh{z}_0=(z_k)$ and projection maps $\wt{s}_k:\wh{Z}\to Z_k$. Additionally, let $\wh{f}=\ilim_{k}\wt{f}_k:(\wh{Y},\wh{y}_0)\to (\wh{Z},\wh{z}_0)$ be the inverse limit map. 
\end{definition}

Working toward the construction of a homotopy inverse of $\wh{f}$, we now construct a specific, coherent system of homotopy inverses $\{\wt{g}_k\}$ for the sequence $\{\wt{f}_k\}$. Fix $1\leq j\leq k$, $\alpha\in\ntkj$, and set $\mcd_{k,j,\alpha}=\wt{f}_k(\mcb_{k,j,\alpha})$. Notice that $\tZ_k$ is a CW-complex with the weak topology with respect the set of subcomplexes of the form $\mcd_{k,j,\alpha}$. Let $\wt{f}_{k,j,\alpha}:\mcb_{k,j,\alpha}\to \mcd_{k,j,\alpha}$ be the quotient map, which is the restriction of $\wt{f}_k$ to $\mcb_{k,j,\alpha}$. Recall that $\wt{f}_{k,j,\alpha}$ collapses $\mct_{k,j,\alpha}$ to a point and, by Proposition \ref{transprop}, $\mct_{k,j,\alpha}$ corresponds to some translation $\beta T_j$ in $\tY_j$, i.e. $\Lambda_{k,j,\alpha}(\mct_{k,j,\alpha})=\beta T_j$. In Section \ref{sectionfixinghe}, we let $\wt{f}_{j,\beta }:\tY_j\to D_{j,\beta}$ be the quotient map collapsing $\beta T_j$ to a point. These observations make the next proposition immediate.

\begin{proposition}\label{littlelambdaprop}
Suppose $\Lambda_{k,j,\alpha}(\mct_{k,j,\alpha})=\beta T_j$ for $\beta\in\pi_1(Y_j)$. Then there is a canonical homeomorphism $\lambda_{k,j,\alpha}:\mcd_{k,j,\alpha}\to D_{j,\beta }$ that makes the following square commute.
\[\xymatrix{
\mcb_{k,j,\alpha} \ar[r]^-{\Lambda_{k,j,\alpha}} \ar[d]_-{\wt{f}_{k,j,\alpha}} & \tY_j \ar[d]^-{f_{j,\beta }} \\\mcd_{k,j,\alpha} \ar@{-->}[r]_-{\lambda_{k,j,\alpha}} & D_{j,\beta }
}\]
\end{proposition}

We use the homeomorphisms $\lambda_{k,j,\alpha}$ to maintain track of what the maps $\wt{s}_{k+1,k}$ do to the spaces $\mcd_{k+1,j,\alpha}$ when $j\leq k$.
\begin{lemma}\label{cubelemma}
Fix $1\leq j< k+1$ and $\alpha\in\nt_{k+1,j}$. Suppose
\begin{itemize}
\item $(r_{k+1,k})_{\#}(\alpha)=\alpha ' \gamma  $ for $\alpha '\in  \ntkj$ and $\gamma \in \pi_1(Y_j)$
\item and $\Lambda_{k,j,\alpha '}(\mct_{k,j,\alpha '})=\beta T_j$ for $\beta \in \pi_1(Y_j)$.
\end{itemize}
Then the following cube commutes.
\begin{equation}\label{C2}
\begin{tikzcd}[row sep=2.5em]
\mcb_{k+1,j,\alpha} \arrow[rr,"{\wt{r}_{k+1,k}}"] \arrow[dr,swap,"{\Lambda_{k+1,j,\alpha}}"] \arrow[dd,swap,"{\wt{f}_{k+1,j,\alpha}}"] &&
  \mcb_{k,j,\alpha '} \arrow[dd,swap,"{\wt{f}_{k+1,j,\alpha}}" near start] \arrow[dr,"\Lambda_{k,j,\alpha '}"] \\
& \tY_j \arrow[rr,swap,crossing over,"\Delta_{\gamma}" near start] &&
  \tY_j \arrow[dd,"{f_{j,\beta}}"] \\
\mcd_{k+1,j,\alpha} \arrow[rr,"{\wt{s}_{k+1,k}}" near end] \arrow[dr,swap,"{\lambda_{k+1,j,\alpha}}"] && \mcd_{k,j,\alpha '} \arrow[dr,swap,"{\lambda_{k,j,\alpha '}}"] \\
& D_{j,\gamma^{-1}\beta} \arrow[rr,swap,"\delta_{\gamma}"] \arrow[uu,<-,crossing over,"f_{j,\gamma^{-1}\beta}" near end]&& D_{j,\beta}
\end{tikzcd}\tag{C2}
\end{equation}
\end{lemma}

\begin{proof}
Recall from Remark \ref{stabilizationremark} that $\Lambda_{k+1,j,\alpha }(\mct_{k+1,j,\alpha })=\gamma^{-1}\beta T_j$. Commutativity of the top face was verified in Section \ref{sectionstructureofyk}. The left and right faces commute by Proposition \ref{littlelambdaprop}. The front face commutes is a special case of the left square in Remark \ref{deltagammaremark}. The setup of the lemma is precisely the situation where $\wt{r}_{k+1,k}$ maps $\mcb_{k+1,j,\alpha}$ homeomorphically to $\mcb_{k,j,\alpha '}$. Therefore, the back face commutes by the definition of $\wt{s}_{k+1,k}$ (recall Lemma \ref{smaplemma}). Since the vertical maps are surjective and all other faces commute, the bottom face commutes.
\end{proof}

\subsection{A homotopy inverse $\wt{g}_k$ for $\wt{f}_k$}

In Section \ref{sectiontranslatesoftj}, we fixed a homotopy inverse $g_{j,\beta}:D_{j,\beta}\to \wt{Y}_j$ of $f_{j,\beta}$ and homotopies $H_{j,\beta}$ and $G_{j,\beta}$. Using these pre-defined structures, we now construct maps $\wt{g}_k:\tZ_k\to \wt{Y}_{\leq k}$ inductively as follows. 

When $k=1$, we have $\tY_1=\tY_j$, $\tZ_1=D_1$ and $\wt{f}_{1}=f_1$. Thus we define $\wt{g}_1=g_1$. For our induction hypothesis, we suppose that $\wt{g}_k$ has been defined so that for all $1\leq j\leq k$ and $\alpha\in\ntkj$, $\wt{g}_k(\mcd_{k,j,\alpha})\subseteq \mcb_{k,j,\alpha}$ and, in particular, $\wt{g}_k$ maps the arc-endpoints of $\mcd_{k,j,\alpha}$ bijectively to the arc-endpoints of $\mcb_{k,j,\alpha}$. Let $\wt{g}_{k,j,\alpha}:\mcd_{k,j,\alpha}\to \mcb_{k,j,\alpha}$ be the corresponding restriction of $\wt{g}_k$.

Fix $1\leq j\leq k+1$. We will determine $\wt{g}_{k+1}$ by defining its restriction to each subcomplex $\mcd_{k+1,j,\alpha}$ as a map $\wt{g}_{k+1,j,\alpha}:\mcd_{k+1,j,\alpha}\to \mcb_{k+1,j,\alpha}$.\\\\
\noindent \textbf{Case I:} If $j=k+1$, then $\mct_{k+1,j,\alpha}$ was constructed so that $\Lambda_{k+1,j,\alpha}(\mct_{k+1,j,\alpha})=T_j$. By Lemma \ref{littlelambdaprop} (in the case $\beta=1$), the left square below commutes. We define $\wt{g}_{k+1,j,\alpha}=\Lambda_{k+1,j,\alpha}^{-1}\circ g_{j} \circ \lambda_{k+1,j,\alpha}$  so that the diagram on the right commutes.
\[\xymatrix{
\mcb_{k+1,j,\alpha} \ar[r]^-{\wt{f}_{k+1}} \ar[d]_-{\Lambda_{k+1,j,\alpha}} & \mcd_{k+1,j,\alpha} \ar[r]^-{\wt{g}_{k+1}} \ar[d]^-{\lambda_{k+1,j,\alpha}} & \mcb_{k+1,j,\alpha} \ar[d]^-{\Lambda_{k+1,j,\alpha}}\\
\tY_j \ar[r]_-{f_j} & D_j \ar[r]_-{g_j} & \tY_j 
}\]
\noindent \textbf{Case II:} Suppose $1\leq j\leq k$ and $\alpha\in\nt_{k+1,j}$. Write
\begin{itemize}
\item $(r_{k+1,k})_{\#}(\alpha)=\alpha '\gamma$ for $\alpha '\in\ntkj$ and $\gamma \in \pi_1(Y_j)$
\item and $\Lambda_{k,j,\alpha '}(\mct_{k,j,\alpha '})=\beta T_j$ for $\beta \in \pi_1(Y_j)$.
\end{itemize}
The restricted maps $\wt{r}_{k+1,k}:\mcb_{k+1,j,\alpha}\mapsto  \mcb_{k,j,\alpha '}$ and $\wt{s}_{k+1,k}:\mcd_{k+1,j,\alpha}\mapsto  \mcd_{k,j,\alpha '}$ are homeomorphisms. Since $\wt{g}_{k,j,\alpha '}$ is defined by hypothesis, we set 
$\wt{g}_{k+1,j,\alpha}=\wt{r}_{k+1,k}|_{\mcb_{k+1,j,\alpha}}^{-1}\circ \wt{g}_{k,j,\alpha '}\circ (\wt{s}_{k+1,k})|_{\mcd_{k+1,j,\alpha}}$ so the the following diagram commutes.
\[\xymatrix{
\mcb_{k+1,j,\alpha} \ar[r]^-{\wt{r}_{k+1,k}} & \mcb_{k,j,\alpha '} \\
\mcd_{k+1,j,\alpha} \ar[u]^-{\wt{g}_{k+1,j,\alpha}}\ar[r]_-{\wt{s}_{k+1,k}} & \mcd_{k,j,\alpha '} \ar[u]_-{\wt{g}_{k,j,\alpha '}}
}\]
In both cases, $\wt{g}_{k+1,j,\alpha}$ is continuous and maps arc-endpoints bijectively to arc-endpoints. This completes the definition of all $\wt{g}_{k,j,\alpha}$. We will use the next lemma to prove that $\wt{g}_k$ is well-defined.

\begin{lemma}\label{gdeflemma}
If $1\leq j\leq k<\infty$, $\alpha\in\ntkj$, and $\Lambda_{k,j,\alpha}(\mct_{k,j,\alpha})=\beta T_j$ for $\beta\in\pi_1(Y_j)$, then the following square commutes.
\[\xymatrix{
\mcb_{k,j,\alpha} \ar[r]^-{\Lambda_{k,j,\alpha}}  & \tY_j  \\
\mcd_{k,j,\alpha} \ar[r]_-{\lambda_{k,j,\alpha}} \ar[u]^-{\wt{g}_{k,j,\alpha}} & D_{j,\beta} \ar[u]_-{g_{j,\beta}}
}\]
\end{lemma}

\begin{proof}
Fix $j\in\bbn$. We proceed by induction on $k$ for $k\geq j$. In the case $k=j$, $\wt{g}_{k,j,\alpha}$ is constructed according to Case I. In particular, $\beta=1$ and the commuting diagram is precisely the definition of $\wt{g}_{k,j,\alpha}$. Suppose that $k\geq j$ and that the diagram commutes for all $\alpha '\in\ntkj$. Since $k+1>j$, the map $\wt{g}_{k+1,j,\alpha}$ is constructed according to Case II. Using the notation from Case II (for $\alpha '\in\ntkj$, and $\gamma,\beta\in\pi_1(Y_j)$), consider the following cube, which is the ``inverse" of that in Lemma \ref{cubelemma}.
\begin{equation}\label{C3}
\begin{tikzcd}[row sep=2.5em]
\mcb_{k+1,j,\alpha} \arrow[rr,"{\wt{r}_{k+1,k}}"] \arrow[dr,swap,"{\Lambda_{k+1,j,\alpha}}"] \arrow[dd,<-,swap,"{\wt{g}_{k+1,j,\alpha}}"] &&
  \mcb_{k,j,\alpha '} \arrow[dd,<-,swap,"{\wt{g}_{k,j,\alpha '}}" near start] \arrow[dr,"\Lambda_{k,j,\alpha '}"] \\
& \tY_j \arrow[rr,swap,crossing over,"\Delta_{\gamma}" near start] &&
  \tY_j \arrow[dd,<-,"{g_{j,\beta}}"] \\
\mcd_{k+1,j,\alpha} \arrow[rr,"{\wt{s}_{k+1,k}}" near end] \arrow[dr,swap,"{\lambda_{k+1,j,\alpha}}"] && \mcd_{k,j,\alpha '} \arrow[dr,swap,"{\lambda_{k,j,\alpha '}}"] \\
& D_{j,\gamma^{-1}\beta} \arrow[rr,swap,"\delta_{\gamma}"] \arrow[uu,crossing over,"g_{j,\gamma^{-1}\beta}" near end]&& D_{j,\beta}
\end{tikzcd}\tag{C3}
\end{equation}
Recall that $\Lambda_{k+1,j,\alpha}(\mct_{k+1,j,\alpha})=\gamma^{-1}\beta T_j$ and therefore, it suffices to show that the left face commutes. The top and bottom faces are the same as in Lemma \ref{cubelemma} and still commute. The back face commutes by the definition of $\wt{g}_{k+1,j,\alpha}$. The commutativity of the front face is a case of the right square in Remark \ref{deltagammaremark}. The right face commutes by our induction hypothesis. Since all of the horizontal maps are homeomorphisms, we conclude that the left face commutes.
\end{proof}

\begin{theorem}
For every $k\in\bbn$, $\wt{g}_{k}:\tZ_k\to\tY_{\leq k}$ is well-defined and continuous. Moreover,
\begin{enumerate}
\item $\wt{g}_k\circ \wt{f}_k:\wt{Y}_{\leq k}\to \wt{Y}_{\leq k}$ is the identity on $p_{\leq k}^{-1}(y_0)$,
\item $\wt{f}_k\circ \wt{g}_k:\tZ_k\to \tZ_k$ is the identity on $\wt{f}_k(p_{\leq k}^{-1}(y_0))$,
\item and $\wt{g}_k\circ \wt{s}_{k+1,k}=\wt{r}_{k+1,k}\circ \wt{g}_{k+1}$.
\end{enumerate}
\end{theorem}

\begin{proof}
Fix $k\in\bbn$. Since $\tZ_k$ has the weak topology with respect to the subcomplexes $\mcd_{k,j,\alpha}$ and each $\wt{g}_{k,j,\alpha}$ is clearly continuous, it suffices to check that $\wt{g}_{k}$ is well defined. First, we make an observation: for any given $j$ and $\alpha$ recall that we have verified the commutativity of the diagram when $\Lambda_{k,j,\alpha}(\mct_{k,j,\alpha})=\beta T_j$.
\[\xymatrix{
\mcb_{k,j,\alpha} \ar[d]_-{\Lambda_{k,j,\alpha}} \ar[r]^-{\wt{f}_{k,j,\alpha}} & \mcd_{k,j,\alpha} \ar[d]_-{\lambda_{k,j,\alpha}} \ar[r]^-{\wt{g}_{k,j,\alpha}}  & \mcb_{k,j,\alpha} \ar[d]_-{\Lambda_{k,j,\alpha}} \ar[r]^-{\wt{f}_{k,j,\alpha}} & \mcd_{k,j,\alpha} \ar[d]_-{\lambda_{k,j,\alpha}} \\
\tY_j \ar[r]_-{f_{j,\beta}}  & D_{j,\beta} \ar[r]_-{g_{j,\beta}} &   \tY_j \ar[r]_-{f_{j,\beta}}  & D_{j,\beta}
}\]
In Section \ref{sectiontranslatesoftj}, we constructed $g_{j,\beta}$ and $f_{j,\beta}$ so that $g_{j,\beta}\circ f_{j,\beta}$ is the identity map on the arc-endpoints of $\tY_j$ and $f_{j,\beta}\circ g_{j,\beta}$ is the identity on the arc-endpoints of $D_{j,\beta}$. It follows from the diagram that $\wt{g}_{k,j,\beta}\circ \wt{f}_{k,j,\beta}$ and $\wt{f}_{k,j,\beta}\circ \wt{g}_{k,j,\beta}$ are the identities on the respective sets of arc-endpoints.

Two distinct subcomplexes of the form $\mcd_{k,j,\alpha}$ (respectively $\mcb_{k,j,\alpha}$) either meet at a single point or do not meet at all. Suppose $\{x\}=\mcd_{k,j,\alpha}\cap \mcd_{k,j',\alpha '}$ where $\mcd_{k,j,\alpha}\neq \mcd_{k,j',\alpha '}$. We check that $\wt{g}_{k,j,\alpha}(x)=\wt{g}_{k,j',\alpha '}(x)$. By the definition of $\wt{f}_k$, we have $\wt{f}_k(y)=x$ for a unique point $y$. In particular, $\{y\}=\mcb_{k,j,\alpha}\cap \mcb_{k,j',\alpha '}$. Thus $\wt{f}_{k,j,\alpha}(y)=\wt{f}_{k,j',\alpha '}(y)=x$. Since $\wt{g}_{k,j,\alpha}\circ \wt{f}_{k,j,\alpha}$ and $\wt{g}_{k,j,\alpha '}\circ \wt{f}_{k,j,\alpha '}$ are the identity on the arc-endpoints, we have \[\wt{g}_{k,j,\alpha}(x)=\wt{g}_{k,j,\alpha}(\wt{f}_{k,j,\alpha}(y))=y=\wt{g}_{k,j',\alpha '}(\wt{f}_{k,j',\alpha '}(y))=\wt{g}_{k,j',\alpha '}(x).\]
This proves that $\wt{g}_k$ is well-defined.

Now that well-definedness of $\wt{g}_k$ is established, both (1) and (2) follow immediately from the fact in the first paragraph of the proof that for all $j$ and $\alpha$, the compositions $\wt{g}_{k,j,\alpha}\circ \wt{f}_{k,j,\alpha}$ and $\wt{g}_{k,j,\alpha}\circ \wt{f}_{k,j,\alpha}$ fix their respecive arc-endpoint sets.

For (3), we verify that the following square commutes by checking that the compositions agree on $\mcd_{k+1,j,\alpha}$. 
\[\xymatrix{
\wt{Y}_{\leq k+1}  \ar[r]^-{\wt{r}_{k+1,k}} & \wt{Y}_{\leq k} \\
Z_{k+1} \ar[u]^-{\wt{g}_{k+1}} \ar[r]_-{\wt{s}_{k+1,k}} & Z_k \ar[u]_-{\wt{g}_{k}}
}\]
When $1\leq j\leq k$, $\wt{r}_{k+1,k}\circ \wt{g}_{k+1}$ and $\wt{g}_{k}\circ \wt{s}_{k+1,k}$ agree on $\mcd_{k+1,j,\alpha}$ by definition of $\wt{g}_{k+1,j,\alpha}$. When $j=k+1$, we have $\wt{r}_{k+1,k}\circ\wt{g}_{k+1,j,\alpha}(\mcd_{k+1,j,\alpha})\subseteq \wt{r}_{k+1,k}(\mcb_{k+1,j,\alpha})=\wt{r}_{k+1,k}(\alpha)$. For the other composition, we have the following, 
\begin{eqnarray*}
\wt{g}_k\circ \wt{s}_{k+1,k}(\mcd_{k+1,j,\alpha}) &=& \wt{g}_k\circ \wt{s}_{k+1,k}\circ \wt{f}_{k+1}(\mcb_{k+1,j,\alpha})\\
&=& \wt{g}_k\circ \wt{f}_k\circ \wt{r}_{k+1,k}(\mcb_{k+1,j,\alpha})\\
&=& \wt{g}_k\circ \wt{f}_k(\wt{r}_{k+1,k}(\alpha))\\
&=& \wt{r}_{k+1,k}(\alpha)
\end{eqnarray*}
where the last equality follows from (1).
\end{proof}

Since the maps $\wt{g}_k$ agree with the bonding maps $\wt{r}_{k+1,k}$ and $\wt{s}_{k+1,k}$, we define $\wh{g}=\varprojlim_{k}\wt{g}_k:\wh{Z}\to\wh{Y}$ to be the inverse limit map. Note that $\wh{g}(\wh{z}_0)=\wh{y}_0$.

\subsection{Coherence of $\wt{f}_k$ and $\wt{g}_k$}

To show the limit maps $\wh{g}:\wh{Z}\to\wh{Y}$ and $\wh{f}:\wh{Y}\to\wh{Z}$ are homotopy inverses, we construct homotopies $\wt{H}_k$ from $id_{\tY_{\leq k}}$ to $\wt{g}_k\circ \wt{f}_k$ and $\wt{G}_k$ from $id_{\tZ_k}$ to $\wt{f}_k\circ \wt{g}_k$, which are coherent with the respective bonding maps.

\begin{theorem}\label{homotopytheorem}
For each $k\in\bbn$, there exist based homotopies $\wt{H}_k:\wt{Y}_{\leq k}\times I\to \wt{Y}_{\leq k}$ from $id_{\wt{Y}_{\leq k}}$ to $\wt{g}_k\circ \wt{f}_{k}$ and $\wt{G}_k:\tZ_k\times I\to \tZ_k$ from $id_{\tZ_k}$ to $\wt{f}_k\circ \wt{g}_{k}$ such that the following squares commute for all $k\in\bbn$.
\[\xymatrix{
\wt{Y}_{\leq k+1}\times I \ar[d]_-{\wt{H}_{k+1}} \ar[rr]^-{\wt{r}_{k+1,k}\times id} && \wt{Y}_{\leq k}\times I \ar[d]^-{\wt{H}_k} & \tZ_{k+1}\times I \ar[d]_-{\wt{G}_{k+1}} \ar[rr]^-{\wt{s}_{k+1,k}\times id} && \tZ_{k}\times I \ar[d]^-{\wt{G}_k} \\
\wt{Y}_{\leq k+1} \ar[rr]_-{\wt{r}_{k+1,k}} && \wt{Y}_{\leq k} & \tZ_{k+1} \ar[rr]_-{\wt{s}_{k+1,k}} && \tZ_k
}\]
\end{theorem}

\begin{remark}\label{squaresremark}
The key to proving Theorem \ref{homotopytheorem} is the relationship between the maps $\wt{f}_k$, $\wt{g}_k$ and the previously defined maps $f_{j,\beta}$ and $g_{j,\beta}$. In particular, whenever $1\leq j\leq k<\infty$, $\alpha\in\ntkj$, and $\Lambda_{k,j,\alpha}(\mct_{k,j,\alpha})=\beta T_j$, the following squares commute:
\[\xymatrix{
\mcb_{k,j,\alpha} \ar[d]_-{\wt{f}_{k,j,\alpha}} \ar[r]^-{\Lambda_{k,j,\alpha}} & \wt{Y}_j \ar[d]^-{f_{j,\beta}} & \mcb_{k,j,\alpha}  \ar[r]^-{\Lambda_{k,j,\alpha}} & \wt{Y}_j  \\
\mcd_{k,j,\alpha} \ar[r]_-{\lambda_{k,j,\alpha}} & D_{j,\beta}  & \mcd_{k,j,\alpha} \ar[r]_-{\lambda_{k,j,\alpha}} \ar[u]^-{\wt{g}_{k,j,\alpha}} & D_{j,\beta} \ar[u]_-{g_{j,\beta}}
}\]
\end{remark}

\begin{proof}[Proof of Theorem \ref{homotopytheorem}]
Let $k\in\bbn$. We define $\wt{H}_k$ and $\wt{G}_k$ piecewise by defining their values on the subcomplexes $\mcb_{k,j,\alpha}\times \ui$ and $\mcd_{k,j,\alpha}\times \ui$ of the respective domains. Fix $1\leq j\leq k$ and $\alpha\in\ntkj$ and suppose $\Lambda_{k,j,\alpha}(\mct_{k,j,\alpha})=\beta T_j$ for $\beta\in\pi_1(Y_j)$. Define maps $\wt{H}_{k,j,\alpha}:\mcb_{k,j,\alpha}\times \ui\to \mcb_{k,j,\alpha}$ and $\wt{G}_{k,j,\alpha}:\mcd_{k,j,\alpha}\times \ui\to \mcd_{k,j,\alpha}$ so that the following squares commute:
\[\xymatrix{
\mcb_{k,j,\alpha}\times\ui \ar[r]^-{\wt{H}_{k,j,\alpha}} \ar[d]_-{\Lambda_{k,j,\alpha}\times id} & \mcb_{k,j,\alpha} \ar[d]^-{\Lambda_{k,j,\alpha}} && \mcd_{k,j,\alpha}\times\ui \ar[r]^-{\wt{G}_{k,j,\alpha}} \ar[d]_-{\lambda_{k,j,\alpha}\times id} & \mcd_{k,j,\alpha} \ar[d]^-{\lambda_{k,j,\alpha}} \\
\tY_j\times\ui \ar[r]_-{H_{j,\beta}}  & \tY_j  && D_{j,\beta}\times\ui \ar[r]_-{G_{j,\beta}}  & D_{j,\beta}
}\]
When the left diagram is restricted to $t=0$, $H_{j,\beta}$ becomes the identity map. Therefore, $\wt{H}_{k,j,\alpha}$ does too. Recall that $H_{j,\beta}$ is a homotopy from $id_{\tY_j}$ to $g_{j,\beta}\circ f_{j,\beta}$. Therefore, $\wt{H}_{k,j,\alpha}$ is a homotopy from
\[\Lambda_{k,j,\alpha}^{-1}\circ id_{\tY_j}\circ \Lambda_{k,j,\alpha}=id_{\mcb_{k,j,\alpha}}\]
to
\begin{eqnarray*}
\Lambda_{k,j,\alpha}^{-1}\circ (g_{j,\beta}\circ f_{j,\beta})\circ \Lambda_{k,j,\alpha} &=& (\Lambda_{k,j,\alpha}^{-1}\circ g_{j,\beta}\circ \lambda_{k,j,\alpha})\circ (\lambda_{k,j,\alpha}^{-1}\circ f_{j,\beta}\circ \Lambda_{k,j,\alpha})\\
&=& \wt{g}_{k,j,\alpha}\circ \wt{f}_{k,j,\alpha}
\end{eqnarray*}
where the second equality comes from the left square in Remark \ref{squaresremark}. Because $G_{j,\beta}$ is a homotopy from $id_{D_{j,\beta}}$ to $f_{j,\beta}\circ g_{j,\beta}$, the same argument using the right square in Remark \ref{squaresremark} shows that $\wt{G}_{k,j,\alpha}$ is a homotopy from $id_{\mcd_{k,j,\alpha}}$ to $\wt{f}_{k,j,\alpha}\circ \wt{g}_{k,j,\alpha}$. Additionally, since $H_{j,\beta}$ and $G_{j,\beta}$ are the constant homotopy on the respective sets of arc-endpoints, the same holds for $\wt{H}_{k,j,\alpha}$ and $\wt{G}_{k,j,\alpha}$. This last observation ensures that $\wt{H}_k$ and $\wt{G}_k$ will be well-defined functions if we define $\wt{H}_k$ to agree with $\wt{H}_{k,j,\alpha}$ on $\mcb_{k,j,\alpha}$ and $\wt{G}_k$ to agree with $\wt{G}_{k,j,\alpha}$ on $\mcd_{k,j,\alpha}$ (a detailed proof follows the same elementary line of argument used to prove $\wt{g}_k$ is well-defined). Since $\wt{Y}_{\leq k}$ has the weak topology with respect to the subcomplexes $\mcb_{k,j,\alpha}$, $\wt{H}_k$ is continuous. Similarly, $\wt{G}_k$ is continuous.

With the definition and continuity of $\wt{H}_k$ and $\wt{G}_k$ established for all $k$, we fix $k$ and work toward proving that the two squares in the statement of the theorem commute. Let $1\leq j\leq k+1$ and $\alpha\in\nt_{k+1,j}$. We will show that $\wt{H}_k\circ (\wt{r}_{k+1,k}\times id)$ and $\wt{r}_{k+1,k}\circ \wt{H}_{k+1}$ agree on each subcomplex $\mcb_{k+1,j,\alpha}\times \ui$.\\

\noindent \textbf{Case I:} If $j=k+1$, then $\wt{r}_{k+1,k}(\mcb_{k+1,k,\alpha})=\wt{r}_{k+1,k}(\alpha)$ and so $\wt{r}_{k+1,k}(\wt{H}_{k+1}(\mcb_{k+1,k,\alpha}\times \ui))=\wt{r}_{k+1,k}(\mcb_{k+1,k,\alpha})=\wt{r}_{k+1,k}(\alpha)$. Moreover,
\[\wt{H}_{k}(\wt{r}_{k+1,k}(\mcb_{k+1,k,\alpha}\times \ui)) =\wt{H}_k(\{\wt{r}_{k+1,k}(\alpha)\}\times I)= \wt{r}_{k+1,k}(\alpha)
\]since $\wt{r}_{k+1,k}(\alpha)\in p_{\leq k}^{-1}(y_0)$ and $\wt{H}_k$ is the constant homotopy on $p_{\leq k}^{-1}(y_0)$.\\

\noindent \textbf{Case II:} Suppose $1\leq j\leq k$. Write $\wt{r}_{k+1,k}(\alpha)=\alpha '\gamma$ for $\alpha '\in\ntkj$ and $\gamma\in\pi_1(Y_j)$. If $\Lambda_{k,j,\alpha '}(\mct_{k,j,\alpha '})=\beta T_j$, then $\Lambda_{k+1,j,\alpha }(\mct_{k+1,j,\alpha })=\gamma^{-1}\beta T_j$. Since $\wt{r}_{k+1,k}$ maps $\mcb_{k+1,j,\alpha}$ homeomorphically onto $\mcb_{j,k,\alpha '}$, it suffices to show that $\wt{H}_{k,j,\alpha '}\circ (\wt{r}_{k+1,k}\times id)=\wt{r}_{k+1,k}\circ \wt{H}_{k+1,j,\alpha}$ agree on $\mcb_{k+1,j,\alpha}\times \ui$, i.e. that the top face of the following cube commutes.
\begin{equation}\label{C4}
\begin{tikzcd}[row sep=2.7em]
\mcb_{k+1,j,\alpha }\times \ui \arrow[rr,"{\wt{H}_{k+1,j,\alpha}}"] \arrow[dr,swap,"{\wt{r}_{k+1,k}\times id}"] \arrow[dd,swap,"{\Lambda_{k+1,j,\alpha}\times id}"] &&
  \mcb_{k+1,j,\alpha } \arrow[dd,swap,"{\Lambda_{k+1,j,\alpha }}" near start] \arrow[dr,"\wt{r}_{k+1,k}"] \\
& \mcb_{k,j,\alpha '}\times \ui \arrow[rr,swap,crossing over,"\wt{H}_{k,j,\alpha '}" near start] &&
  \mcb_{k,j,\alpha '} \arrow[dd,"{\Lambda_{k,j,\alpha '}}"] \\
\tY_j\times I \arrow[rr,swap,"{H_{j,\gamma^{-1}\beta}}" near end] \arrow[dr,swap,"{\Delta_{\gamma}\times id}"] && \tY_j \arrow[dr,"{\Delta_{\gamma}}"] \\
& \tY_j\times I \arrow[rr,swap,"H_{j,\beta}"] \arrow[uu,<-,crossing over,"\Lambda_{k,j,\alpha '}\times id" near end]&& \tY_j
\end{tikzcd}\tag{C4}
\end{equation}
The front and back faces commute by the definition of $\wt{H}_{k+1,j,\alpha}$ and $\wt{H}_{k,j,\alpha '}$. Commutativity of the right face was given in Section \ref{sectionstructureofyk} and the left face follows immediately. It suffices to check the bottom face. Recall that for any $\nu\in \pi_1(Y_j)$ and $\eta\in\wt{Y}_j$, the formula for $H_{j,\nu}$ is $H_{j,\nu}(\eta,t)=\nu H_{j}(\nu^{-1}\eta ,t)$. Therefore, if $(\eta,t)\in \tY_j\times \ui$, then
\[
H_{j,\beta}\circ (\Delta_{\gamma}\times id)(\eta,t) =H_{j,\beta}(\gamma\eta,t)= \beta H_j(\beta^{-1}\gamma\eta,t)\]
The other direction is
\[\Delta_{\gamma}\circ H_{j,\gamma^{-1}\beta}(\eta,t) = \Delta_{\gamma}(\gamma^{-1}\beta H_{j,\beta}(\beta^{-1}\gamma\eta,t))= \beta H_{j,\beta}(\beta^{-1}\gamma\eta,t)\]
This proves the bottom square commutes. We conclude that the top square commutes, completing the proof for Case II.

The proof for the coherence of the maps $\wt{G}_k$ with $\wt{s}_{k+1,k}$ is nearly identical, replacing the homomorphisms $\Lambda_{-,-,-}$ with $\lambda_{-,-,-}$ and $\Delta_-$ with $\delta_-$ so we omit the details.
\end{proof}

\begin{theorem}
The maps $\wh{f}:\wh{Y}\to \wh{Z}$ and $\wh{g}:\wh{Z}\to\wh{Y}$ are homotopy inverses.
\end{theorem}

\begin{proof}
The maps $\wt{H}_k$, $k\in\bbn$ form an inverse system of based homotopies.
\[\xymatrix{
 \wh{Y}\times \ui \ar[d]_-{\wh{H}} & \cdots \ar[r] & \wt{Y}_3\times \ui  \ar[d]_-{\wt{H}_3} \ar[r]^-{\wt{r}_{3,2}\times id} & \wt{Y}_2\times \ui \ar[d]_-{\wt{H}_2} \ar[r]^-{\wt{r}_{2,1}\times id} & \wt{Y}_1\times \ui \ar[d]_-{\wt{H}_1} \\
   \wh{Y} &\cdots \ar[r] & \wt{Y}_3 \ar[r]_-{\wt{r}_{3,2}} & \wt{Y}_2 \ar[r]_-{\wt{r}_{2,1}} & \wt{Y}_1
}\]
The limit of this first system is a map $\wh{H}:\wh{Y}\times \ui\to \wh{Y}$, which, by construction, is a based homotopy from $id_{\wh{Y}}$ to $\wh{g}\circ \wh{f}$. Here we are implicitly using the fact that inverse limits commute with finite products to identify the limit with $\varprojlim_{k}(\wt{Y}_{\leq k},\wt{r}_{k+1,k})\times \varprojlim_{k}(\ui, id)=\wh{Y}\times \ui$. 

Similarly, the homotopies $\wt{G}_k$, $k\in\bbn$, form an inverse system of based homotopies.
\[\xymatrix{
 \wh{Z}\times \ui \ar[d]_-{\wh{G}} & \cdots \ar[r] & \tZ_3\times \ui  \ar[d]_-{\wt{G}_3} \ar[r]^-{\wt{s}_{3,2}\times id} & \tZ_2\times \ui \ar[d]_-{\wt{G}_2} \ar[r]^-{\wt{s}_{2,1}\times id} & \tZ_1\times \ui \ar[d]_-{\wt{G}_1} \\
   \wh{Z} &\cdots \ar[r] & \tZ_3 \ar[r]_-{\wt{s}_{3,2}} & \tZ_2 \ar[r]_-{\wt{s}_{2,1}} & \tZ_1
}\]
The limit of this second system is a map $\wh{G}:\wh{Z}\times \ui\to \wh{Z}$, which, by construction, is a based homotopy from $id_{\wh{Z}}$ to $\wh{f}\circ \wh{g}$. 
\end{proof}

\begin{proposition}
For the maps $\wh{f}$, $\wh{H}$, and $\wh{G}$ defined above, the following square commutes.
\[\xymatrix{
\wh{Y} \times \ui \ar[r]^-{\wh{H}} \ar[d]_-{\wh{f}\times id} & \wh{Y} \ar[d]^-{\wh{f}} \\
\wh{Z}\times \ui \ar[r]_-{\wh{G}} & \wh{Z}
}\]
\end{proposition}

\begin{proof}
It suffices to check that for all $k\in\bbn$, the following square commutes. Once this is established, the result follows from taking the inverse limit over $k$ with the appropriate bonding maps in each position.
\[\xymatrix{
 \wt{Y}_{\leq k} \times \ui \ar[r]^-{\wt{H}_k} \ar[d]_-{\wt{f}_k\times id} & \wt{Y}_{\leq k} \ar[d]^-{\wt{f}_k} \\
 \tZ_k \times \ui \ar[r]_-{\wt{G}_k} & \tZ_k 
}\]
Fix $1\leq j\leq k<\infty$ and $\alpha\in\ntkj$. Suppose $\Lambda_{k,j,\alpha}(\mct_{k,j,\alpha})=\beta T_j$ for $\beta\in\pi_1(Y_j)$. We check that the two compositions agree on $\mcb_{k,j,\alpha}\times\ui$, i.e. that the top face of the following cube commutes.
\begin{equation}\label{C5}
\begin{tikzcd}[row sep=2.7em]
\mcb_{k,j,\alpha }\times \ui \arrow[rr,"{\wt{f}_{k,j,\alpha}\times id}"] \arrow[dr,swap,"{\wt{H}_{k,j,\alpha}}"] \arrow[dd,swap,"{\Lambda_{k,j,\alpha}\times id}"] &&
  \mcd_{k,j,\alpha }\times I \arrow[dd,swap,"{\lambda_{k,j,\alpha }\times id}" near start] \arrow[dr,"\wt{G}_{k,j,\alpha}"] \\
& \mcb_{k,j,\alpha } \arrow[rr,swap,crossing over,"\wt{f}_{k,j,\alpha}" near start] &&
  \mcd_{k,j,\alpha } \arrow[dd,"{\lambda_{k,j,\alpha}}"] \\
\tY_j\times I \arrow[rr,swap,"{f_{j,\beta}\times id}" near end] \arrow[dr,swap,"{H_{j,\beta}}"] && D_{j,\beta} \times I \arrow[dr,"{G_{j,\beta}}"] \\
& \tY_j \arrow[rr,swap,"f_{j,\beta}"] \arrow[uu,<-,crossing over,"\Lambda_{k,j,\alpha }" near end]&& D_{j,\beta}
\end{tikzcd}\tag{C5}
\end{equation}
The front and back faces commute by the definition of $\wt{f}_{k,j,\alpha}$. The left and right faces commute by the definitions of $\wt{H}_{k,j,\alpha}$ and $\wt{G}_{k,j,\alpha}$ respectively. The bottom face is the bottom face of Cube \ref{C1} in Section \ref{sectiontranslatesoftj}. Since the vertical maps are homeomorphisms, the top face commutes.
\end{proof}

Recall that $\wh{Y}_0$ is the path component of $\wh{y}_0$ in $\wh{Y}$. Let $\wh{Z}_0=\wh{f}(\wh{Y}_0)$. Since $\wh{Z}_0$ is path connected and $\wh{g}(\wh{z}_0)\in \wh{Y}_0$, we have $\wh{g}(\wh{Z}_0)\subseteq \wh{Y}_0$. Thus the restrictions $\wh{f}_0:\wh{Y}_0\to \wh{Z}_0$ of $\wh{f}$ and $\wh{g}_0:\wh{Z}_0\to \wh{Y}_0$ of $\wh{g}$ are well-defined maps. A similar argument gives restricted homotopies $\wh{H}_0:\wh{Y}_0\times\ui\to\wh{Y}_0$ and $\wh{G}_0:\wh{Z}_0\times\ui\to\wh{Z}_0$. Hence, we have the following corollary.

\begin{corollary} \label{zeromapscorollary}
The restricted maps $\wh{f}_0:\wh{Y}_0\to \wh{Z}_0$ and $\wh{g}_0:\wh{Z}_0\to \wh{Y}_0$ are based homotopy inverses. In particular, $\wh{H}_0$ is a based homotopy from $id_{\wh{Y}_0}$ to $\wh{g}_0\circ \wh{f}_0$ and $\wh{G}_0$ is a based homotopy from $id_{\wh{Z}_0}$ to $\wh{f}_0\circ \wh{g}_0$. Moreover, these homotopies make the following square commute.
\[\xymatrix{
\wh{Y}_0 \times \ui \ar[r]^-{\wh{H}_0} \ar[d]_-{\wh{f}_0\times id} & \wh{Y}_0 \ar[d]^-{\wh{f}_0} \\
\wh{Z}_0\times \ui \ar[r]_-{\wh{G}_0} & \wh{Z}_0
}\]
\end{corollary}

Just as with $\wh{Y}_0$, the space $\wh{Z}_0$ need not be locally path connected. In the next section we give a detailed account of the structure of $\wt{Y}$ so that we may construct a locally path connected counterpart $\tZ$ for $\wh{Z}_0$

\section{The spaces $\tY$ and $\tZ$}

\subsection{The topological structure of $\tY$}

In this section, we provide a description of $\tY$ similar to that of $\tY_j$, namely a tree-like decomposition into copies of $\tY_j$. Just as the copies of $\tY_j$ appear in $\tY_{\leq k}$ according to the reduced words in $\pi_1(\tY_{\leq k})$, the copies of $\tY_j$ in $\tY$ will be arranged according to the infinite word structure of $\varoast_{j}\pi_1(Y_j)$. The main difference between these two situations is that words in $\pi_1(Y)$ may be indexed by an infinite linear order and thus copies of $\tY_j$ will appear in a corresponding manner.

\begin{definition}
Fix $j\in \bbn$ and let $\alpha:\ui\to Y$ be a reduced loop based at $y_0$. We say that $\alpha$ is \textit{non-$Y_{j}$-terminal} if either the linear order $\ov{\alpha}$ does not have a maximal element or if $(a,b)\in \ov{\alpha}$ is maximal and $\alpha|_{[a,b]}$ is a loop in $\bigcup_{i\neq j}Y_i$. For each $j\in \bbn$, let $\ntij\subseteq \pi_1(Y)$ denote the subset of homotopy classes of non-$Y_{j}$-terminal reduced loops.
\end{definition}

\begin{remark}
Just like the finite case, the set $\ntij\subseteq \pi_1(Y)$ provides a canonical choice of representatives for the coset space $\pi_1(Y)/\pi_1(Y_{j})$. Indeed, the projection $\pi_1(Y)\to \pi_1(Y)/\pi_1(Y_{j})$ restricts to a bijection $\ntij\to \pi_1(Y)/\pi_1(Y_{j})$.

Moreover, $p^{-1}(y_0)=\pi_1(Y)=\bigcup_{j\in\bbn}\ntij$ since for every $\alpha\in\pi_1(Y)$, the corresponding reduced word $w_{\alpha}$ either has no terminal letter or does not terminate in a letter of $\pi_1(Y_j)$ for all but one $j$.
\end{remark}

Define the following subsets of $\wt{Y}$ for each $\alpha\in\ntij$. 
\begin{itemize}
\item $\mcb_{\infty,j,\alpha}=\{\alpha\beta\in \tY\mid \beta\in\tY_j\}$,
\item $\mcu_{\infty,j,\alpha}=\{\alpha\beta\in\tY\mid \beta\in\tY_j\backslash\pi_1(Y_j)\}$,
\item $\mca_{\infty,j,\alpha}=\{\alpha\beta\tau_j\in \tY\mid \beta\in\tY_j\}$.
\end{itemize}

By definition, we have $\mca_{\infty,j,\alpha}\subseteq \mcu_{\infty,j,\alpha}\subseteq \mcb_{\infty,j,\alpha}$.

\begin{proposition}\label{disjointunion}
Fix $j\in\bbn$. Then
\begin{enumerate}
\item $p^{-1}(Y_j)$ is the disjoint union of the sets $\mcb_{\infty,j,\alpha}$, $\alpha\in \ntij$.
\item $p^{-1}(Y_j\backslash\{y_0\})$ is the disjoint union of the sets $\mcu_{\infty,j,\alpha}$, $\alpha\in \ntij$,
\item $p^{-1}(X_j)$ is the disjoint union of the sets $\mca_{\infty,j,\alpha}$, $\alpha\in \ntij$.
\end{enumerate}
\end{proposition}

\begin{proof}
$\bigcup_{\alpha\in\ntij}\mcb_{\infty,j,\alpha}\subseteq p^{-1}(Y_j)$ is clear. Suppose $\gamma\in p^{-1}(Y_j)$. If $\gamma\in\ntij$, then $\gamma\in\mcb_{\infty,j,\gamma}$. If $\gamma\notin\ntij$, then we may write $\gamma=\alpha\beta$ where $\alpha\in \ntij$ and $\beta\in\wt{Y}_j$, giving $\gamma\in\mcb_{\infty,j,\alpha}$. This proves $p^{-1}(Y_j)=\bigcup\{\mcb_{\infty,j,\alpha}\mid \alpha\in\ntij\}$. 

To show the sets $\mcb_{\infty,j,\alpha}$, $\alpha\in\ntij$ are disjoint, suppose $\gamma\in\mcb_{\infty,j,\alpha}\cap\mcb_{\infty,j,\alpha '}$ for $\alpha,\alpha '\in \ntij$. Write $\gamma=\alpha\beta=\alpha '\beta '$ for paths $\beta,\beta '\in\tY_j\backslash \pi_1(Y_j)$. Now $\beta ''=\beta(\beta ')^{-1}\in \pi_1(Y_j)$ corresponds to a letter in the group of reduced words $\varoast_{j}\pi_1(Y_j)$. Now $\delta=\alpha\beta '' (\alpha ')^{-1}=1$ in $\pi_1(Y)$ and so $[w_{\delta}]=e$ in $\varoast_{j}\pi_1(Y_j)$. However, this means that the unreduced word $w_{\alpha}w_{\beta ''}w_{\alpha '}^{-1}$ is equivalent to the empty word. Because $w_{\alpha}$ and $w_{\alpha '}^{-1}$ are already reduced, this is only possible if $[w_{\beta ''}]=e$, i.e. if $\beta ''=1$ in $\pi_1(Y_j)$. It follows that $\alpha =\alpha '$.

(1) provides the non-trivial parts of the arguments for (2) and (3). The remainder of these proofs are straightforward.
\end{proof}

Because $Y$ is locally contractible at each point of $Y\backslash\{y_0\}$, a straightforward argument gives the next proposition.

\begin{proposition}\label{uprop}
For each $\alpha\in\ntij$,
\begin{enumerate}
\item $\mcu_{\infty,j,\alpha}$ is open in $\tY$,
\item $\mcu_{\infty,j,\alpha}$ is a path component of $p^{-1}(Y_j\backslash\{y_j\})$,
\item $p|_{\mcu_{\infty,j,\alpha}}:\mcu_{\infty,j,\alpha}\to Y_j\backslash\{y_j\}$ is a surjective local homeomorphism.
\end{enumerate}
\end{proposition}

\begin{proposition}
For each $\alpha\in\ntij$, $\mcb_{\infty,j,\alpha}$ is the closure of $\mcu_{\infty,j,\alpha}$ in $\tY$.
\end{proposition}

\begin{proof}
Note that $\mcb_{\infty,j,\alpha}$ consists of $\mcu_{\infty,j,\alpha}$ and the elements $\alpha\beta$, $\beta\in \pi_1(Y_j)$. Fixing $\beta\in \pi_1(Y_j)$, let $V$ be a path-connected neighborhood of $y_0$ in $Y$ so that $N(\alpha\beta,V)$ is a basic and path-connected neighborhood of $\alpha\beta$ in $\tY$. Since we have assumed from the start that $\pi_1(Y_j)\neq 1$, $Y_j$ consists of more than a point and so there exists a path $\eta:(\ui,0)\to (V\cap Y_j,y_j)$ with $\eta(1)\neq y_0$. Writing $\gamma=[\eta]$, we have $\alpha\beta\gamma\in N(\alpha\beta,V)\cap \mcu_{\infty,j,\alpha}$. This proves $\mcb_{\infty,j,\alpha}\subseteq \ov{\mcu_{\infty,j,\alpha}}$. Since $\mcu_{\infty,j,\alpha}\subseteq \mcb_{\infty,j,\alpha}$, it suffices to prove that $\mcb_{\infty,j,\alpha}$ is closed in $\tY$. Since $Y_j$ is closed in $Y$, $p^{-1}(Y_j)$ is closed. Hence, it is enough to show that $\mcb_{\infty,j,\alpha}$ is closed in $p^{-1}(Y_j)$.

Let $\gamma\in p^{-1}(Y_j)\backslash \mcb_{\infty,j,\alpha}$. Proposition \ref{disjointunion} gives that $\gamma\in \mcb_{\infty,j,\alpha '}$ where $\alpha\neq \alpha '\in\ntij$. Write $\gamma=\alpha ' \beta '$ for $\beta '\in\tY_j$.

 \textbf{Case I:} If $\beta '$ does not represent a loop, i.e. $\beta '\notin\pi_1(Y_j)$, then by (1) of the previous proposition, $\mcu_{\infty,j,\alpha '}$ is an open neighborhood of $\gamma$ disjoint from $\mcb_{\infty,j,\alpha }$.

\textbf{Case II:} If $\beta '\in \pi_1(Y_j)$, we may use the fact that $\tY$ is Hausdorff to find a path-connected neighborhood $V$ of $y_0$ in $Y$ such that $\alpha \notin N(\alpha '\beta ',V)$. Moreover, since $j$ is fixed, we may choose $V$ small enough so that $V\cap Y_j$ is homeomorphic to a half-open arc in $Y_j\backslash X_j$. We claim that $N(\alpha '\beta ',V)\cap \mcb_{\infty,j,\alpha}=\emptyset$. Suppose, to obtain a contradiction, that $\alpha\beta=\alpha ' \beta ' \delta $ for $\beta\in\tY_j$ and $\delta\in\wt{V}$ (where $V$ has basepoint $y_0$). Note that $\delta (1)$ lies in the half-open arc $Y_j\backslash X_j$. Let $\epsilon$ be the homotopy class of the path $\ui\to Y_j$, which parameterize the arc from $\beta(1)=\delta (1)$ to $y_j=y_0$. By our choice of $V$, we must have $\delta\epsilon\in \pi_1(\bigcup_{i\neq j}Y_i,y_0)$, that is, $w_{\delta\epsilon}$ contains no letters from $\pi_1(Y_j)$. 

Recall that $\alpha,\alpha '$ are being represented by reduced loops such that the corresponding reduced words $w_{\alpha}$ and $w_{\alpha '}$ do not end in a letter from $\pi_1(Y_j)$. If $\beta \epsilon=1$ in $\pi_1(Y_j)$, then $\alpha=\alpha '\beta '\delta  \epsilon$, however, this contradicts $\alpha \notin N(\alpha '\beta ',V)$. If $1\neq \beta \epsilon\in  \pi_1(Y_j)$, then we have $\alpha\beta \epsilon=\alpha '\beta '\delta  \epsilon$ in $\pi_1(Y)$ where $\beta \epsilon\in \pi_1(Y_j)$. Now, $w_{\alpha \beta\epsilon}$ consists of the reduced word $w_{\alpha}$, which does not end in a letter from $\pi_1(Y_j)$, followed by the non-trivial letter $(\beta \epsilon)\in\pi_1(Y_j)$. Thus $w_{\alpha\beta\epsilon}$ is a reduced representative of $w_{\alpha '}\beta ' w_{\delta \epsilon}$. However, $w_{\alpha '}\beta '$ is reduced and $w_{\delta \epsilon} $ has no letters from $\pi_1(Y_j)$. The only way for $w_{\alpha '}\beta ' w_{\delta \epsilon}$ to reduce to $w_{\alpha}(\beta\epsilon)$ is if $[w_{\delta \epsilon}]=e$, that is, if $\delta\epsilon=1$. However, this would give $\alpha(\beta \epsilon)=\alpha '\beta '$. Since $w_{\alpha}(\beta \epsilon)=w_{\alpha '}\beta '$ as reduced words, the terminal letters must be equal, i.e. $\beta\epsilon=\beta '$. This implies $\alpha=\alpha '$; a contradiction.
\end{proof}

\begin{proposition}\label{bprop}
For each $\alpha\in\ntij$, $\mcb_{\infty,j,\alpha}$ is a path component of $p^{-1}(Y_j)$. Moreover, $\mcb_{\infty,j,\alpha}$ is locally path connected and simply connected.
\end{proposition}

\begin{proof}
Let $\Gamma:\ui\to p^{-1}(Y_j)$ be a path. By Proposition \ref{disjointunion}, we may assume $\Gamma(0)=\alpha \beta$ for $\alpha\in \ntij$ and $\beta\in\tY_j$. Now $\gamma=p\circ \Gamma:\ui\to Y_j$ starts at $\beta(1)$. Let $\wt{\gamma}_{s}(t)=\alpha\beta[\gamma_t]$ be the standard lift as described in Remark \ref{standardliftremark}. Now it is clear that $\wt{\gamma}_{s}$ is a path in $\mcb_{\infty,j,\alpha}$ and by unique path-lifting, we have $\Gamma=\wt{\gamma}_s$. Moreover, since $Y_j$ is path-connected, standard lifts may be used to show each $\mcb_{\infty,j,\alpha}$ is path-connected. We conclude that the sets $\mcb_{\infty,j,\alpha}$, $\alpha\in\ntij$ are the path components of $p^{-1}(Y_j)$.

Next, we fix $\alpha\in\ntij$ and show $\mcb_{\infty,j,\alpha}$ is locally path connected. It is clear from (3) of Proposition \ref{uprop} that $\mcb_{\infty,j,\alpha}$ is locally path connected at the points of $\mcu_{\infty,j,\alpha}$. Fix $\alpha\beta\in \mcb_{\infty,j,\alpha}$ for $\beta\in \pi_1(Y_j)$. Let $V$ be a neighborhood of $y_0$ in $Y$ such that $V\cap Y_j$ is a half-open arc in $Y_j\backslash X_j$. We will show that $N(\alpha\beta,V)\cap \mcb_{\infty,j,\alpha}$ is path connected. Consider $\alpha\beta\gamma=\alpha \beta '$ for $\gamma\in\wt{V}$ and $\beta '\in\tY_j$. From this, we have $\gamma=\beta^{-1}\beta '$. Therefore, $\gamma$ represents a path in $Y_j\cap V$. In particular, if $\epsilon:\ui\to Y_j\cap V$ parameterizes the arc from $y_0$ to $\gamma(1)$, then the standard lift $\wt{\epsilon}_s(t)=\alpha\beta[ \epsilon_t]$ gives a path in $N(\alpha\beta,V)\cap \mcb_{\infty,j,\alpha}$ from $\alpha\beta$ to $\alpha\beta\gamma$.

Finally, suppose $\ell:\ui\to \mcb_{\infty,j,\alpha}$ is a loop based at $\alpha$. Since $\tY$ is simply connected, $\ell$ is null-homotopic in $\tY$. Since $p\circ \ell$ has image in $Y_j$ and $Y_j$ is a retract of $Y$, $p\circ \ell$ is null-homotopic loop in $Y_j$. Let $K:(\ui^2,\{0,1\}\times I\cup I\times \{1\})\to (Y_j,y_j)$ be a null-homotopy of $p\circ \ell$. Consider the lift $\wt{K}: (\ui^2,\{0,1\}\times I\cup I\times \{1\})\to (\tY,\alpha)$ of $K$. Since $\im(\wt{K})$ is path-connected and contains $\alpha$, $\im(\wt{K})$ must lie in the path component $\mcb_{\infty,j,\alpha}$ of $p^{-1}(Y_j)$ (recall Proposition \ref{uprop}). Unique lifting ensures that $\wt{K}$ is a null-homotopy of $\ell$. We conclude that $\mcb_{\infty,j,\alpha}$ is simply connected.
\end{proof}

\begin{theorem}
For each $\alpha\in\ntij$, the restriction $p|_{\mcb_{\infty,j,\alpha}}:\mcb_{\infty,j,\alpha}\to Y_j$ is a universal covering map.
\end{theorem}

\begin{proof}
The restriction $p|_{\mcb_{\infty,j,\alpha}}$ inherits the required lifting property of a generalized covering map from that of $p$. Proposition \ref{bprop} ensures that $\mcb_{\infty,j,\alpha}$ is path connected, locally path connected, and simply connected. This fulfills all requirements for $p|_{\mcb_{\infty,j,\alpha}}$ to be a generalized universal covering map. It follows from standard covering space theory \cite{Spanier66} that every generalized covering map where the codomain is path connected, locally path connected, and semilocally simply connected is a covering map in the usual sense. Since $Y_j$ meets all of these conditions, $p|_{\mcb_{\infty,j,\alpha}}$ is a universal covering map.
\end{proof}

Recall that $p_j:(\tY_j,\wt{y}_j)\to (Y_j,y_j)$ denotes the universal covering map over $Y_j$ (also with the whisker topology construction).

\begin{corollary}
For each $\alpha\in\ntij$, the map $\Lambda_{\infty,j,\alpha}:\mcb_{\infty,j,\alpha}\to\tY_j$, $\Lambda_{\infty,j,\alpha}(\alpha\beta)=\beta$ is a homeomorphism satisfying $p_j\circ\Lambda_{\infty,j,\alpha}=p|_{\mcb_{\infty,j,\alpha}}$
\end{corollary}

Using the previous results for $\mcb_{\infty,j,\alpha}$ and the fact that $\Lambda_{\infty,j,\alpha}(\mca_{\infty,j,\alpha})$ may be identified with the copy of $\tX_j$ in $\tY$, we also have the following.

\begin{proposition}
For each $\alpha\in\ntij$,
\begin{itemize}
\item $\mca_{\infty,j,\alpha}$ is closed in $\tY$,
\item $\mca_{\infty,j,\alpha}$ is a path component of $p^{-1}(X_j)$,
\item The restriction $p|_{\mca_{\infty,j,\alpha}}:\mca_{\infty,j,\alpha}\to X_j$ is a universal covering map.
\end{itemize}
\end{proposition}

%
%

\begin{definition}
For each $\alpha\in\ntij$ and $\gamma\in \pi_1(Y_j)$, define 
\begin{itemize}
\item $\bfe_{\infty,j,\alpha,\gamma}=\{\alpha\gamma \tau_{j,s}\mid s\in \ui\}$,
\item $\wt{x}_{\infty,j,\alpha,\gamma}=\alpha\gamma\tau$ 
\item $\wt{y}_{\infty,j,\alpha,\gamma}=\alpha\gamma$ 
\end{itemize}
We refer to the points $\wt{y}_{\infty,j,\alpha,\gamma}$ as the \textit{arc-endpoints} of $\mcb_{\infty,j,\alpha}$.
\end{definition}

\begin{remark}
Under the identification of $\Lambda_{\infty,j,\alpha}$, $\mcb_{\infty,j,\alpha}$ consists of $\mca_{\infty,j,\alpha}$ with an arc $\bfe_{\infty,j,\alpha,\gamma}$ attached at $\wt{x}_{\infty,j,\alpha,\gamma}$ for each $\gamma\in\pi_1(Y_j)$. Moreover, the subspace of arc-endpoints $\mcb_{\infty,j,\alpha}\backslash \mcu_{\infty,j,\alpha}=\{\wt{y}_{\infty,j,\alpha,\gamma}\mid \gamma\in\pi_1(Y_j)\}$ is discrete and closed in $\tY$.
\end{remark}

Recall that $\phi_Y:\tY\to \wh{Y}_0$, $\phi_Y(\alpha)=(\varrho_k(\alpha))$ is a continuous bijection, which need not be a homeomorphism. We end this section by showing that $\phi_Y$ is a homeomorphism on the individual subspaces $\mcb_{\infty,j,\alpha}$ of $\tY$.

\begin{theorem}\label{littleimagetheorem}
Fix $j\in\bbn$ and $\alpha\in\ntij$. Write $\varrho_k(\alpha)=\alpha_{k}'\gamma_k$ for $\alpha_{k}'\in\ntkj$ and $\gamma_k\in\pi_1(Y_j)$. The continuous injection $\phi_{Y}:\tY\to \wh{Y}_0$ maps $\mcb_{\infty,j,\alpha}$ homeomorphically onto $\wh{Y}_0\cap \left( \prod_{k\geq j}\mcb_{k,j,\alpha_{k}'}\times \prod_{k=1}^{j-1} \{\alpha_{k}'\}\right)$.
\end{theorem}

\begin{proof}
To simplify notation, let $A_{\infty}=\mcb_{\infty,j,\alpha}$, $A_k=\{\alpha_{k}'\}$ for $1\leq j\leq k-1$ and $A_k=\mcb_{k,j,\alpha_{k}'}$ for $k\geq j$. We wish to show that $\phi_{Y}$ maps $A_{\infty}$ onto $\mathbf{A}=\wh{Y}_0\cap  \left( \prod_{k\geq 1}A_k\right)$. We first check that $\phi_Y(\mcb_{\infty,j,\alpha})\subseteq \mathbf{A}$.

Suppose $\alpha\beta\in A_{\infty}$ for $\beta\in\tY_j$. When $k\geq j$, $\varrho_k(\alpha\beta)=\alpha_{k}'(\gamma_k\beta)$ where $\gamma_k\beta\in\tY_j$. Thus $\varrho_k(\alpha\beta)\in A_k$. When $1\leq k\leq j-1$, we have $\varrho_k(\alpha\beta)=\alpha_{k}'$. Thus $\phi_{Y}(\alpha\beta)\in \mathbf{A}$, proving the desired inclusion.

When $k\geq j$, the restricted map $\varrho_k: A_{\infty}\to A_k$ is given by $\varrho_k(\alpha\beta)=\alpha_{k}'(\gamma_k\beta)$ and thus the following diagram commutes.
\[\xymatrix{
A_{\infty} \ar[r]^-{(\varrho_k)|_{A_{\infty}}} \ar[d]_-{\Lambda_{\infty,j,\alpha}} & A_k\ar[d]^-{\Lambda_{k,j,\alpha_{k}'}} \\
\tY_j \ar[r]_-{\Delta_{\gamma_k}}  &\tY_j
}\]
Since the veritical and bottom maps are homeomorphisms, so is the top map.  Then we have a sub-inverse system $(A_k,(\wt{r}_{k+1,k})|_{A_{k+1}})$ of $(\tY_{\leq k},\wt{r}_{k+1,k})$ with limit $\mathbf{A}=\varprojlim_{k}A_k$. Since the maps $(\wt{r}_{k+1,k})|_{A_{k+1}}$ are homeomorphisms for $k\geq j$, the projection maps $(\wt{r}_k)|_{\mathbf{A}}:\mathbf{A}\to A_k$ are homeomorphisms for $k\geq j$. Thus for all $k\geq j$, the maps $(\wt{r}_k)|_{\mathbf{A}}^{-1}\circ(\varrho_k)|_{A_{\infty}}:A_{\infty}\to \mathbf{A}$ are homeomorphisms. Since $\wt{r}_k\circ\phi_Y=\varrho_{k}$, we have $\phi_{Y}|_{A_{\infty}}=(\wt{r}_k)|_{\mathbf{A}}^{-1}\circ(\varrho_k)|_{A_{\infty}}$.
\end{proof}

\subsection{Stabilization of trees and the quotient map $\wt{f}:\tY\to\tZ$} 

In the proof of Theorem \ref{littleimagetheorem}, we observed that the ``bookkeeping" homeomorphisms $\Lambda_{\infty,j,\alpha}$ allowed us to show that the maps $\varrho_k$, $k\geq j$ send each $\mcb_{\infty,j,\alpha}$ homeomorphically to $\mcb_{k,j,\alpha_k'}$ for some $\alpha_k'\in\ntkj$ (when $k\geq j$). The trees $\mct_{k,j,\alpha_k'}$ then correspond to a sequence of trees $\Lambda_{k,j,\alpha_k'}(\mct_{k,j,\alpha_k'})=\beta_kT_j$ in $\tY_j$. In order to make a unique choice of tree $\mct_{\infty,j,\alpha}$ in $\mcb_{\infty,j,\alpha}$ that is coherent with the trees $\mct_{k,j,\alpha_k'}$, we need for the sequence $\{\beta_k\}$ in $\pi_1(Y_j)$ to stabilize. This stabilization is established in the next lemma.

\begin{lemma}[Stabilization]\label{stabilizationlemma}
Let $1\leq j<\infty$, $\alpha\in\ntij$, and set $\alpha_k=\varrho_k(\alpha)=\alpha_k'\gamma_k$ for $\alpha_k'\in\ntkj$ and $\gamma_k\in\pi_1(Y_j)$. If $\Lambda_{k,j,\alpha_k}(\mct_{k,j,\alpha_k'})=\beta_kT_j$ for $\beta_k\in\pi_1(Y_j)$, then sequence $\{\beta_k\}_{k\geq j}$ is eventually constant.
\end{lemma}

\begin{proof} Fix $j\in\bbn$ and $\alpha\in\ntij$. Using the notation established in the statement of the lemma, recall that $\Lambda_{j,j,\alpha_j}(\mct_{j,j,\alpha_j})=T_j$ and thus $\beta_j=1$. Moreover, according to the inductive definition of the trees $\mct_{-,-,-}$ and the summary in Remark \ref{stabilizationremark}, we have $\beta_{k+1}=\gamma_{k}^{-1}\beta_{k}$ for all $k\geq j$. Suppose, to obtain a contradiction, that there exists $j\leq k_1<k_2<k_3<\cdots$ such that $\beta_{k_{i+1}}\neq \beta_{k_i}$. Then $1\neq \gamma_{k_i}\in \pi_1(Y_j)$ for all $i\in\bbn$. It follows that the reduced word $w_{\alpha_{k_i}}$ corresponding to $\alpha_{k_i}=\alpha_{k_i}'\gamma_{k_i}$ in the free product $\pi_1(Y_{\leq k_i})$ terminates in a non-trivial letter from $\pi_1(Y_j)$. 

However, the word $w_{\alpha}\in \varoast_{j}\pi_1(Y_j)$ corresponding to $\alpha\in \ntkj$ does not terminate in a letter from $\pi_1(Y_j)$. Since some of the projection words $w_{\alpha_{k}'}$ contain elements of $\pi_1(Y_j)$, $w_{\alpha}$ must contain some letters from $\pi_1(Y_j)$. However, $w_{\alpha}$ only contains finitely many letters from $\pi_1(Y_j)$ and so we may write $w_{\alpha}=w_{\alpha '}\ell w_{\eta}$ where $\ell\in \pi_1(Y_j)$ is the last appearance of a letter from $\pi_1(Y_j)$ in $w_{\alpha}$ and $w_{\eta}\neq 1$ has no letters from $\pi_1(Y_j)$. Find $K$ large enough so that $\alpha$ and $\alpha_K=\varrho_{K}(\alpha)$ have the same number of letters from $\pi_1(Y_j)$ and such that $\varrho_{K}(\eta)\neq 1$. Find $i\in\bbn$ with $k_i>K$. Then $\varrho_{k_i}(\alpha)$ has corresponding reduced word $\varrho_{k_i}(\alpha ')\ell \varrho_{k_i}(\eta)$ where $\varrho_{k_i}(\eta)$ is a non-trivial word in $\pi_1(Y_{\leq k_i})$ with no letters from $\pi_1(Y_j)$. This contradicts the fact that $\alpha_{k_i}$ ends in a letter from $\pi_1(Y_j)$.
\end{proof}

\begin{definition}
Fix $1\leq j<\infty$, $\alpha\in\ntij$, $\alpha_k$, and $\beta_k$ as in Lemma \ref{stabilizationlemma}. We define the subspace $\mct_{\infty,j,\alpha}$ of $\mcb_{\infty,j,\alpha}$ to be the tree $\Lambda_{\infty,j,\alpha}^{-1}(\beta_{\alpha} T_j)$ where $\beta_{\alpha}\in\pi_1(Y_j)$ is the eventual value of the sequence $\{\beta_k\}$.
\end{definition}

Let $\tZ$ be the quotient space of $\tY$ where each subspace $\mct_{\infty,j,\alpha}$ is collapsed to a point. Let $\tf:\tY\to \tZ$ be the quotient map and $z_0=\wt{f}(\wt{y}_0)$. Since $\tY$ is path connected and locally path connected so is $\tZ$. We characterize $\wt{f}$ on the subspaces of $\tY$ in the same way that we did for $\wt{f}_k$. Many of the proofs are diagram chases similar to those earlier and so we will omit some details.

\begin{definition}
For each $j\in\bbn$ and $\alpha\in \ntij$, set $\mcd_{\infty,j,\alpha}=\wt{f}(\mcb_{\infty,j,\alpha})$ and let $\wt{f}_{\infty,j,\alpha}:\mcb_{\infty,j,\alpha}\to \mcd_{\infty,j,\alpha}$ denote the restricted quotient map of $\wt{f}$.
\end{definition}

\begin{proposition}\label{littlelambdaprop2}
Fix $j\in\bbn$ and $\alpha\in \ntij$ and suppose $\Lambda_{\infty,j,\alpha}(\mct_{\infty,j,\alpha})=\beta_{\alpha} T_j$ for $\beta_{\alpha}\in\pi_1(Y_j)$. Then there is a canonical homeomorphism $\lambda_{\infty,j,\alpha}:\mcd_{\infty,j,\alpha}\to D_{j,\beta_{\alpha} }$ that makes the following square commute.
\[\xymatrix{
\mcb_{\infty,j,\alpha} \ar[r]^-{\Lambda_{\infty,j,\alpha}} \ar[d]_-{\wt{f}_{\infty,j,\alpha}} & \tY_j \ar[d]^-{f_{j,\beta_{\alpha}}} \\ \mcd_{\infty,j,\alpha} \ar@{-->}[r]_-{\lambda_{\infty,j,\alpha}} & D_{j,\beta_{\alpha} }
}\]
\end{proposition}

\begin{proposition}\label{rksquareprop}
Fix $j\in\bbn$, $\alpha\in\ntij$, and $k\geq j$. If $\varrho_{k}(\alpha)=\alpha_{k} '\gamma_{k}$ for $\alpha_{k}'\in \ntkj$ and $\gamma_{k}\in\pi_1(Y_j)$, then $\varrho_k(\mct_{\infty,j,\alpha})=\mct_{\infty,j,\alpha_k'}$.
\end{proposition}

\begin{proof}
Recall from the proof of Theorem \ref{littleimagetheorem} that the following diagram of homeomorphisms commutes.
\[\xymatrix{
\tY_j \ar[r]^-{\Delta_{\gamma_k}}  & \tY_j\\
\mcb_{\infty,j,\alpha} \ar[u]^-{\Lambda_{\infty,j,\alpha}} \ar[r]_-{\varrho_{k}} & \mcb_{k,j,\alpha_{k}'} \ar[u]_-{\Lambda_{k,j,\alpha_{k}'}}
}\]Suppose $\Lambda_{k,j,\alpha_k'}(\mct_{k,j,\alpha_k'})=\beta_kT_j$ for $\beta_k\in\pi_1(Y_j)$. By Lemma \ref{stabilizationlemma}, there exists $\beta_{\alpha}\in\pi_1(Y_j)$ and $K\in\bbn$ such that $\beta_k=\beta_{\alpha}$ and $\gamma_k=1$ for all $k\geq K$. Moreover, we defined $\mct_{\infty,j,\alpha}$ so that $\Lambda_{\infty,j,\alpha}(\mct_{\infty,j,\alpha})=\beta_{\alpha}T_j$. When $k\geq K$, we have 
\begin{eqnarray*}
\varrho_k(\mct_{\infty,j,\alpha}) &=& \Lambda_{k,j,\alpha_{k}'}^{-1}(\Lambda_{\infty,j,\alpha}(\mct_{\infty,j,\alpha}))\\
&=& \Lambda_{k,j,\alpha_{k}'}^{-1}(\beta_{\alpha}T_j)\\
&=& \Lambda_{k,j,\alpha_{k}'}^{-1}(\beta_{k}T_j)\\
&=& \mct_{k,j,\alpha_k'}
\end{eqnarray*}
The cases $1\leq k<K$ follow directly from the case $k=K$ and the equalities $\wt{r}_{k+1,k}(\mct_{k+1,j,\alpha_{k+1}'})=\mct_{k,j,\alpha_k'}$ and $\wt{r}_{k+1,k}\circ \varrho_{k+1}=\varrho_k$ .
\end{proof}
\begin{lemma}\label{unionlemma}
For each $k\in\bbn$, we have \[\bigcup_{j\in\bbn}\bigcup_{\alpha\in \ntij}\mct_{\infty,j,\alpha}=\varrho_{k}^{-1}\left(\bigcup_{1\leq j\leq k}\bigcup_{\beta\in \ntkj}\mct_{\infty,j,\beta}\right).\]
\end{lemma}

\begin{proof}
The inclusion $\subseteq $ follows from Proposition \ref{rksquareprop}. Suppose $a\in \tY\backslash  \bigcup_{j\in\bbn}\bigcup_{\alpha\in \ntij}\mct_{\infty,j,\alpha}$. Since $\tY$ is the union of the subspaces $\mcb_{\infty,j,\alpha}$, we have $a\in\mcb_{\infty,j,\alpha}\backslash \mct_{\infty,j,\alpha}$ for some $j\in\bbn$ and $\alpha\in\ntij$. Write $\varrho_k(\alpha)=\alpha_{k}'\gamma_{k}$ for $\alpha_{k}'\in\ntkj$ and $\gamma_{k}\in\tY_j$. By Proposition \ref{rksquareprop}, $\varrho_k$ maps $\mcb_{\infty,j,\alpha}$ homeomorphically onto $\mcb_{k,j,\alpha_{k}'}$ and $\varrho_k(\mct_{\infty,j,\alpha})= \mct_{k,j,\alpha_{k}'}$ Thus $\varrho_k(a)\in \mcb_{k,j,\alpha_{k}'}\backslash \mct_{k,j,\alpha_{k}'}$. It follows that $a\notin \varrho_{k}^{-1}(\mct_{\infty,j,\beta})$ for any $1\leq j\leq k$ and $\beta\in\ntkj$.
\end{proof}

\begin{theorem}\label{psitheorem}
There is a canonical, continuous bijection $\psi:\tZ\to \wh{Z}_{0}$ such that the following square commutes.
\[\xymatrix{
\tY \ar[d]_-{\wt{f}} \ar[r]^-{\phi_Y} & \wh{Y}_{0} \ar[d]^-{\wh{f}_{0}}\\
\tZ \ar[r]_-{\psi} & \wh{Z}_0
}\]
\end{theorem}

\begin{proof}
To show that $\psi$ is well-defined, we must show that $\wh{f}_{0}\circ \phi_Y$ is constant on each tree $\mct_{\infty,j,\alpha}$, $j\in\bbn$, $\alpha\in\ntij$. Fixing such $j$ and $\alpha$, define $\alpha_k=\varrho_k(\alpha)$, and $\beta_k$, $k\geq j$ as in Lemma \ref{stabilizationlemma}. As in the proof of that lemma, set $\alpha_k=\alpha_{k} '\gamma_{k}$ for $\alpha_{k}'\in\ntkj$ and $\gamma_{k}\in\pi_1(Y_j)$ so that $\wt{r}_{k+1,k}(\mct_{k+1,k,\alpha_{k+1}'})=\mct_{k,k,\alpha_{k}'}$. Recall that $\beta_{k+1}=\gamma_{k}^{-1}\beta_{k}$ for all $k\geq j$. The conclusion of Lemma \ref{stabilizationlemma} is that there exists $K\geq j$ such that $\beta_k=\beta_{\alpha}$ for all $k\geq K$. Thus, for all $k\geq K$, we have $\gamma_{k}=1$ and thus $\alpha_k=\alpha_{k}'\in \ntkj$. It follows that $\wt{r}_{k+1,k}$ maps $\mct_{k+1,j,\alpha_{k+1}}$ homeomorphically onto $\mct_{k,j,\alpha_k}$ whenever $k\geq K$.

Recall that $\varrho_k$ maps $\mct_{\infty,j,\alpha}$ homeomorphically onto $\mct_{\infty,j,\alpha_k '}$ for all $k$. Hence, for all $k\geq K$, $\varrho_k$ maps $\mct_{\infty,j,\alpha}$ homeomorphically onto $\mct_{k,j,\alpha_k}$. Additionally, we have $\wt{f}_k(\mct_{k,j,\alpha_k'})=c_{k,j,\alpha_k'}\in \tZ_k$. Thus, whenever $k+1\geq k\geq K$, we have 
\begin{eqnarray*}
\wt{s}_{k+1,k}(c_{k+1,j,\alpha_{k+1}'}) &=& \wt{s}_{k+1,k}(c_{k+1,j,\alpha_{k+1}})\\
&=& \wt{s}_{k+1,k}(\wt{f}_{k+1}(\mct_{k+1,j,\alpha_{k+1}'}))\\
&=& \wt{f}_{k}(\wt{r}_{k+1,k}(\mct_{k+1,j,\alpha_{k+1}'}))\\
&=& \wt{f}_{k}(\mct_{k,j,\alpha_k'})\\
&=& c_{k,j,\alpha_k'}\\
&=& c_{k,j,\alpha_k}
\end{eqnarray*}
Representing $\wh{Y}$ as $\varprojlim_{k\geq K}\wt{Y}_{\leq k}$ and $\wh{f}=\varprojlim_{k\geq K}\wt{f}_k$, we have 
\[
\wh{f}\circ \phi_{Y}(\mct_{\infty,j,\alpha})= \wh{f}\left(\prod_{k\geq K}\mct_{k,j,\alpha_k}\right)=\{(c_{k,j,\alpha_k})_{k\geq K}\},\]
which is a coherent sequence and thus represents a point in $\wh{Z}$. This verifies that $\psi$ is well-defined. Since $\wt{f}$ is a quotient map and $\wh{f}\circ \phi_Y$ is continuous, $\psi$ is continuous. Since $\wh{Z}_0=\wh{f}(\wh{Y}_0)$ by definition and $\phi_{Y}$ is bijection, $\wh{f}\circ \phi_Y$ is surjective. It follows that $\psi$ is surjective.

Finally, we check that $\psi$ is injective. It is enough to check that $\wt{f}$ is constant on the fibers of $\wh{f}_0\circ \phi_{Y}$. Suppose $a \neq b$ in $\tY$ and $\wh{f}_0\circ \phi_{Y}(a)=\wh{f}_0\circ \phi_{Y}(b)$. Write $a_k=\varrho_k(a)=a_{k}'c_{k}$ and $b_k=\varrho_k(b)=b_{k}'d_{k}$ for $a_{k}',b_{k}'\in\ntkj$ and $c_k,d_k\in \wt{Y}_j$. Since $a\neq b$ and $\phi_{Y}$ is injective, there exists $K\in\bbn$ such that $a_k\neq b_k$ for all $k\geq K$. However, $\wt{f}_k( \varrho_{k}(a))=\wt{f}_k( \varrho_{k}(b))$ for all $k\geq 1$ and $\wt{f}_k$ only identifies points in trees of the form $\mct_{k,j,\alpha}$. Thus, for all $k\geq K$, we have $\{\varrho_k(a),\varrho_{k}(b)\}\subseteq \mct_{k,j_k,\alpha_k'}$ for some $j_k\in\bbn$ and $\alpha_k'\in\ntkj$. Specifically, we must have $a_{k}'=b_{k}'=\alpha_{k}'$ for $k\geq K$. 

By Lemma \ref{unionlemma}, we have $a\in \mct_{\infty,j,\alpha}$ and $b\in \mct_{\infty,j',\beta}$ for some $j,j'\in\bbn$ and $\alpha,\beta\in\ntij$. For $k\geq K$, $\varrho_k$ maps $\mct_{\infty,j,\alpha}$ and $\mct_{\infty,j',\beta}$ to $\mct_{k,j_k,\alpha_k'}$. Thus $j=j'=j_k$ for all $k\geq K$. Additionally, we must have $\varrho_k(\alpha)=\alpha_{k}'\gamma_{k}$ and $\varrho_k(\beta)=\alpha_{k}'\delta_{k}$ for $\gamma_k,\delta_k\in\pi_1(Y_j)$. However, Lemma \ref{stabilizationlemma} ensures that there exists $M\geq K$ such that $\gamma_k=\delta_k=1$ for all $k\geq M$. Thus $\varrho_k(\alpha)=\varrho_k(\beta)$ for all $k\geq M$. The injectivity of $\phi_{Y}$ then gives $\alpha=\beta$. Since $a,b$ both lie in $\mct_{\infty,j,\alpha}$, we have $\wt{f}(a)=\wt{f}(b)$.
\end{proof}

\begin{definition}
For each $k\in\bbn$, let $\sigma_{k}:Z\to Z_k$ be the composition $\sigma_k=\wt{s}_k\circ \psi$ so that the following that the following diagram commutes for all $k\in\bbn$. \[\xymatrix{
\tY  \ar[d]_-{\wt{f}} \ar@/^1pc/[rr]^-{\varrho_k} \ar[r]_-{\phi_{Y}} & \wh{Y} \ar[d]^-{\wh{f}} \ar[r]_-{\wt{r}_k} & \wt{Y}_{\leq k} \ar[d]^-{\wt{f}_k}\\
Z \ar@/_1pc/[rr]_-{\sigma_k} \ar[r]^-{\psi} & \wh{Z} \ar[r]^-{\wt{s}_k} & Z_k
}\]
\end{definition}

The proof of the next proposition follows directly from established constructions.

\begin{proposition}\label{projectiondiagramprop}
Fix $j\in\bbn$, $\alpha\in\ntij$, and $k\geq j$. If $\varrho_{k}(\alpha)=\alpha_{k} '\gamma_{k}$ for $\alpha_{k}'\in \ntkj$, and $\gamma_{k}\in\pi_1(Y_j)$, then $\sigma_k$ maps $\mcd_{\infty,j,\alpha}$ homeomorphically onto $\mcd_{k,j,\alpha_{k}'}$ and the following square commutes.
\[\xymatrix{
\mcb_{\infty,j,\alpha} \ar[r]^-{\varrho_k} \ar[d]_-{\wt{f}_{\infty,j,\alpha}} & \mcb_{k,j,\alpha_{k}'} \ar[d]^-{f_{k,j,\alpha_{k}'}}    \\
 \mcd_{\infty,j,\alpha} \ar[r]_-{\sigma_k} & \mcd_{k,j,\alpha_{k}'} 
}\]
\end{proposition}


\begin{lemma}\label{zimagelemma}
Fix $j\in\bbn$, $\alpha\in\ntij$ and suppose that for all $k\geq j$, we have $\varrho_{k}(\alpha)=\alpha_{k} '\gamma_{k}$ for $\alpha_{k}'\in \ntkj$ and $\gamma_{k}\in\pi_1(Y_j)$. Then $\psi$ maps $\mcd_{\infty,j,\alpha}$ homeomorphically onto $\wh{Z}\cap \left(\prod_{k\geq j}\mcd_{k,j,\alpha_{k}'}\times \prod_{k=1}^{j-1}\{\wt{f}_k(\alpha_{k}')\}\right)$.
\end{lemma}

\begin{proof}
Consider the following commutative diagram where the right map is the restriction of $\wh{f}_0$.
\[\xymatrix{
\mcb_{\infty,j,\alpha} \ar[rr]^-{\phi_{Y}} \ar[d]_-{\wt{f}_{\infty,j,\alpha}} && \wh{Y}_0\cap \left(\prod_{k\geq j}\mcb_{k,j,\alpha_{k}'}\times \prod_{k=1}^{j-1}\{\alpha_{k}'\}\right) \ar[d]^-{\wh{f}_{0}}\\
\mcd_{\infty,j,\alpha} \ar[rr]_-{\psi} & & \wh{Z}_{0}\cap \left(\prod_{k\geq j}\mcd_{k,j,\alpha_{k}'}\times \prod_{k=1}^{j-1}\{\wt{f}_k(\alpha_{k}')\}\right)
}\]
The left map is quotient and the top map is a homeomorphism by Theorem \ref{littleimagetheorem}. The right map in the diagram may be represented as the inverse limit $\wh{f}_0=\varprojlim_{k\geq j}\tf_{k,j,\alpha_{k}'}$ of quotient maps. Since both sets of bonding map for this inverse system are homeomorphisms, it follows that $\wh{f}_0$ is a quotient map. Since $\psi$ is a bijection by Theorem \ref{psitheorem}, we conclude that the restriction of $\psi$ in the diagram is a homeomorphism.
\end{proof}

Recall that the open sets of $\lpc(\wh{Z}_0)$ are the path components of open sets in $\wh{Z}_0$. We will use the fact that $\phi_{Y}:\tY\to \lpc(\wh{Y}_0)$ is a homeomorphism, to prove the analgous fact for $\tZ$. 

\begin{theorem}\label{lpctheorem}
$\psi:\tZ\to \lpc(\wh{Z}_0)$ is a homeomorphism.
\end{theorem}

\begin{proof}
Just as we defined the open sets $\mcu_{\infty,j,\alpha}$ in $\tY$, define $\mcu_{k,j,\alpha}=\mcb_{k,j,\alpha}\backslash\{\alpha\beta\in\tY_{\leq k}\mid \beta\in\pi_1(Y_j)\}$ for all $1\leq j\leq k<\infty$ and $\alpha\in\ntkj$. Set $\mcv_{k,j,\alpha}=\wt{f}_{k}(\mcu_{k,j,\alpha})$ even in the case $k=\infty$ and note that $\mcv_{k,j,\alpha}$ is homeomorphic to $D_j$ without its arc-endpoints. Since $\mcu_{k,j,\alpha}$ is saturated with respect to $\wt{f}_{k}$, $\mcv_{k,j,\alpha}$ is open in $\tZ_{k}$. Fix $j\in\bbn$ and $\alpha\in\ntij$. We will show that $\psi(\mcv_{\infty,j,\alpha})$ is open in $\lpc(\wh{Z}_0)$. Consider the open sets
\begin{enumerate}
\item $\scru=\wh{Y}_0\cap \left(  \mcu_{j,j,\varrho_j(\alpha)} \times \prod_{m\neq j}\tY_{\leq m}\right)$ in $\wh{Y}_0$,
\item $\scrv=\wh{Z}_0\cap \left(  \mcv_{j,j,\varrho_j(\alpha)} \times \prod_{m\neq j}Z_m\right)$ in $\wh{Z}_0$.
\end{enumerate}
Because $\wt{f}_j(\mcu_{j,j,\varrho_j(\alpha)})\subset\mcv_{j,j,\varrho_j(\alpha)}$ and $\wt{g}_j(\mcv_{j,j,\varrho_j(\alpha)})\subset\mcu_{j,j,\varrho_j(\alpha)}$, we have $\wh{f}_0(\scru)\subseteq\scrv$ and $\wh{g}_0(\scrv)\subseteq \scru$. 

Note that $\phi_{Y}^{-1}(\scru)$ is the disjoint union of the path-connected open sets $\mcu_{\infty,j,\beta}$, $\beta\in \varrho_{j}^{-1}(\varrho_j(\alpha))$. Because $\phi_{Y}:\tY\to \lpc(\wh{Y}_0)$ is a homeomorphism, the sets $\psi(\mcu_{\infty,j,\beta})$, $\beta\in \varrho_{j}^{-1}(\varrho_j(\alpha))$ are the path components of $\scru$. 

Similarly, $\psi^{-1}(\scrv)$ is the disjoint union of the open sets $\mcv_{\infty,j,\beta}$ for all $\beta\in \varrho_{j}^{-1}(\varrho_j(\alpha))$. Thus $\scrv$ is the disjoint union of the path connected sets $\psi(\mcv_{\infty,j,\beta})$, $\beta\in \varrho_{j}^{-1}(\varrho_j(\alpha))$. If there was a path $\ell:\ui\to \scrv$ with $\ell(0)\in\psi(\mcv_{\infty,j,\beta_0})$ and $\ell(1)\in\psi(\mcv_{\infty,j,\beta_1})$ for $\beta_0\neq \beta_1$ in $\varrho_{j}^{-1}(\varrho_j(\alpha))$, then $\wh{g}_0\circ \ell:\ui\to \scru$ would be a path in $\scru$ from a point in $\phi_{Y}(\mcu_{\infty,j,\beta_0})$ to a point in $\phi_{Y}(\mcu_{\infty,j,\beta_1})$. However, this contradicts the previous paragraph. We conclude that the sets $\psi(\mcv_{\infty,j,\beta})$, $\beta\in \varrho_{j}^{-1}(\varrho_j(\alpha))$ are the path components of $\scrv$. In particular, $\psi(\mcv_{\infty,j,\alpha})$ is open in $\lpc(\wh{Z}_0)$.

Since $\mcv_{\infty,j,\alpha}$ is locally path-connected, the restriction $\psi:\mcv_{\infty,j,\alpha}\to \lpc(\wh{Z}_0)$ is continuous. Lemma \ref{zimagelemma} implies that $\psi$ maps $\mcv_{\infty,j,\alpha}$ homeomorphically onto its image in $\wh{Z}_0$. Therefore, if $V\subseteq \mcv_{\infty,j,\alpha}$ is open, then $\psi(V)$ is open in the subspace $\psi(\mcv_{\infty,j,\alpha})$ of $\wh{Z}_0$. Since $\lpc(\wh{Z}_0)$ has a finer topology than $\wh{Z}_0$, $\psi(V)$ is open in the subspace $\psi(\mcv_{\infty,j,\alpha})$ of $\lpc(\wh{Z}_0)$. Thus $\psi:\mcv_{\infty,j,\alpha}\to \lpc(\wh{Z}_0)$ is an open embedding. We conclude that the restriction of $\psi$ on $\tZ\backslash \wt{f}(p^{-1}(y_0))$ is an open embedding.

To complete the proof that $\psi$ is a homeomorphism, we fix $\alpha\in p^{-1}(y_0)$ and set $z=\wt{f}(\alpha)$. It is enough to show that $\psi$ maps basic neighborhoods of $z$ to open sets in $\lpc(\wh{Z}_0)$. A basic neighborhood of $\alpha$ has the form $N(\alpha,V)$ where $V$ is an open neigborhood of $y_0$ in $Y$. In particular, we may assume there is a $J$ such that $\bigcup_{j>J}Y_j\subseteq V$ and $V\cap Y_j$ is an open arc in $Y_j\backslash X_j$ whenever $1\leq j\leq J$. Note that if $N(\alpha,V)$ meets $\mct_{\infty,j,\alpha}$, then $j>J$ and it follows that $\mct_{\infty,j,\alpha}\subseteq N(\alpha,V)$. Hence, $N(\alpha,V)$ is saturated with respect to $\wt{f}$ and $\wt{f}(N(\alpha,V))$ is a basic open neighborhood of $z$ in $\tZ$. We check that $\psi(\wt{f}(N(\alpha,V)))$ is open in $\lpc(\wh{Z}_0)$.

Note that $V_J=V\cap Y_{\leq J}$ is an open neighborhood of $y_0$ in $Y_{\leq J}$ consisting of a wedge of open arcs. Thus $N(\varrho_J(\alpha),V_J)$ is an open neighborhood of $\varrho_J(\alpha)$ in $\tY_{\leq J}$. In fact, $p_{\leq J}$ maps $N(\varrho_J(\alpha),V_J)$ homeomorphically onto $V_J$ and $N(\varrho_J(\alpha),V_J)$ does not meet any of the trees $\mct_{J,j,\gamma}$ in $\tY_{\leq J}$. Thus $\wt{f}_J$ maps $N(\varrho_J(\alpha),V_J)$ homeomorphically onto the open subset $\wt{f}_J(N(\varrho_J(\alpha),V_J))$ of $\tZ_J$. Recall that we originally constructed $H_j$ to be the constant homotopy on some neighborhood of each arc-endpoint of $\tY_j$. Thus, we may choose the size of the arcs in $V\cap Y_{\leq J}$ to be small enough so that $\wt{g}_J$ maps $\wt{f}_J(N(\varrho_J(\alpha),V_J))$ homeomorphically onto $N(\varrho_J(\alpha),V_J)$. Consider the open sets
\begin{enumerate}
\item $\scru=\wh{Y}_0\cap \left( N(\varrho_J(\alpha),V_J) \times \prod_{k\neq J}\tY_{\leq k}\right)$ in $\wh{Y}_0$,
\item $\scrv=\wh{Z}_0\cap \left(  \wt{f}_J( N(\varrho_J(\alpha),V_J)) \times \prod_{k\neq j}\tZ_k\right)$ in $\wh{Z}_0$.
\end{enumerate}
By our choice of $V$, we have $\wh{f}_0(\scru)\subseteq \scrv$ and $\wh{g}_0(\scrv)\subseteq \scru$. Now $\phi_{Y}^{-1}(\scru)=\varrho_{J}^{-1}(N(\wt{R}_J(\alpha),V_J))$ is a disjoint union of open sets of the form $N(\beta,V)$, $\beta\in \varrho_{J}^{-1}(\varrho_J(\alpha))$. Such sets are path connected. Therefore, the sets $\psi(N(\beta,V))$, $\beta\in \varrho_{J}^{-1}(\varrho_J(\alpha))$ are the path components of $\scru$.

It follows that $\psi^{-1}(\scrv)=\sigma_{J}^{-1}(\wt{f}(N(\varrho_J(\alpha),V_J)))$ is the disjoint union of path connected sets of the form $\wt{f}(N(\beta,V))$, $\beta\in \varrho_{J}^{-1}(\varrho_J(\alpha))$. Then $\scrv$ is the disjoint union of the path-connected sets $\psi(\wt{f}(N(\beta,V)))$, $\beta\in \varrho_{J}^{-1}(\varrho_J(\alpha))$. 

Suppose that there exists a path $\ell:\ui\to \scrv$ with $\ell(0)\in \psi(\wt{f}(N(\beta_0,V)))$ and $\ell(1)\in \psi(\wt{f}(N(\beta_1,V)))$ for $\beta_0\neq \beta_1$ in $\varrho_{J}^{-1}(\varrho_J(\alpha))$. Then $\wh{g}_0\circ \ell:\ui\to \scru$ is a path from a point in $N(\beta_0,V)$ to a point in $N(\beta_1,V)$. However, this is a contradiction of the fact that $N(\beta_0,V)$ and $N(\beta_1,V)$ are distinct path components of $\scru$. We conclude that the sets $\psi(\wt{f}(N(\beta,V)))$, $\beta\in \varrho_{J}^{-1}(\varrho_J(\alpha))$ are the path components of $\scrv$. In particular, $\psi(\wt{f}(N(\alpha,V)))$ is open in $\lpc(\wh{Z}_0)$.
\end{proof}

\begin{corollary}\label{continuitycor}
Suppose $W$ is locally path connected and $h:W\to Z$ is a function. Then the following are equivalent:
\begin{enumerate}
\item $h:W\to Z$ is continuous,
\item $\psi\circ h:W\to \wh{Z}_0$ is continuous,
\item $\sigma_k\circ h:W\to Z_k$ is continuous for all $k\in\bbn$.
\end{enumerate}
\end{corollary}

We will also need the following characterization of $\Lambda_{\infty,j,\alpha}$ and $\lambda_{\infty,j,\alpha}$.

\begin{proposition}\label{lambdaeventualprop}
Fix $j\in\bbn$, $\alpha\in\ntij$, and suppose that $\varrho_{k}(\alpha)=\alpha_{k} '\gamma_{k}$ for $\alpha_{k}'\in \ntkj$ and $\gamma_{k}\in\pi_1(Y_j)$. Then $\Lambda_{\infty,j,\alpha}=\Lambda_{k,j,\alpha_{k}'}\circ \varrho_k$ and $\lambda_{\infty,j,\alpha}=\lambda_{k,j,\alpha_{k}'}\circ \sigma_k$ for all but finitely many $k$.
\end{proposition}

\begin{proof}
We focus on the case $k\geq j$ where we may write $\Lambda_{k,j,\alpha_{k}'}(\mct_{k,j,\alpha_{k}'})=\beta_k T_j$. By Lemma \ref{stabilizationlemma}, we have $\beta_k=\beta_{\alpha}$ and $\gamma_k=1$ for all but finitely many $k$. Fix a sufficiently large $k$ and consider the following cube.
\begin{equation}\label{C6}
\begin{tikzcd}[row sep=2.5em]
\mcb_{\infty,j,\alpha}\arrow[rr,"{\varrho_{k}}"] \arrow[dr,swap,"{\wt{f}_{\infty,j,\alpha}}"] \arrow[dd,swap,"{\Lambda_{\infty,j,\alpha}}"] &&
  \mcb_{k,j,\alpha_{k}'} \arrow[dd,swap,"{\Lambda_{k,j,\alpha_{k}'}}" near start] \arrow[dr,"\tf_{k,j,\alpha_{k}'}"] \\
& \mcd_{\infty,j,\alpha} \arrow[rr,swap,crossing over,"\sigma_k" near start] &&
  \mcd_{k,j,\alpha_{k}'} \arrow[dd,"{\lambda_{k,j,\alpha_{k}'}}"] \\
\tY_j \arrow[rr,"{\Delta_{\gamma_k}}" near end] \arrow[dr,swap,"{f_{j,\beta_{\alpha}}}"] && \tY_j \arrow[dr,swap,"{f_{j,\beta_k}}"] \\
& D_{j,\beta_{\alpha}} \arrow[rr,swap,"\delta_{\gamma_k}"] \arrow[uu,<-,crossing over,"{\lambda_{\infty,j,\alpha}}" near end]&& D_{j,\beta_k}
\end{tikzcd}\tag{C6}
\end{equation}
Since $\Delta_{\gamma_k}=id$ and $\delta_{\gamma_k}=id$, the bottom face commutes. Commutativity of the top face is proved in Proposition \ref{projectiondiagramprop}. The left face is given in Proposition \ref{littlelambdaprop2}. The right face is given in Proposition \ref{littlelambdaprop}. Commutativity of the back face was verified in the proof of Theorem \ref{littleimagetheorem}. Since the diagonal maps are surjective, the front face commutes. Because $\Delta_{\gamma_k}=id$ and $\delta_{\gamma_k}=id$, the front and back faces collapse into the desired triangles.
\end{proof}

\subsection{A homotopy inverse $\wt{g}$ for $\wt{f}$}

In the next theorem, we define a homotopy inverse for $\wt{f}$.

\begin{theorem}\label{hitheorem}
The quotient map $\tf:(\tY,\ty_0)\to (\tZ,\tz_0)$ has a based homotopy inverse $\tg:(\tZ,\tz_0)\to (\tY,\ty_0)$. In particular, there are homotopies $\wt{H}$ from $id_{\tY}$ to $\tg\circ \tf$ and $\wt{G}$ from $id_{\tZ}$ to $\tf\circ \tg$ that make the following square commute.
\[\xymatrix{
\tY\times \ui \ar[r]^-{\wt{H}} \ar[d]_-{\tf\times id} & \tY \ar[d]^-{\wt{f}}\\
\tZ\times \ui \ar[r]_-{\wt{G}} & \tZ
}\]
\end{theorem}

\begin{proof}
Recall that $\wh{f}_0:\wh{Y}_0\to \wh{Z}_0$ and $\wh{g}_0:\wh{Z}_0\to \wh{Y}_0$ are homotopy inverses. Also, $\wh{H}:\wh{Y}_0\times \ui\to\wh{Y}_0$ is a based homotopy from $id_{\wh{Y}_0}$ to $\wh{g}_0\circ \wh{f}_0$ and $\wh{G}:\wh{Z}_0\times \ui\to\wh{Z}_0$ is a based homotopy from $id_{\wh{Z}_0}$ to $\wh{f}_0\circ  \wh{g}_0$. Set $\wt{g}=\phi^{-1}\circ\wh{g}_0\circ \psi:\tZ\to \tY$ so that the left square commutes.
\[\xymatrix{
\tY \ar[r]^-{\phi}_-{\cong} & \lpc(\wh{Y}_0) \ar[r]^-{id} & \wh{Y}_0 \\
\tZ \ar[u]^-{\wt{g}} \ar[r]_-{\psi}^-{\cong} & \lpc(\wh{Z}_0) \ar[u]_-{\lpc(\wh{g}_0)} \ar[r]^-{id} & \wh{Z}_0 \ar[u]_-{\wh{g}_0}
}\]

Since $\phi_{Y}\circ \wt{g}=\wh{g}_0\circ \psi:\tZ\to \wh{Y}_0$ is continuous and $\tZ$ is locally path connected, $\wt{g}$ is continuous by Corollary \ref{continuitycor1}.

Since $\lpc$ preserves products, we have $\lpc(\wh{Y}_0\times \ui)=\lpc(\wh{Y}_0)\times \ui$. Define $\wt{H}:\wt{Y}\times\ui\to\wt{Y}$ by $\wt{H}=\phi_{Y}^{-1}\circ \wh{H}_0\circ (\phi_{Y}\times id)$. The same argument used for $\wt{g}$ gives the continuity of $\wt{H}$. Define $\wt{G}:\tZ\times\ui\to \tZ$ by $\wt{G}=\psi^{-1}\circ\wh{G}_0\circ (\psi\times id)$. Since $\psi\circ \wt{G}=\wh{G}_0\circ (\psi\times id)$ is continuous, $\wt{G}$ is continuous by Corollary \ref{continuitycor}. Based on the definitions given and the established results for $\wh{H}_0$ and $\wh{G}_0$, a straightforward check shows that $\wt{H}$ is a based homotopy from $id_{\tY}$ to $\wt{g}\circ \wt{f}$ and $\wt{G}$ is a based homotopy from $id_{\tZ}$ to $\wt{f}\circ \wt{g}$.

Since $\wh{f}_0\circ \wh{H}_0=\wh{G}_0\circ (\wh{f}_0\times id)$ (Corollary \ref{zeromapscorollary}), a direct verification from the formulas gives $\wt{G}\circ (\wt{f}\times id)=\wt{f}\circ \wt{H}$.
\end{proof}

\begin{remark}
Recall that $\wh{g}_0$ is the restriction of $\wh{g}=(\wt{g}_k)$. Note that $\wt{g}$ is defined precisely so that the left square in the following diagram commutes. In particular, $\psi(z)=(\sigma_k(z))$. Commutativity of other parts of the diagram have already been established as well.
\[\xymatrix{
\tY  \ar@/^1pc/[rr]^-{\varrho_k} \ar[r]_-{\phi_{Y}} & \wh{Y}_0 \ar[r]_-{\wt{r}_k} & \wt{Y}_{\leq k} \\
\tZ  \ar[u]^-{\wt{g}} \ar@/_1pc/[rr]_-{\sigma_k} \ar[r]^-{\psi} & \wh{Z}_0 \ar[u]^-{\wh{g}_0}  \ar[r]^-{\wt{s}_k} & \tZ_k \ar[u]_-{\wt{g}_k}
}\]It follows that $\varrho_k\circ \wt{g}=\wt{g}_k\circ \sigma_k$ for all $k\in\bbn$.
\end{remark}

We will have need to characterize the behavior of $\wt{g}$, $\wt{H}$, and $\wt{G}$ on some relevant subspaces. We begin with $\wt{H}$.

\begin{lemma}\label{hlemma}
For all $j\in\bbn$ and $\alpha\in\ntij$, we have $\wt{H}(\mcb_{\infty,j,\alpha}\times\ui)\subseteq \mcb_{\infty,j,\alpha}$. Moreover, if $\wt{H}_{\infty,j,\alpha}:\mcb_{\infty,j,\alpha}\times\ui\to \mcb_{\infty,j,\alpha}$ is the corresponding restriction of $\wt{H}$, then the following diagram commutes
\[\xymatrix{
\mcb_{\infty,j,\alpha}\times\ui \ar[d]_-{\Lambda_{\infty,j,\alpha}\times id} \ar[r]^-{\wt{H}_{\infty,j,\alpha}} &\mcb_{\infty,j,\alpha} \ar[d]^-{\Lambda_{\infty,j,\alpha}} \\
\tY_j\times \ui \ar[r]_-{H_{j,\beta_{\alpha}}} & \tY_j
}\]
\end{lemma}

\begin{proof}
Write $\varrho_{k}(\alpha)=\alpha_{k} '\gamma_{k}$ for $\alpha_{k}'\in \ntkj$, and $\gamma_{k}\in\pi_1(Y_j)$. For brevity, write $\mathbf{A}=\wh{Y}_0\cap \left(\prod_{k\geq j}\mcb_{k,j,\alpha_{k}'}\times \prod_{k=1}^{j-1}\{\alpha_{k}'\}\right)$. We have already established the following.
\begin{enumerate}
\item $\phi_{Y}$ maps $\mcb_{\infty,j,\alpha}$ homeomorphically onto $\mathbf{A}$,
\item $\wh{H}_0$ is given by $\wt{H}_k$ in each coordinate,
\item $\wt{H}_k(\{\alpha_{k}'\}\times \ui)=\{\alpha_{k}'\}$ and $\wt{H}_k(\mcb_{k,j,\alpha_{k}'}\times \ui)\subseteq \mcb_{k,j,\alpha_{k}'}$ for all $k$.
\end{enumerate}
Therefore, if $((a_k),t)\in \mathbf{A}\times \ui$, then (2) and (3) give $\wh{H}_0((a_k),t)=(\wt{H}_k((a_k,t)))\in\mathbf{A}$. Thus $\wh{H}_0(\mathbf{A}\times \ui)\subseteq \mathbf{A}$. This gives
\begin{eqnarray*}
\wt{H}(\mcb_{\infty,j,\alpha}\times \ui) &=& \phi_{Y}^{-1}\circ \wh{H}_0\circ (\phi_{Y}\times id)(\mcb_{\infty,j,\alpha}\times \ui)\\
&=& \phi_{Y}^{-1}\circ \wh{H}_0(\mathbf{A}\times\ui)\\
&\subseteq & \phi_{Y}^{-1}(\mathbf{A})\\
&=& \mcb_{\infty,j,\alpha}
\end{eqnarray*}
To verify commutativity of the square, we let $\Lambda_{k,j,\alpha_{k}'}(\mct_{k,j,\alpha_{k}'})=\beta_k T_j$ when $k\geq j$ so that $\beta_{\alpha}$ is the eventual value of the sequence $\{\beta_k\}_{k\geq j}$. Consider the following cube.
\begin{equation}\label{C7}
\begin{tikzcd}[row sep=2.5em]
\mcb_{\infty,j,\alpha}\times I \arrow[rr,"{\wt{H}_{\infty,j,\alpha}}"] \arrow[dr,swap,"{\varrho_k\times id}"] \arrow[dd,swap,"{\Lambda_{\infty,j,\alpha}\times id}"] &&
  \mcb_{\infty,j,\alpha} \arrow[dd,swap,"{\Lambda_{\infty,j,\alpha}}" near start] \arrow[dr,"\varrho_k"] \\
& \mcb_{k,j,\alpha_{k}'}\times\ui \arrow[rr,swap,crossing over,"\wt{H}_{k,j,\alpha_{k}'}" near start] &&
  \mcb_{k,j,\alpha_{k}'} \arrow[dd,"{\Lambda_{k,j,\alpha_{k}'}\times id}"] \\
\tY_j\times I \arrow[rr,"{H_{j,\beta_{\alpha}}}" near end] \arrow[dr,swap,"{\Delta_{\gamma_k}\times id}"] && \tY_j \arrow[dr,swap,"{\Delta_{\gamma_k}}"] \\
& \tY_j\times I \arrow[rr,swap,"H_{j,\beta_{k}}"] \arrow[uu,<-,crossing over,"\Lambda_{k,j,\alpha_{k}'}" near end]&& \tY_j
\end{tikzcd}\tag{C7}
\end{equation}
Fix $k$ sufficiently large so that $\beta_k=\beta_{\alpha}$ and $\gamma_{k}=1$. For such $k$, $H_{j,\beta_{\alpha}}=H_{j,\beta_k}$ and $\Delta_{\gamma_k}=id$, which makes the bottom face commute. Proposition \ref{lambdaeventualprop} gives the commutativity of the left and right faces commute. The front face commutes for all $k\geq j$ by the construction of $\wt{H}_{k}$. All of the diagonal maps are homeomorphisms. Thus the back face commutes.
\end{proof}

We characterize the behavior of $\wt{g}$ and $\wt{G}$ in a similar fashion.

\begin{lemma}
For all $j\in\bbn$ and $\alpha\in\ntij$, we have $\wt{g}(\mcd_{\infty,j,\alpha})\subseteq \mcb_{\infty,j,\alpha}$ and $\wt{G}(\mcd_{\infty,j,\alpha}\times\ui)\subseteq \mcd_{\infty,j,\alpha}$.
\end{lemma}

\begin{proof}
Since $\psi$ is bijective, we may show $\phi_{Y}\circ\wt{g}(\mcd_{\infty,j,\alpha})\subseteq \phi_{Y}(\mcb_{\infty,j,\alpha})$. Write $\varrho_{k}(\alpha)=\alpha_{k} '\gamma_{k}$ for $\alpha_{k}'\in \ntkj$, and $\gamma_{k}\in\pi_1(Y_j)$. For brevity, write $\mathbf{A}=\wh{Y}_0\cap \left(\prod_{k\geq j}\mcb_{k,j,\alpha_{k}'}\times \prod_{k=1}^{j-1}\{\alpha_{k}'\}\right)$ and $\mathbf{B}=\wh{Z}_0\cap \left(\prod_{k\geq j}\mcd_{k,j,\alpha_{k}'}\times \prod_{k=1}^{j-1}\{\wt{f}_k(\alpha_{k}')\}\right)$. We have already established the following:
\begin{enumerate}
\item $\phi_{Y}$ is a bijection and $\phi_{Y}(\mcb_{\infty,j,\alpha})=\mathbf{A}$
\item $\psi$ is a bijection and $\psi(\mcd_{\infty,j,\alpha})=\mathbf{B}$
\item $\wt{g}_k(\wt{f}_k(\alpha_{k}'))=\alpha_{k}'$ for all $k$ and $\wt{g}_k(\mcd_{k,j,\alpha_{k}'})=\mcb_{k,j,\alpha_{k}'}$ for all $k\geq j$.
\end{enumerate}
Note that (3) gives $\wh{g}_0(\mathbf{B})\subseteq\mathbf{A}$. Thus
\[
\wt{g}(\mcd_{\infty,j,\alpha}) = \phi_{Y}^{-1}\circ \wh{g}_0\circ \psi(\mcd_{\infty,j,\alpha})
= \phi_{Y}^{-1}\circ \wh{g}_0\left(\mathbf{B}\right) \subseteq  \phi_{Y}^{-1}\left(\mathbf{A}\right)
= \mcb_{\infty,j,\alpha}
\]
Similarly, the fact that $\wt{G}_k(\{\wt{f}_k(\alpha_{k}')\}\times\ui)=\alpha_{k}'$ for all $k$ and $\wt{G}_k(\mcd_{k,j,\alpha_{k}'}\times\ui)=\mcd_{k,j,\alpha_{k}'}$ for all $k\geq j$ gives $\wh{G}_0(\mathbf{B}\times\ui)\subseteq \mathbf{B}$. Thus $
\wt{G}(\mcd_{\infty,j,\alpha}\times\ui) =\psi^{-1}\circ \wh{G}_0\circ (\psi\times id)(\mcd_{\infty,j,\alpha}\times\ui)
= \psi^{-1}\circ \wh{G}_0\left(\mathbf{B}\times \ui\right)\subseteq  \psi^{-1}\left(\mathbf{B}\right)
= \mcd_{\infty,j,\alpha}$.
\end{proof}

For each $j\in\bbn$ and $\alpha\in\ntij$, let $\wt{g}_{\infty,j,\alpha}:\mcd_{\infty,j,\alpha}\to \mcb_{\infty,j,\alpha}$ and $\wt{G}_{\infty,j,\alpha}:\mcd_{\infty,j,\alpha}\times \ui\to \mcd_{\infty,j,\alpha}$ be the corresponding restrictions of $\wt{g}$ and $\wt{G}$.

\begin{proposition}\label{gprop}
For every $j\in\bbn$ and $\alpha\in\ntij$, the following squares commute.
\[\xymatrix{
\mcb_{\infty,j,\alpha} \ar[r]^-{\Lambda_{\infty,j,\alpha}} & \tY_j & \mcd_{\infty,j,\alpha}\times\ui \ar[d]_-{\wt{G}_{\infty,j,\alpha}} \ar[rr]^-{\lambda_{\infty,j,\alpha}\times id} && D_{j,\beta_{\alpha}}\times\ui \ar[d]^-{G_{j,\beta_{\alpha}}} \\
\mcd_{\infty,j,\alpha}  \ar[u]^-{g_{\infty,j,\alpha}} \ar[r]_-{\lambda_{\infty,j,\alpha}} & D_{j,\beta_{\alpha}} \ar[u]_-{\wt{g}_{j,\beta_{\alpha}}}  & \mcd_{\infty,j,\alpha}  \ar[rr]_-{\lambda_{\infty,j,\alpha}} && D_{j,\beta_{\alpha}}
}\]
\end{proposition}

\begin{proof}
Write $\varrho_{k}(\alpha)=\alpha_{k} '\gamma_{k}$ for $\alpha_{k}'\in \ntkj$, and $\gamma_{k}\in\pi_1(Y_j)$ and write $\Lambda_{k,j,\alpha_{k}'}(\mct_{k,j,\alpha_{k}'})=\beta_k T_j$. Find $k$ sufficiently large so that $\beta_{\alpha}=\beta_{k}$ and (using Lemma \ref{lambdaeventualprop}) such that $\Lambda_{\infty,j,\alpha}=\Lambda_{k,j,\alpha_{k}'}\circ \varrho_k$ and
$\lambda_{\infty,j,\alpha}=\lambda_{k,j,\alpha_{k}'}\circ \sigma_k$. We then have that the top and bottom triangles of the following prism commute.
\begin{center}
\begin{tikzcd}[row sep=2.5em]
\mcb_{\infty,j,\alpha} \arrow[rr,"{\varrho_{k}}"] \arrow[dr,swap,"{\Lambda_{\infty,j,\alpha}}"]  &&
  \mcb_{k,j,\alpha_{k}'}  \arrow[dl,"{\Lambda_{k,j,\alpha_{k}'}}"]  \\
& \tY_j  \\
\mcd_{\infty,j,\alpha} \arrow[uu,"{g_{\infty,j,\alpha}}"]  \arrow[rr,"{\sigma_{k}}", near start] \arrow[dr,swap,"{\lambda_{\infty,j,\alpha}}"]  &&
  \mcd_{k,j,\alpha_{k}'} \arrow[uu,swap,"{g_{k,j,\alpha_{k}'}}"] \arrow[dl,"{\lambda_{k,j,\alpha_{k}'}}"]  \\
& D_{j,\beta_{\alpha}} \arrow[uu,crossing over,near end,swap,"{g_{j,\beta_{\alpha}}}"]  
\end{tikzcd}
\end{center}
Restricting $\wt{g}_k\circ \sigma_k=\varrho_k\circ \wt{g}$ gives the commutativity of the back face. Since  $\beta_{\alpha}=\beta_{k}$, Lemma \ref{gdeflemma} gives the commutativity of the right face. Therefore, the left face commutes.

The second square in the proposition appears as the bottom face in the following cube.
\begin{equation}\label{C8}
\begin{tikzcd}[row sep=2.7em]
\mcb_{\infty,j,\alpha }\times \ui \arrow[rr,"{\wt{H}_{\infty,j,\alpha}}"] \arrow[dr,"{\Lambda_{\infty,j,\alpha}\times id}"] \arrow[dd,swap,"{\wt{f}_{\infty,j,\alpha}\times id}"] &&
  \mcb_{\infty,j,\alpha} \arrow[dd,swap,"{\wt{f}_{\infty,j,\alpha}}" near start] \arrow[dr,"\Lambda_{\infty,j,\alpha}"] \\
& \tY_j\times \ui \arrow[rr,swap,crossing over,"H_{j,\beta_{\alpha}}" near start] &&
  \tY_j \arrow[dd,"{f_{j,\beta_{\alpha}}}"] \\
\mcd_{\infty,j,\alpha }\times \ui \arrow[rr,swap,"{\wt{G}_{\infty,j,\alpha}}" near end] \arrow[dr,swap,"{\lambda_{\infty,j,\alpha}\times id}"] && \mcd_{\infty,j,\alpha }  \arrow[dr,"{\lambda_{\infty,j,\alpha}}"] \\
& D_{j,\beta_{\alpha}}\times\ui \arrow[rr,swap,"G_{j,\beta_{\alpha}}"] \arrow[uu,<-,crossing over,"f_{j,\beta_{\alpha} }\times id" near end]&& D_{j,\beta_{\alpha}}
\end{tikzcd}\tag{C8}
\end{equation}
From the proof of Theorem \ref{hitheorem}, we have $\wt{f}\circ\wt{H}=\wt{g}\circ (\wt{f}\times id)$, which gives the commutativity of the back face. The front face is a case of the bottom face of Cube \ref{C1} in Section \ref{sectiontranslatesoftj}. The right and left faces follow from Proposition \ref{littlelambdaprop2}. Commutativity of the top face was proved in Lemma \ref{hlemma}. Since the vertical maps are surjective, the bottom face commutes.
\end{proof}

\subsection{The topological structure of $\tZ$}

In this section, we give a more detailed topological description of the space $Z$. Recall from Section \ref{sectiontranslatesoftj} the definition of the maximal tree $\bft_j$ in $\tY_j$ and its translates $\beta \bft_j$, $\beta\in\pi_1(Y_j)$.

\begin{definition}
Fix $j\in\bbn$, $k\in\bbn\cup\{\infty\}$ with $j\leq k$ and $\alpha\in\ntkj$. If $\Lambda_{k,j,\alpha}(\mct_{k,j,\alpha})=\beta T_j$, we define
\begin{enumerate}
\item $\bft_{k,j,\alpha}=\Lambda_{k,j,\alpha}^{-1}(\beta\bft_j)$,
\item and set $\bft_{k}=\bigcup_{1\leq j\leq k}\bigcup_{\alpha\in\ntkj}\bft_{k,j,\alpha}$.
\end{enumerate} 
as subspaces of $\tY_{\leq k}$ ($\tY$ if $k=\infty$).
\end{definition}

Each $\bft_{k,j,\alpha}$ consists of the maximal tree $\mct_{k,j,\alpha}$ of $\mca_{k,j,\alpha}$ and $\overline{\mcb_{k,j,\alpha}\backslash\mca_{k,j,\alpha}}$, which is the disjoint union of the arcs $\bfe_{k,j,\alpha,\gamma}$, $\gamma\in\pi_1(Y_j)$. Therefore, whenever $k<\infty$, $\bft_{k}$ is a maximal tree in $\tY_{\leq k}$. In the case $k=\infty$, it is worth noting that $p^{-1}(y_0)\subseteq \bft_{\infty}$. Indeed, if $\alpha\in p^{-1}(y_0)$, find a $j$ such that the reduced word $w_{\alpha}$ does not end in a letter from $\pi_1(Y_j)$. Then $\alpha$ is an arc-endpoint of $\mcb_{\infty,j,\alpha}$ and therefore a free endpoint of $\bft_{\infty,j,\alpha}$. Although, $\bft_{\infty}$ is not a tree in the usual sense it will serve as a kind of analogue of a ``maximal tree" in $\tY$. 

\begin{lemma}\label{treeinverselimit}
For all $k\in\bbn$, $\varrho_{k}(\bft_{\infty})=\bft_{k}$ and $\wt{r}_{k+1,k}(\bft_{k+1})=\bft_k$.
\end{lemma}

\begin{proof}
Fix a subspace $\bft_{\infty,j,\alpha}$ of $\bft_{\infty}$. Write $\varrho_{k}(\alpha)=\alpha_{k}'\gamma_{k}$ for $\alpha_{k}'\in\ntkj$ and $\gamma_{k}\in\pi_1(Y_j)$. As previously established, $\varrho_{k}$ maps $\mcb_{\infty,j,\alpha}$ homeomorphically onto $\mcb_{k,j,\alpha_{k}'}$. In particular $\varrho_k$ maps the arcs in $\overline{\mcb_{\infty,j,\alpha}\backslash \mca_{\infty,j,\alpha}}$ to the arcs $\overline{\mcb_{k,j,\alpha_k'}\backslash \mca_{k,j,\alpha_k'}}$ and $\mct_{\infty,j,\alpha}$ to $\mct_{k,j,\alpha_k'}$. Thus $\varrho_{k}(\bft_{\infty,j,\alpha})\subseteq \bft_{k}$. The inclusion $\varrho_{k}(\bft_{\infty})=\bft_{k}$ follows for all $k$. It follows immediately that $\wt{r}_{k+1,k}(\bft_{k+1})=\wt{r}_{k+1,k}(\varrho_{k+1}(\bft_{\infty}))=\varrho_{k}(\bft_{\infty})=\bft_k$.
\end{proof}

\begin{lemma}\label{tlemma}
$\bft_{\infty}$ is a uniquely arcwise connected and locally path-connected subspace of $\tY$.
\end{lemma}

\begin{proof}
Consider distinct points $x,y\in \bft_{\infty}$. Find an arc $\ell:\ui\to \tY$ from $x$ to $y$. For each $j\in\bbn$ and $\alpha\in\ntij$, the set $W_{\infty,j,\alpha}=\mca_{\infty,j,\alpha}\backslash (\overline{\mcb_{k,j,\alpha}\backslash\mca_{k,j,\alpha}})$ is open in $\mcu_{\infty,j,\alpha}$. Since $\mcu_{\infty,j,\alpha}$ is open in $\tY$, $W_{\infty,j,\alpha}$ is open in $\tY$. We will define another path $\eta$ that is path-homotopic to $\ell$. If $\ell$ already has image in $\bft_{\infty}$, we take $\eta=\ell$. Otherwise, fix $j\in\bbn$ and $\alpha\in\ntij$ and let $(a,b)$ be a component of $\ell^{-1}(W_{\infty,j,\alpha})$. If $\ell|_{[a,b]}:\ui\to \mca_{\infty,j,\alpha}$ is a loop, then we define $\eta|_{[a,b]}$ to be the constant path at $\ell(a)=\ell(b)$. If $\ell$ is not a loop, then $\ell(a)$ and $\ell(b)$ are endpoints in the tree $\mct_{\infty,j,\alpha}$. Choose a path $\eta|_{[a,b]}:\ui\to \mct_{\infty,j,\alpha}$ from $\ell(a)$ to $\ell(b)$ that parameterizes the unique arc between these points. By construction, we have $\im(\ell)\subseteq \bft_{\infty}$, $\ell(0)=x$, and $\ell(1)=y$ so it suffices to verify continuity. Note that we obtain the projection $\varrho_{k}\circ \eta:\ui\to \wt{Y}_{\leq k}$ from $\varrho_{k}\circ \ell:\ui\to \wt{Y}_{\leq k}$ by either making subloops of $\varrho_{k}\circ \ell$ constant and possibly modifying finitely many other subpaths. Hence, $\varrho_{k}\circ\eta:\ui\to \wt{Y}_{\leq k}$ is continuous for each $k$. By Corollary \ref{continuitycor1}, $\eta:\ui\to \wt{Y}$ is continuous.

For later in the proof we go ahead and construct a homotopy $H:I^2\to \wt{Y}$ from $\ell$ to $\eta$. In particular, if $(a,b)$ is a component of $\ell^{-1}(W_{\infty,j,\alpha})$, we define $H$ so that $H([a,b]\times I)\subseteq \mca_{\infty,j,\alpha}$. The only part of the construction of $H$ that cannot be done arbitrarily is the following: let $U$ be a contractible neighborhood of $\ell(a)$ in $\mca_{\infty,j,\alpha}$ with contraction $c:U\times \ui\to U$. If $\ell|_{[a,b]}$ is a loop in $U$ based at $\ell(a)$, we define $H(s,t)=c(\ell(s),t)$. The argument for the continuity of $H$ is the same as it was for $\eta$.

Lemma \ref{treeinverselimit} gives $\wt{r}_{k+1,k}(\bft_{k+1})\subseteq \bft_k$ for all $k\in\bbn$. Therefore, we may form the sub-inverse limit $\varprojlim_{k}\bft_k$ of trees, which is a subspace of $\wh{Y}$. It is well-known that every path component of an inverse limit of trees is uniquely arcwise connected. Since $\varrho_k(\bft_{\infty})\subseteq \bft_k$ for all $k$, we have $\phi_{Y}(\bft_{\infty})\subseteq \varprojlim_{k}\bft_k$. Since $\phi_{Y}$ continuously injects $\bft_{\infty}$ into a uniquely arcwise connected space, we conclude that $\bft_{\infty}$ is uniquely arcwise connected.

Finally, we check that $\bft_{\infty}$ is locally path connected. Local path connectivity is clear at all points in $\bft_{\infty}\backslash p^{-1}(y_0)$. Let $\alpha\in p^{-1}(y_0)$ and consider a basic neighborhood $N(\alpha,U)\cap \bft_{\infty}$ of $\alpha$ in the subspace topology where $U$ is a neighborhood of $y_0$ in $Y$. We may assume that there is a $J\in\bbn$ such that $Y_j\subseteq U$ for all $j> J$ and $Y_j\cap U$ is an open arc in $Y_j\backslash X_j$ when $j\leq J$. We will show that $N(\alpha,U)\cap \bft_{\infty}$ is path connected.

Choose $\alpha\beta\in N(\alpha,U)\cap\bft_{\infty}$ for $\beta\in\wt{U}$ ($U$ has basepoint $y_0$). Note that there exists $j_0\in \bbn$, $\gamma\in \pi_1(U,y_0)\cap \nt_{\infty,j_0}$, and possibly trivial $\delta\in \tY_{j_0}\cap p_{j_0}^{-1}(U\cap Y_{j_0})$ such that $\beta=\gamma\delta$. We use $\delta$ to denote a path $(\ui,0)\to (U\cap Y_{j_0},y_0)$ representing the homotopy class $\delta$ and similarly, we use $\gamma$ to denote a reduced loop representing $\gamma$. In the case that $j_0\leq J$, we choose $\gamma$ to be a parameterization of the arc from $y_0$ to $\delta(1)\in U\cap Y_{j_0}$. Additionally, note that $w_{\gamma}$ only has letters from $\pi_1(Y_j)$, $j>J$. Let $\ell_1:\ui\to \tY$ be the standard lift of $\gamma$ starting at $\alpha$ and let $\ell_2:\ui\to \tY$ be the standard lift of $\delta$ starting at $\alpha\gamma$. Then $\ell_1$ is a path in $N(\alpha,U)$ from $\alpha$ to $\alpha\gamma$ and $\ell_2$ is a path in $N(\alpha,U)$ from $\alpha\gamma$ to $\alpha\beta$. Moreover, the image of $\ell_1$ only meets $p^{-1}(Y_j)$ for $j>J$. Using the construction in the second paragraph, find paths $\eta_i:\ui\to \bft_{\infty}$, $i\in\{1,2\}$ such that $\eta_i$ is path-homotopic to $\ell_i$ in $\tY$. In particular, let $H_i:\ui\times\ui\to \tY$, $i\in\{1,2\}$ be the path-homotopy from $\ell_i$ to $\eta_i$ constructed in the third paragraph. 

Because we chose $\gamma$ to be reduced, we have $\im(p\circ H_1)\subseteq U$. Unique path-lifting ensures that $H_1$ has image in $N(\alpha,U)$ and thus $\eta_1$ has image in $N(\alpha,U)\cap\bft_{\infty}$. If $j_0\leq J$, then $\delta$ has image in $N(\alpha,U)\cap \bft_{\infty}$. In this case, $\eta_2=\ell_2$ and $\wt{H}_2$ is the constant homotopy, which makes it clear that $\eta_2$ has image in $N(\alpha,U)\cap \bft_{\infty}$. If $j_0>J$, then $\im(p\circ H_2)\subseteq U$ and thus $H_2$ has image in $N(\alpha,U)$. Thus $\eta_2$ also has image in $N(\alpha,U)\cap \bft_{\infty}$. Since $\eta_1\cdot\eta_2$ is a path from $\alpha$ to $\alpha\beta$ in $N(\alpha,U)\cap \bft_{\infty}$, we conclude that $N(\alpha,U)\cap \bft_{\infty}$ is path connected.
\end{proof}

\begin{definition}
A \textit{topological $\bbr$-tree} is a topological space, which is metrizable, uniquely arc-wise connected and locally arcwise-connected.
\end{definition}

It is known that every topological $\bbr$-tree may be equipped with an $\bbr$-tree metric \cite{MORtree} and  that $\bbr$-trees are contractible \cite{morganrtree}

\begin{corollary}\label{tiscontractible}
If $X_j$ is locally finite for every $j\in\bbn$, then $\bft_{\infty}$ is contractible.
\end{corollary}

\begin{proof}
As noted in Remark \ref{metrizableremark}, if each $X_j$ is locally finite, then $\tY$ is metrizable. Thus $\bft_{\infty}$ is metrizable. By Lemma \ref{tlemma}, $\bft_{\infty}$ meets the other conditions of being a topological $\bbr$-tree.
\end{proof}

\begin{definition}
Fix $k\in\bbn\cup\{\infty\}$, $j\in\bbn$ with $j\leq k$, and $\alpha\in\ntkj$. Identify $\wt{f}=\wt{f}_{\infty}$ in the case $k=\infty$. Define
\begin{enumerate}
\item $\bfee_{k,j,\alpha}=\wt{f}_k(\bft_{k,j,\infty})$ whenever $j\in\bbn$ and $1\leq j\leq k$,
\item and $\bfee_{k}=\wt{f}_{k}(\bft_k)$
\end{enumerate} 
as subspaces of $Z_k$ ($Z$ when $k=\infty$).
\end{definition}

\begin{remark}\label{preimageremark}
For each $k\in\bbn\cup\{\infty\}$ we have $\wt{f}_{k}^{-1}(\bfee_{k})=\bft_{k}$. Indeed, if $x\in \tY_{\leq k}\backslash \bft_{k}$, then $x\in\mca_{k,j,\alpha}\backslash\mct_{k,j,\alpha}$ for some $j\in\bbn$ and $\alpha\in\ntkj$. This would give $\wt{f}(x)\in \mcc_{k,j,\alpha}\backslash \{c_{k,j,\alpha}\}$, where $\mcc_{k,j,\alpha}\backslash \{c_{k,j,\alpha}\}$ is clearly disjoint from $\bfee_{k}$. The same reasoning gives $\tf^{-1}(\bfee_{\infty})=\bft_{\infty}$.

When $k<\infty$, $\bfee_k$ is a tree such that $\tZ_k$ consists of $\bfee_k$ with the space $\mcc_{k,j,\alpha}$ attached at $c_{k,j,\alpha}\in\bfee_k$.
\end{remark}

\begin{lemma}\label{structureofzlemma}
$\bfee_{\infty}$ is a uniquely arcwise connected and locally path-connected subspace of $\tZ$ such that
\[\tZ=\bfee_{\infty}\cup \bigcup_{j\in \bbn}\bigcup_{\alpha\in\ntij}\mcc_{\infty,j,\alpha}.\]
Moreover, for each $j\in\bbn$ and $\alpha\in\ntij$, we have
\[\tZ=(\mcc_{\infty,j,\alpha},c_{\infty,j,\alpha})\vee (\overline{\tZ\backslash \mcc_{\infty,j,\alpha}},c_{\infty,j,\alpha}).\]
\end{lemma}

\begin{proof}
For the moment, fix $j\in\bbn$ and $\alpha\in\ntij$. Since $\mca_{\infty,j,\alpha}=\wt{f}^{-1}(\mcc_{\infty,j,\alpha})$ is closed in $\tY$, $\mcc_{\infty,j,\alpha}$ is closed in $\tZ$. Additionally, recall from the start of the proof of Lemma \ref{tlemma} that $W_{\infty,j,\alpha}=\mca_{\infty,j,\alpha}\backslash (\overline{\mcb_{k,j,\alpha}\backslash\mca_{k,j,\alpha}})$ is open in $\tY$. Since $W_{\infty,j,\alpha}=\wt{f}^{-1}(\mcc_{\infty,j,\alpha}\backslash\{c_{\infty,j,\alpha}\})$, the set $\mcc_{\infty,j,\alpha}\backslash\{c_{\infty,j,\alpha}\}$ is open in $\tZ$. This is enough to give the wedge-sum factorization \[\tZ=(\mcc_{\infty,j,\alpha},c_{\infty,j,\alpha})\vee (\overline{\tZ\backslash \mcc_{\infty,j,\alpha}},c_{\infty,j,\alpha}).\]

Since the above paragraph holds for all pairs $(j,\alpha)$ and\[\bfee_{\infty}=\tZ\backslash\left(\bigcup_{j\in\bbn}\bigcup_{\alpha\in\ntij}\mcc_{\infty,j,\alpha}\backslash\{c_{\infty,j,\alpha}\}\right),\]
it follows that $\bfee_{\infty}$ is closed in $\tZ$. Recall from Remark \ref{preimageremark} that $\wt{f}^{-1}(\bfee_{\infty})=\bft_{\infty}$ and so the restricted map $\wt{f}|_{\bft_{\infty}}:\bft_{\infty}\to\bfee_{\infty}$ is a quotient map. Since $\bft_{\infty}$ is path connected and locally path connected by Lemma \ref{tlemma}, so is its quotient space $\bfee_{\infty}$. Additionally, Lemma \ref{tlemma} gives that $\bft_{\infty}$ is uniquely arcwise connected. Every fiber of $\wt{f}|_{\bft_{\infty}}:\bft_{\infty}\to\bfee_{\infty}$ is either a point or a closed sub-tree of the form $\mct_{\infty,j,\alpha}$. It follows easily that every point in $\bfee_{\infty}$ separates $\bfee_{\infty}$ into two disjoint, open components. Therefore, $\bfee_{\infty}$ is uniquely arcwise connected.
\end{proof}

The second statement of Lemma \ref{structureofzlemma} implies that for each $j\in\bbn$ and $\alpha\in\ntij$, there is a retraction $\mu_{j,\alpha}:\tZ\to \mcc_{\infty,j,\alpha}$, that collapses $\tZ\backslash \mcc_{\infty,j,\alpha}$ to $c_{\infty,j,\alpha}$. Additionally, there is a retraction $\mu:\tZ\to \bfee_{\infty}$, which collapses each $\mcc_{\infty,j,\alpha}$ to $c_{\infty,j,\alpha}$.

\begin{corollary}
The spaces $\bfee_{\infty}$ and $\mcc_{\infty,j,\alpha}$, $j\in\bbn$, $\alpha\in\ntij$ are retracts of $\tZ$.
\end{corollary}

\begin{corollary}\label{dendritecor}
If $D$ is a Peano continuum in $\bft_{\infty}$ or $\bfee_{\infty}$, then $D$ is a dendrite.
\end{corollary}


\begin{theorem}
The restricted maps $\wt{f}|_{\bft_{\infty}}:\bft_{\infty}\to\bfee_{\infty}$ and $\wt{g}|_{\bft_{\infty}}:\bfee_{\infty}\to\bft_{\infty}$ are based homotopy inverses.
\end{theorem}

\begin{proof}
Since $\bft_{\infty}$ and $\bfee_{\infty}$ contain the basepoints of $\tY$ and $\tZ$ respectively, the corresponding restrictions of $\wt{H}$ and $\wt{G}$ will verify the desired homotopy equivalence if we show that $\wt{g}(\bfee_{\infty})\subseteq \bft_{\infty}$, $\wt{H}(\bft_{\infty}\times\ui)\subseteq \bft_{\infty}$, and $\wt{G}(\bfee_{\infty}\times\ui)\subseteq \bfee_{\infty}$. Fix $j\in\bbn$ and $\alpha\in\ntij$. Recall that each of $\wt{g}$, $\wt{H}$, and $\wt{G}$ are determined by their restrictions $\wt{g}_{\infty,j,\alpha}$, $\wt{H}_{\infty,j,\alpha}$, and $\wt{G}_{\infty,j,\alpha}$. Therefore, it suffices to check that $\wt{g}_{\infty,j,\alpha}(\bfee_{\infty,j,\alpha})\subseteq \bft_{\infty,j,\alpha}$, $\wt{H}_{\infty,j,\alpha}(\bft_{\infty,j,\alpha}\times\ui)\subseteq \bft_{\infty,j,\alpha}$, and $\wt{G}_{\infty,j,\alpha}(\bfee_{\infty,j,\alpha}\times\ui)\subseteq \bfee_{\infty,j,\alpha}$. In all cases, we use the fact that $\Lambda_{\infty,j,\alpha}(\bft_{\infty,j,\alpha})=\beta_{\alpha}\bft_{j}$ and $\lambda_{\infty,j,\alpha}(\bfee_{\infty,j,\alpha})=E_{j,\beta_{\alpha}}$. 

From the left square in Proposition \ref{gprop}, we have $g_{\infty,j,\alpha}=\Lambda_{\infty,j,\alpha}^{-1}\circ g_{j,\beta_{\alpha}}\circ \lambda_{\infty,j,\alpha}$. Recall from Section \ref{sectiontranslatesoftj} that we have $g_{j,\beta_{\alpha}}(E_{j,\beta_{\alpha}})\subseteq \beta_{\alpha} \bft_{j}$. If follows that
\begin{eqnarray*}
\wt{g}_{\infty,j,\alpha}(\bfee_{\infty,j,\alpha}) &=& \Lambda_{\infty,j,\alpha}^{-1}\circ g_{j,\beta_{\alpha}}\circ \lambda_{\infty,j,\alpha}(\bfee_{\infty,j,\alpha})\\
&=& \Lambda_{\infty,j,\alpha}^{-1}\circ g_{j,\beta_{\alpha}}(E_{j,\beta_{\alpha}})\\
&\subseteq & \Lambda_{\infty,j,\alpha}^{-1}(\beta_{\alpha} \bft_{j})\\
&=& \bft_{\infty,j,\alpha}
\end{eqnarray*}
The same argument using the square in Lemma \ref{hlemma} for $\wt{H}$ and the right square in Proposition \ref{gprop} for $\wt{G}$ gives the other inclusions.
\end{proof}

Combining the previous theorem with Corollary \ref{tiscontractible}, we have the following. 

\begin{corollary}\label{eiscontractible}
If $X_j$ is locally finite for every $j$, then $\bfee_{\infty}$ is contractible.
\end{corollary}

Even when $X_j$ is not locally finite, every Peano continuum in $\bfee_{\infty}$ is a dendrite and therefore contractible. Nevertheless, the author anticipates that $\bfee_{\infty}$ is contractible even when $X_j$ is not locally finite. However, this requires a characterization of contractible, uniquely arcwise connected, and locally arcwise connected spaces, which apparently does not exist in the current literature.

\subsection{Shrinking adjunction spaces in $Z$}

\begin{definition}\label{sadef}
Let $S$ be $\{1,2,\dots M\}$ or $\bbn$. Let $D$ be a path-connected space, $\{d_i\}_{i\in S}$ be a sequence (of not necessarily distinct points) in $D$ and $\{(A_i,a_i)\}_{i\in S}$ be a corresponding sequence of based, connected CW-complexes. The \textit{shrinking adjunction space with core $D$ and attachment spaces $\{(A_i,a_i)\}_{i\in S}$} is the space $\mathbf{X}=X\sqcup \coprod_{i\in S}A_i/\mathord{\sim}$ where $d_i\sim a_i$ for each $i\in S$. We give $\mathbf{X}$ the following topology: a set $U\subseteq \mathbf{X}$ is open if and only if
\begin{enumerate}
\item $D\cap U$ is open in $D$,
\item $A_i\cap U$ is open in $A_i$ for all $i\in S$,
\item whenever $S=\bbn$ and $i_1<i_2<i_3<\cdots$ such that $\{d_{i_m}\}_{m\in \bbn}$ converges to a point $d\in D\cap U$, we have $A_{i_m}\subseteq U$ for all but finitely many $m\in\bbn$. 
\end{enumerate}
\end{definition}

Note that the case in Definition \ref{sadef} where $S$ is finite is simply to allow for an ordinary finite adjunction space to be considered as a degenerate case of a shrinking adjunction space. Note that $D$ and all attachment spaces $A_i$ are retracts of $\mathbf{X}$ as defined above. Let $u_i:\mathbf{X}\to A_i$, $i\in \bbn$ be the retraction that collapses $D$ and $A_j$ for $j\neq i$ to $a_i$. The following is one of the main results of \cite{Braznsequential}.

\begin{theorem}\label{sequentialpapermainresults}
Let $n\geq 2$ and $\mathbf{X}$ be a shrinking adjunction space as described in Definition \ref{sadef} where $D$ is a Peano continuum with basepoint $d_0$. Then canonical homomorphism $\Upsilon_{\mathbf{X}}:\pi_n(\mathbf{X},d_0)\to \prod_{i\in \bbn}\pi_n(A_i,a_i)$, $\Upsilon_{\mathbf{X}}([\ell])=([u_i\circ \ell])$ is a split epimorphism. Moreover, if $D$ is a dendrite and $A_i$ is an $(n-1)$-connected CW-complex for all $i\in \bbn$, then $\Upsilon_{\mathbf{X}}$ is an isomorphism. 
\end{theorem}

In the next lemma we identify the relevant shrinking adjunction spaces within $Z$.

\begin{lemma}\label{salemma}
Let $D\subseteq \bfee_{\infty}$ be a dendrite and $\{(j_i,\alpha_i)\}_{i\in\bbn}$ be a sequence of distinct pairs where $j_i\in\bbn$, $\alpha_i\in \nt_{\infty,j_i}$, and $c_{\infty,j_i,\alpha_i}\in D$. Then $\{j_i\}_{i\in\bbn}\to \infty$ and $\mcp=D\cup \bigcup_{i\in\bbn}\mcc_{\infty,j_i,\alpha_i}$ is a shrinking adjunction space with core $D$ and attachment spaces $\mcc_{\infty,j_i,\alpha_i}$.
\end{lemma}

\begin{proof}
To simplify our notation, we write $A_i$ for $\mcc_{\infty,j_i,\alpha_i}$. 

First, we prove that $\{j_i\}_{i\in\bbn}\to \infty$. If $\{j_i\}_{i\in\bbn}\nrightarrow \infty$, then there exists $J\in\bbn$ and $i_1<i_2<i_3<\cdots$ such that $j_{i_m}=J$ and that $\alpha_{i_m}\neq \alpha_{i_{m'}}$ whenever $m\neq m'$. Since $\{c_{\infty,J,\alpha_{i_m}}\}_{m\in\bbn}$ is a sequence in the compact set $D$, we may replace $\{i_m\}$ by a subsequence so that $\{c_{\infty,J,\alpha_{i_m}}\}_{m\in\bbn}\to z$ for some $z\in D$. Since all pairs $(J,\alpha_{i_m})$ are distinct and $Z\backslash \wt{f}(p^{-1}(y_0))$ is the disjoint union of the open sets $\wt{f}(\mcu_{\infty,j,\alpha})$, we must have $z=\tf(\alpha)$ for some $\alpha\in p^{-1}(y_0)$. Since $\tg(\mcc_{\infty,J,\alpha_{i_m}})\subseteq \mca_{\infty,J,\alpha_{i_m}}$, we have $a_m=\tg(c_{\infty,J,\alpha_{i_m}})\in \mca_{\infty,J,\alpha_{i_m}}$ where $\{a_m\}_{m\in\bbn}\to \tg(z)=\alpha$. Thus $\{p(a_m)\}_{m\in\bbn}\to y_0$ in $Y$ where $p(a_m)\in X_J$ for all $m$. This is a contradiction since $X_J$ is closed in $Y$ and $y_0\notin X_J$.

With $\{j_i\}_{i\in\bbn}\to \infty$ established, note that the subspace $\mcp$ of $\tZ$ has the underlying set of the desired shrinking adjunction space. Since $D$ and all $\mcc_{\infty,j,\alpha}$ are retracts of $Z$, Conditions (1) and (2) in Definition \ref{sadef} hold; it suffices to check Condition (3). Recall that the attachment points in $D$ are $d_{i_m}=c_{\infty,j_{i_m},\alpha_{i_m}}$. Suppose $U$ is an open set in $\tZ$ and that $i_1<i_2<i_3<\cdots $ is such that $\{d_{i_m}\}_{m\in\bbn}\to d$ for $d\in U$. Suppose, to obtain a contradiction, that there exist $m_1<m_2<m_3<\cdots$ such that $A_{i_{m_{r}}}$ is not contained in $U$. By replacing $\{d_{i_m}\}$ with the corresponding subsequence, we may assume that $A_{i_m}\nsubseteq U$ for all $m$. Then there exists $a_m\in A_{i_m}\backslash U$ for all $m\in\bbn$. Since $\{d_{i_m}\}_{m\in\bbn}\to d$, $d_{i_m}\in U$ for all but finitely many $m$. Thus we may assume $a_m\neq d_{i_m}$, i.e. $a_m\in  A_{i_m}\backslash\{d_{i_m}\}$. 

Using our characterization of the structure of $Z$ again, we must have $d\in D\cap \wt{f}(p^{-1}(x_0))$. Since $\wt{f}$ is bijective on $p^{-1}(x_0)$ (with inverse $\wt{g}$), there is a unique $\beta=\wt{g}(d)\in p^{-1}(x_0)\subseteq \tY$ with $\wt{f}(\beta)=d$. Choose $\beta_m\in \mca_{\infty,j_{i_m},\alpha_{i_m}}$ with $\wt{f}(\beta_m)=a_m$. For each $m$, we have $\beta_m=\alpha_{i_m}\gamma_{m}\in\mca_{\infty,j_{i_m},\alpha_{i_m}}$ for $\gamma_{m}\in\tY_{j_{i_m}}$. 

Since $\{d_{i_m}\}_{m\in\bbn}=\{c_{\infty,j_{i_m},\alpha_{i_m}}\}_{m\in\bbn}\to d$ in $Z$, we have $\{\wt{g}(d_{i_m})\}\to \beta$ in $\tY$. The definition of $\wt{g}$ ensures that $\wt{g}(d_{i_m})\in \mcb_{\infty,j_{i_m},\alpha_{i_m}}$. Thus $\wt{g}(d_{i_m})=\alpha_{i_m}\delta_{m}$ for $\delta_{m}\in\tY_{j_{i_m}}$. In summary, $\{\alpha_{i_m}\delta_{m}\}_{m\in\bbn}\to \beta $ in $\tY$ for $\delta_{m}\in\tY_{j_{i_m}}$.

Since $\wt{f}(\beta)=d\in U$, $\wt{f}^{-1}(U)$ is an open neighborhood of $\beta$ in $\tY$. Find a basic open neighborhood $N(\beta,V)\subseteq \wt{f}^{-1}(U)$ where $V$ is an open neighborhood of $y_0$ in $Y$. We may assume that there is a $J$ such that $Y_j\subseteq V$ for all $j\geq J$. Since $\{j_{i_m}\}_{m\in\bbn}\to\infty$, we may find $M$ such that $j_{i_m}\geq J$ for all $m\geq M$. Additionally, since $\{\alpha_{i_m}\delta_{m}\}_{m\in\bbn}\to \beta$, we may choose $M$ large enough so that $\alpha_{i_m}\delta_{m}\in N(\beta,V)$ for all $m\geq M$. Fix $m\geq M$. A path representing $\delta_m$ has image in $Y_{j_{i_m}}\subseteq V$ and thus $\alpha_{i_m}\in N(\beta,V)$. It follows that $N(\beta,V)=N(\alpha_{i_m},V)$. We have $\gamma_{i_m}\in \tY_{j_{i_m}}$ and $j_{i_m}\geq J$. Thus $\beta_m=\alpha_{i_m}\gamma_{i_m}\in N(\alpha_{i_m},V)=N(\beta,V)\subseteq \wt{f}^{-1}(U)$. This gives $a_m=\wt{f}(\beta_m)\in U$; a contradiction.
\end{proof}

We refer to a subspace of $Z$ of the form $\mcp=D\cup \bigcup_{i}\mcc_{\infty,j_i,\alpha_i}$ as described in Lemma \ref{salemma} as a \textit{$D$-subcomplex of $Z$.} Note that every $D$-subcomplex of $Z$ is a retract of $Z$. Indeed, since $\bfee_{\infty}$ is uniquely arcwise connected and locally arcwise connected, there is a canonical retraction $\bfee_{\infty}\to D$ and this can easily be extended to a retract of $Z$ using the second statement of Lemma \ref{structureofzlemma}.

\begin{corollary}\label{pcinZcorollary}
Every Peano continuum in $Z$ containing $z_0$ is a subset of a $D$-subcomplex of $Z$ for some dendrite $D\subseteq \bfee_{\infty}$.
\end{corollary}

\begin{proof}
Let $P\subseteq \tZ$ be a Peano continuum containing $z_0$. Then $D=P\cap\bfee_{\infty}$ is a dendrite. Since the sets $\mcc_{\infty,j,\alpha}\backslash\{c_{\infty,j,\alpha}\}$ are open and disjoint in $\tZ$ (ranging over all pairs $(j,\alpha)$), $P$ can meet $\mcc_{\infty,j,\alpha}\backslash\{c_{\infty,j,\alpha}\}$ for at most countably many pairs $(j,\alpha)$. Otherwise, the separability of $P$ would be violated. Let $(j_i,\alpha_i)$ be an enumeration of the pairs $(j,\alpha)$ for which $P$ meets $\mcc_{\infty,j,\alpha}\backslash\{c_{\infty,j,\alpha}\}$. If the sequence $\{(j_i,\alpha_i)\}_{i}$ is finite, we define $P$ to be the finite union $\mcp=D\cup \bigcup_{i}\mcc_{\infty,j,\alpha}$ (this is trivially a shrinking adjunction space) and it is clear that $P\subseteq \mcp$. We now assume that $\{(j_i,\alpha_i)\}$ is indexed by $\bbn$. Since $P$ is path connected, we must have $c_{\infty,j_i,\alpha_i}\in D$ for all $i\in\bbn$. By Lemma \ref{salemma}, $\mcp=D\cup \bigcup_{i\in\bbn}\mcc_{\infty,j_i,\alpha_i}$ is a $D$-subcomplex of $Z$. Clearly, $P\subseteq \mcp$.
\end{proof}

Let $\mathfrak{DS}$ be the set of $D$-subcomplexes $\mcp$ in $Z$ such that $z_0\in\mcp$. Subset inclusion defines a partial ordering of $\mathfrak{DS}$. Whenever $\mcp_1 \subseteq \mcp _2$ in $\mathfrak{DS}$, $\mcp_1$ is a retract of $\mcp_2$ and so we have a canonical injective homomorphism $\varphi_{\mcp_1,\mcp_2}:\pi_n(\mcp_1,z_0)\to \pi_n(\mcp_2,z_0)$. These maps form a directed system of injective homomorphism for which we have the direct limit $\varinjlim_{\mcp\in \mathfrak{DS}}\pi_n(\mcp,z_0)$. Moreover, since each $\mcp\in\mathfrak{DS}$ is a retract of $Z$, the homomorphism $\varphi_{\mcp}:\pi_n(\mcp,z_0)\to \pi_n(\tZ,z_0)$ induced by inclusion $\mcp\to Z$ is injective. 

\begin{theorem}\label{directlimit}
For all $n\geq 2$, inclusion maps $\mcp\to Z$, $\mcp\in\mathfrak{DS}$ induce a canonical isomorphism $\varphi:\varinjlim_{\mcp\in \mathfrak{DS}}\pi_n(\mcp,z_0)\to \pi_n(\tZ,z_0)$.
\end{theorem}

\begin{proof}
Injectivity of each $\varphi_{\mcp}$ ensures that $\varphi$ is injective. Surjectivity is a direct consequence of Corollary \ref{pcinZcorollary}.
\end{proof}

Unfortunately, the isomorphism in Theorem \ref{directlimit} is impractical for understanding $\pi_n(Z)$ in terms of the homotopy groups of the spaces $X_j$. We provide another approach in the next section.


\section{Main Results}

\subsection{The homomorphism $\Psi$ and its image}

The previous section implies that every map $\ell:I^n\to Z$ will have image in some $D$-subcomplex of $Z$. Hence $\ell$ can meet countably many of the space $\mcc_{\infty,j,\alpha}\backslash\{c_{\infty,j,\alpha}\}$. Here, we show that an even stronger statement holds: we can deform $\ell$ so that for any fixed $j\in\bbn$, $\ell$ will only meet $\mcc_{\infty,j,\alpha}\backslash\{c_{\infty,j,\alpha}\}$ for finitely many pairs $(j,\alpha)$.

\begin{lemma}\label{directsumlemma}
Every map $\ell:(\ui^n,\partial I^n)\to (\tZ,z_0)$ is homotopic rel. $\partial I^n$ to a map $\ell '':(\ui^n,\partial I^n)\to (\tZ,z_0)$ such that for every $j\in\bbn$, the set of connected components of $(\ell '')^{-1}(\bigcup_{\alpha\in\ntij}\mcc_{\infty,j,\alpha}\backslash \{c_{\infty,j,\alpha}\})$ is finite.
\end{lemma}

\begin{proof}
Fix $j\in\bbn$ and let $\ell:(\ui^n,\partial \ui^n)\to (\tZ,z_0)$ be a map. Since $\wt{f}\circ \wt{g}\simeq id_{\tZ}$, we have $\ell '=\wt{f}\circ \wt{g}\circ\ell\simeq \ell$. Let $\wt{\kappa} =\wt{g}\circ \ell$ and $\kappa=p\circ \wt{\kappa}:(\ui^n,\partial \ui^n)\to (Y,y_0)$. 
\[\xymatrix{
& (\ui^n,\partial \ui^n) \ar[dl]_{\kappa} \ar[d]_-{\wt{\kappa}} \ar[dr]^-{\ell '}\\
Y & \tY \ar[l]^-{p} \ar[r]_-{\wt{f}} & \tZ
}\]
It suffices to verify the lemma for $\ell '$. We will use $\ell '$ to define a new map $\ell '':(\ui^n,\partial \ui^n)\to (\tZ,z_0)$. First, we define $\ell ''$ to agree with $\ell '$ on the closed set $(\ell ')^{-1}(\bfee_{\infty})$.

Let $U$ be a contractible neighborhood of $c_j$ in $C_j$ and let $U_{j,\alpha}=\lambda_{\infty,j,\alpha}^{-1}(U)$ be the corresponding neighborhood of $c_{\infty,j,\alpha}$ in $\mcc_{\infty,j,\alpha}$. Let $K_{j,\alpha}:U_{j,\alpha}\times \ui\to U_{j,\alpha}$ be a based contraction, i.e. a based homotopy from the identity of $U_{j,\alpha}$ to the constant map at $c_{\infty,j,\alpha}$. Fixing $\alpha\in\ntij$, consider the open set $V_{\alpha}=(\ell ')^{-1}(\mcc_{\infty,j,\alpha}\backslash c_{\infty,j,\alpha})$ in $(0,1)^n$. If $W$ is a connected component of $V_{\alpha}$, then $\ell '(\partial W)=c_{\infty,j,\alpha}$. Now, if $\ell '(W)\subseteq U_{j,\alpha}$, we define $\ell ''$ to be constant on $\ov{W}$, that is, $\ell ''(\ov{W})=c_{\infty,j,\alpha}$. On the other hand if $\ell '(W)\nsubseteq U_{j,\alpha}$, we define $\ell ''$ to agree with $\ell '$ on $W$. Define a homotopy $L:\ui^n\times \ui\to \tZ$ as follows: $L$ is the constant homotopy on $(\ell ')^{-1}(\bfee_{\infty})$ and on any component $W$ of $V_{\alpha}$ where $\ell '(W)\nsubseteq U_{j,\alpha}$. If $W$ is a component of $V_{\alpha}$ where $\ell '(W)\subseteq U_{j,\alpha}$, we define $L$ so that $L|_{W\times \ui}(w,t)=K_{j,\alpha}(\ell '(w),t)$ for $w\in W$, $t\in\ui$.

For each $k\in\bbn$, it is straightforward to check that the projection $\sigma_k\circ L:\ui^n\times \ui\to \tZ_k$ is continuous. By Corollary \ref{continuitycor}, $L:\ui^n\times \ui\to \wh{Z}$ is continuous. Thus $L$ is a homotopy from $\ell '$ to $\ell ''$. 

To check that $\ell ''$ has the desired property, suppose there is a $j\in\bbn$, distinct $\alpha_1,\alpha_2,\alpha_3,\dots \in \ntij$, and points $x_i\in \ui^n$ such that $\ell ''(x_i)\in \mcc_{\infty,j,\alpha_i}\backslash \{c_{\infty,j,\alpha_i}\}$. By our construction of $\ell ''$, we must have $\ell ''(x_i)=\ell '(x_i)=\wt{f}\circ \wt{\kappa}(x_i)\in \mcc_{\infty,j,\alpha_i}$. Thus $\wt{\kappa}(x_i)\in \mca_{\infty,j,\alpha_i}$ in $\wt{Y}$. Replacing $\{x_i\}$ with a subsequence, we may assume $\{x_i\}\to x$ in $\ui^n$. The subsets $\wt{f}(\mcu_{\infty,j,\alpha})$, $\alpha\in\ntij$ of $\tZ$ are all disjoint and open in $\tZ$ and the same is true of the subsets $\mcc_{\infty,j,\alpha}\backslash\{c_{\infty,j,\alpha}\}$. This observation with the fact that the $\alpha_i$ are all distinct, gives $\wt{f}\circ \wt{\kappa}(x)\in \wt{f}(p^{-1}(y_0))$. Since $\wt{\kappa}(x)\in p^{-1}(y_0)$, we have $\kappa(x)=y_0\in Y$. In summary, we have $\kappa(x_i)\to y_0$ in $Y$ where $x_i\in X_{j}$. However, this is a contradiction since $j$ is fixed, $X_j$ is closed in $Y$, and $y_0\notin X_j$. We conclude that $(\ell '')^{-1}(\mcc_{\infty,j,\alpha}\backslash \{c_{\infty,j,\alpha}\})$ is non-empty for finitely many $\alpha$ and has finitely many components when it is non-empty.
\end{proof}

Recall that $\mu_{j,\alpha}:Z\to \mcc_{\infty,j,\alpha}$ denotes the canonical retraction.

\begin{theorem}\label{psiconstructiontheorem}
Let $n\geq 2$. There is a canonical group homomorphism
\[\Psi:\pi_n(\tZ,z_0)\to \prod_{j\in\bbn}\bigoplus_{\alpha\in\ntij}\pi_{n}(\mcc_{\infty,j,\alpha})\]
given by $\Psi([\ell])=([\mu_{j,\alpha}\circ\ell])$. 
\end{theorem}

\begin{proof}
It is clear that $\Psi$ well-defined with codomain $\prod_{j\in\bbn}\prod_{\alpha\in\ntij}\pi_{n}(\mcc_{\infty,j,\alpha})$. Lemma \ref{directsumlemma} ensures that $\Psi$ has image in the subgroup $\prod_{j\in\bbn}\bigoplus_{\alpha\in\ntij}\pi_{n}(\mcc_{\infty,j,\alpha})$.
\end{proof}

Next we characterize the image of $\Psi$ using the topology on $\pi_1(Y)$.

\begin{theorem}\label{imagetheorem}
If $[\ell_{j,\alpha}]\in \pi_{n}(\mcc_{\infty,j,\alpha})$ for $j\in\bbn$ and $\alpha\in\ntij$, then $([\ell_{j,\alpha}])\in \prod_{j\in\bbn}\bigoplus_{\alpha\in\ntij}\pi_{n}(\mcc_{\infty,j,\alpha})$ is in the image of $\Psi$ if and only if the closure of $\bigcup_{j\in\bbn}\{\alpha\in\ntij\mid [\ell_{j,\alpha}]\neq 0\}$ in $\pi_1(Y)$ is compact.
\end{theorem}

\begin{proof}
Suppose $t:(\ui^n,\partial \ui^n)\to (\tZ,\tz_0)$ and let $t_{j,\alpha}=\mu_{j,\alpha}\circ t$. Using the based homotopy inverses $\wt{f}$ and $\wt{g}$, we may assume that $t=\wt{g}\circ \wt{\kappa}$ for $\wt{\kappa}:(\ui^n,\partial \ui^n)\to (\tY,\ty_0)$. Note that $\wt{f}(\alpha)$ lies on the unique arc in $\mathbf{E}_{\infty}$ from $z_0$ to $c_{\infty,j,\alpha}$. Therefore, if $[t_{j,\alpha}]\neq 0$, we must have $\wt{f}(\alpha)\in \im(t)$. Since $\wt{f}$ is bijective on $p^{-1}(y_0)$ (with inverse $\wt{g}$) we must have $\alpha\in \im\left(\wt{\kappa}\right)$. Thus $\bigcup_{j\in\bbn}\{\alpha\in\ntij\mid [t_{j,\alpha}]\neq 1\}\subseteq \im\left(\wt{\kappa}\right)\cap p^{-1}(y_0)=\im(\wt{\kappa})\cap \pi_1(Y)$. Since $\im(\wt{\kappa})$ is compact and $\pi_1(Y)$ is closed in $\tY$, $\im\left(\wt{\kappa}\right)\cap \pi_1(Y)$ is compact as a subspace of $\pi_1(Y)$. It follows that the closure of $\bigcup_{j\in\bbn}\{\alpha\in\ntij\mid [t_{j,\alpha}]\neq 0\}$ in $\pi_1(Y)$ is contained in $\im\left(\wt{\kappa}\right)\cap \pi_1(Y)$ and therefore compact.

For the converse, suppose $(a_{j,\alpha})$ is an element of the codomain such that the closure of $A=\bigcup_{j\in\bbn}\{\alpha\in\ntij\mid a_{j,\alpha}\neq 0\}$ in $\pi_1(Y)$ is compact. If $A$ is finite, then standard methods may be used to construct a map $t:(\ui^n,\partial \ui^n)\to (\tZ,\tz_0)$ such that $\Psi([t])=(a_{j,\alpha})$. From now on, we assume that $A$ is infinite with compact closure. Let $B=A\cup \bigcup_{j\in\bbn}\{\alpha\tau_j\in\tY\mid a_{j,\alpha}\neq 1\}$. We have extended $A$ to $B$ so that $c_{\infty,j,\alpha}\in\wt{f}(B)$ whenever $a_{j,\alpha}\neq 0$.

First, we claim that $\ov{B}$ is compact using a the following fact: for a sequence $\{\alpha_i\}_{i\in\bbn}$ in $\ov{A}$ with $\alpha_i\in\nt_{\infty,j_i}$, we have $\{\alpha_i\}_{i\in\bbn}\to \alpha$ in $\tY$ (and thus $\ov{A}$) if and only if $\{\alpha_i\tau_{j_i}\}_{i\in\bbn}\to \alpha$. The proof of this argument follows the same reasoning used in the proof of Lemma \ref{salemma} so we omit it. To prove compactness, suppose $\scru '$ is an open cover of $\ov{B}$ and let $\scru\subseteq \scru '$ be a finite subset that covers $\ov{A}$. It is enough to show that there are only finitely many points of the form $\alpha\tau_j$ not in $\bigcup\scru$. To obtain a contradiction, suppose $\alpha_{i}\tau_{j_i}\notin \bigcup\scru$ for an infinite sequence $\{(j_i,\alpha_i)\}_{i\in\bbn}$ of distinct pairs. However, since $\ov{A}$ is compact, $\{\alpha_i\}_{i\in\bbn}$ has a convergent subsequence $\{\alpha_{i_m}\}_{m\in\bbn}\to\alpha$ where $\alpha \in \ov{A}$. Thus $\alpha_{i_m}\tau_{j_{i_m}}\to\alpha$ in $\tY$. However, $\alpha\in \bigcup\scru$ and so $\alpha_{i_m}\tau_{j_{i_m}}\in \bigcup\scru$ for sufficiently large $m$; a contradiction.

With the compactness of $\ov{B}$ established, let $\mathbf{D}$ be the union of all arcs in $\mathbf{T}_{\infty}$ from $\ty_0$ the points of $\ov{B}$. Clearly, $\mathbf{D}$ is uniquely arcwise connected. It is straightforward to show that $\mathbf{D}$ is compact using the compactness of $\ov{B}$. It follows that $\mathbf{D}$ and $D=\wt{f}(\mathbf{D})$ are dendrites. By our construction of $B$, we have $c_{\infty,j_i,\alpha_i}\in D$ for all $i\in\bbn$. Thus, by Lemma \ref{salemma}, $\mcp=D\cup \bigcup_{i\in\bbn}\mcc_{\alpha,j_i,\alpha_i}$ is a $D$-subcomplex of $Z$. 

Let $\iota:\mcp\to Z$ be the inclusion and $\Upsilon_{\mcp}:\pi_n(\mcp,z_0)\to \prod_{i\in\bbn}\pi_n(A_i)$ be the canonical homomorphism from Theorem \ref{sequentialpapermainresults} induced by the retractions $\mcp\to A_i$. Let $\Xi:\prod_{i}\pi_n(\mcc_{\infty,j_i,\alpha_i})\to \prod_{j\in\bbn}\bigoplus_{\alpha\in\ntij}\pi_{n}(\mcc_{\infty,j,\alpha})$ be the inclusion map induced by the projections $\prod_{i}\pi_n(\mcc_{\infty,j_i,\alpha_i})\to \pi_{n}(\mcc_{\infty,j_i,\alpha_i})$ (and trivial maps in the other coordinates). It is straightforward to check that the following diagram commutes.
\[\xymatrix{
\pi_n(\mcp,z_0) \ar[r]^-{\Upsilon_{\mcp}} \ar[d]_-{\iota_{\#}} & \prod_{i\in\bbn}\pi_n(\mcc_{\infty,j_i,\alpha_i}) \ar[d]^-{\Xi} \\
\pi_n(Z) \ar[r]_-{\Psi} & \prod_{j\in\bbn}\bigoplus_{\alpha\in\ntij}\pi_{n}(\mcc_{\infty,j,\alpha})
}\]
Since $\Upsilon_{\mcp}$ is surjective by Theorem \ref{sequentialpapermainresults}, we have $\Upsilon_{\mcp}([\ell])=(a_{j_i,\alpha_i})_{i}$ for some $[\ell]\in \pi_n(\mcp,z_0)$. Now $\Psi([\iota\circ\ell])=(a_{j,\alpha})_{j,\alpha}$ for $[\iota\circ\ell]\in \pi_n(Z)$. This completes the proof that $\Psi$ is surjective.
\end{proof}

\begin{problem}
According to Theorem \ref{sequentialpapermainresults}, the homomorphism $\Upsilon_{\mcp}$ always has a section. However, it is not clear to the author if these can be chosen in a coherent way. Does $\Psi:\pi_n(Z)\to \im(\Psi)$ always split?
\end{problem}

\subsection{The injectivity of $\Psi$ when each $\tX_j$ is $(n-1)$-connected}

In this section, we identify a case where $\Psi$ is injective. In particular, we fix $n\geq 2$ and suppose that $\pi_m(Y_j)=0$ for $2\leq m\leq n-1$ or, equivalently, that $\tY_j$ is $(n-1)$-connected. We will directly use the injectivity result in Theorem \ref{sequentialpapermainresults}.

\begin{theorem}\label{psiinjtheorem}
If $\tY_j$ is $(n-1)$-connected for all $j\in\bbn$, then the canonical homomorphism
\[\Psi:\pi_n(Z)\to \prod_{j\in\bbn}\bigoplus_{\alpha\in\ntij}\pi_{n}(\mcc_{\infty,j,\alpha})\]
is injective.
\end{theorem}

\begin{proof}
First, note that since $\tY_j$ is $(n-1)$-connected and $\mcc_{\infty,j,\alpha}\cong \wt{Y}_j/\beta_{\alpha}T_j\simeq \tY_j$, it follows that $\mcc_{\infty,j,\alpha}$ is $(n-1)$-connected.

Let $\ell:(\ui^n,\partial \ui^n)\to (Z,z_0)$ be a map such that $\Psi([\ell])=0$. Since $\im(\ell)$ is a Peano continuum in $Z$, by Corollary \ref{pcinZcorollary}, we may find a $D$-subcomplex $\mcp=D\cup\bigcup_{i\in\bbn}\mcc_{\infty,j_i,\alpha_i}$ containing $\im(\ell)$. In particular, the core is $D=\im(\ell)\cap \bfee_{\infty}$ and the attachment spaces are $\mcc_{\infty,j_i,\alpha_i}$ for a (possibly finite) non-repeating sequence of pairs $(j_i,\alpha_i)$. By assumption the projection $\ell_{j,\alpha}=\mu_{j,\alpha}\circ\ell:(\ui^n,\partial \ui^n)\to (\mcc_{\infty,j,\alpha},c_{\infty,j,\alpha})$ is null-homotopic in $\mcc_{\infty,j,\alpha}$ for all pairs $(j,\alpha)$. In particular, for each $i\in\bbn$, $\ell_{j_i,\alpha_i}$ is null-homotopic in $\mcc_{\infty,j_i,\alpha_i}$. 

Let $\Upsilon_{\mcp}:\pi_n(\mcp,z_0)\to \prod_{i\in\bbn}\pi_n(\mcc_{\infty,j_i,\alpha_i},c_{\infty,j_i,\alpha_i})$ be the canonical homomorphism induced by the retractions $\mcp\to \mcc_{\infty,j_i,\alpha_i}$. According to Theorem \ref{sequentialpapermainresults}, $\Upsilon_{\mcp}$ is an isomorphism. Let $\iota:\mcp\to Z$ be the inclusion map. As in the proof of Theorem \ref{imagetheorem}, let $\Xi$ be canonical the inclusion homomorphism so that the following square commutes.
\[\xymatrix{
\pi_n(\mcp) \ar[r]^-{\Upsilon_{\mcp}}_{\cong} \ar[d]_-{\iota_{\#}} & \prod_{i\in\bbn}\pi_n(\mcc_{\infty,j_i,\alpha_i}) \ar[d]^-{\Xi} \\
\pi_n(Z) \ar[r]_-{\Psi} & \prod_{j\in\bbn}\bigoplus_{\alpha\in\ntij}\pi_{n}(\mcc_{\infty,j,\alpha})
}\]
Viewing $\ell$ as a map $I^n\to P$ with $\iota\circ \ell=\ell$, we have $\Xi\circ \Upsilon_{P}([\ell])=\Psi([\ell])=0$ by assumption. Since $\Xi$ and $\Upsilon_{P}$ are injective, $[\ell]=0$ in $\pi_n(\mcp)$. Thus $[\ell]=0$ in $\pi_n(Z)$.
\end{proof}

\subsection{The homomorphism $\Theta$ and a proof of Theorem \ref{mainthm}}

Finally, we put everything together to prove the main result of this paper.

\begin{proof}[Proof of Theorem \ref{mainthm}]
For each $1\leq j\leq k$ and $\alpha\in\ntkj$, there is a canonical retraction $b_{k,j,\alpha}:\tY_{\leq k}\to \mcb_{k,j,\alpha}$. Fixing $j\in\bbn$, $\alpha\in\ntij$, and letting $\varrho_k(\alpha)=\alpha_{k}'\gamma_{k}$ for $\alpha_{k}'\in\ntkj$, the $\lpc$-coreflection of the inverse limit $\varprojlim_{k\geq j}b_{k,j,\alpha_{k}'}$ gives a canonical retraction $b_{\infty,j,\alpha}:\tY\to \mcb_{\infty,j,\alpha}$. In short, $b_{\infty,j,\alpha}$ maps points outside of $\mcb_{\infty,j,\alpha}$ to the ``nearest" arc-endpoint of $\mcb_{\infty,j,\alpha}$.

Recall that we are still identifying $\tX_j$ as a subspace of $\tY_j$ and $\mca_{\infty,j,\alpha}$ as a subspace of $\mcb_{\infty,j,\alpha}$. The quotient maps $\zeta_j:Y_j\to X_j$ collapsing the attached arc form the quotient map $\zeta:Y\to X$, which induces the homotopy equivalence $\wt{\zeta}:\tY\to \tX$. Since $\wt{\zeta}$ collapses the arcs of each $\mcb_{\infty,j,\alpha}$ and maps $\mca_{\infty,j,\alpha}$ homeomorphically onto its image, we will also write $\mca_{\infty,j,\alpha}$ to denote the subspace $\wt{\zeta}(\mcb_{\infty,j,\alpha})$ of $\tX$. Let $\wt{\zeta}_{\infty,j,\alpha}:\mcb_{\infty,j,\alpha}\to \mca_{\infty,j,\alpha}$ be the corresponding restriction of $\wt{\zeta}$. There is a canonical homeomorphism $\Gamma_{\infty,j,\alpha}:\mca_{\infty,j,\alpha}\to \tX_j$ such that the following square commutes.
\[\xymatrix{
\mcb_{\infty,j,\alpha} \ar[r]^-{\Lambda_{\infty,j,\alpha}} \ar[d]_{\wt{\zeta}_{\infty,j,\alpha}} & \tY_j \ar[d]^{\wt{\zeta}_j} \\
\mca_{\infty,j,\alpha} \ar[r]_-{\Gamma_{\infty,j,\alpha}} & \tX_j 
}\]
In the same way, we constructed the maps $b_{\infty,j,\alpha}$, we may construct canonical maps $a_{\infty,j,\alpha}:\tX\to \mca_{\infty,j,\alpha}$ such that the following square commutes. 
\[\xymatrix{
\tY \ar[r]^-{b_{\infty,j,\alpha}} \ar[d]_{\wt{\zeta}} & \mcb_{\infty,j,\alpha} \ar[d]^{\wt{\zeta}_{\infty,j,\alpha}} \\
\tX \ar[r]_-{a_{\infty,j,\alpha}} & \mca_{\infty,j,\alpha}
}\]
Note that both of the above squares are diagrams of based maps, Although, we identify them in our notation, the basepoint of $\mca_{\infty,j,\alpha}$ is $\alpha\tau_j$ if viewed as a subspace of $\tY$ and $\wt{\zeta}(\alpha)$ if viewed as a subspace of $\tX$. Putting it all together, we consider the following diagram where all products are indexed over pairs $(j,\alpha)$ with $j\in\bbn$ and $\alpha\in\ntij$. For example, $\prod \pi_n(Y_j)$ denotes $\prod_{j\in\bbn}\prod_{\alpha\in\ntij}\pi_n(Y_j)$.
\[\xymatrix{
& \pi_n(Z) \ar[r]^-{\Psi} & \prod \pi_n(\mcc_{\infty,j,\alpha})  & \prod \pi_n(\mcd_{\infty,j,\alpha}) \ar[l]_-{R} \\
\pi_n(Y) \ar[d]_-{\zeta_{\#}} & \pi_n(\tY) \ar[d]^-{\wt{\zeta}_{\#}} \ar[u]_-{\wt{f}_{\#}} \ar[l]_-{p_{\#}} \ar[rr]^-{b} && \prod \pi_n(\mcb_{\infty,j,\alpha}) \ar[u]_-{\prod (f_{\infty,j,\alpha})_{\#}} \ar[r]^-{\Lambda} \ar[d]_{\prod (\wt{\zeta}_{\infty,j,\alpha})_{\#}} & \prod \pi_n(Y_j) \ar[d]^-{\prod \zeta_{j\#}} \\
\pi_n(X) \ar@/_2.5pc/[rrrr]_-{\Theta} & \pi_n(\tX) \ar[l]_-{q_{\#}} \ar[rr]^-{a} && \prod \pi_n(\mca_{\infty,j,\alpha}) \ar[r]^-{\Gamma} & \prod \pi_n(X_j)
}\]
In the above diagram,
\begin{itemize}
\item $R$ is the product of the isomorphisms induced by the retractions $\mcd_{\infty,j,\alpha}\to\mcc_{\infty,j,\alpha}$ that collapse the attached arcs to $c_{\infty,j,\alpha}$.
\item $b$ is induced by the maps $(b_{\infty,j,\alpha})_{\#}:\pi_n(\tY)\to \pi_n(\mcb_{\infty,j,\alpha})$, 
\item  $a$ is induced by the maps $(a_{\infty,j,\alpha})_{\#}:\pi_n(\tX)\to \pi_n(\mca_{\infty,j,\alpha})$, 
\item $\Lambda=\prod (p_j\circ\Lambda_{\infty,j,\alpha})_{\#}$,
\item $\Gamma=\prod (q_j\circ\Gamma_{\infty,j,\alpha})_{\#}$,
\item and $\Theta=\Gamma\circ a\circ q_{\#}^{-1}$.
\end{itemize}
The bottom left square commutes by the definition of $\wt{\zeta}$. Commutativity of the other two bottom squares follows from the two squares given earlier in the proof. Because $\Psi$ is induced by the retractions $Z\to \mcc_{\infty,j,\alpha}$, a direct check shows the top square commutes. Although the upper square depends on the choice of the trees $T_j$, the lower squares do not. Previous results ensure that all maps except for $\Psi$, $a$, and $b$ are isomorphisms.

Recall from Theorem \ref{psiconstructiontheorem} that $\Psi$ has image in $\prod_{j\in\bbn}\bigoplus_{\alpha\in\ntij}\pi_n(\mcc_{\infty,j,\alpha})$. It follows from the diagram that $\Theta$ has image in $\prod_{j\in\bbn}\bigoplus_{\alpha\in\ntij}\pi_n(X_j)$. Finally, recall that there are canonical bijections $\ntij\to\pi_1(Y)/\pi_1(Y_j) \to\pi_1(X)/\pi_1(X_j)$ (the first is the restriction of the projection $\pi_1(Y)\to \pi_1(Y)/\pi_1(Y_j)$ and the second is induced by $\zeta$). Thus we may canonically identify the indexing sets $\ntij$ and $\pi_1(X)/\pi_1(X_j)$. This gives the desired homomorphism $\Theta$ as described in the statement of Theorem \ref{mainthm}.

Lastly, when $\tX_j$ is $(n-1)$-connected for all $j\in\bbn$ $\tY_j$ is $(n-1)$-connected for all $j\in\bbn$. The homomorphism $\Psi$ is injective by Theorem \ref{psiinjtheorem}. It follows from the diagram that $b$ and $a$ are injective. We conclude that $\Theta$ is injective.  
\end{proof}

\begin{remark}[The image of $\Theta$]\label{thetaimageremark}
Combining Theorem \ref{imagetheorem} with the diagram from the proof of Theorem \ref{mainthm}, a direct proof gives the following characterization of $Im(\Theta)$: Letting $H_j=\pi_1(X_j)$, we denote elements of $\pi_1(X)/H_j$ by $\beta H_j$. An element $(\ell_{j,\beta H_j})\in \prod_{j}\bigoplus_{\pi_1(X)/H_j}\pi_n(X_j)$ lies in $Im(\Theta)$ if and only if the closure of $\bigcup_{j\in\bbn}\{\beta\in \pi_1(X)\mid \ell_{j,\beta H_j}\neq 0\}$ in $\pi_1(X)$ (with the whisker topology) is compact.
\end{remark}

With Theorem \ref{mainthm} established, we identifying an alternative description of $\pi_n(X)$ using the $n$-shape homotopy group $\check{\pi}_n(X)=\varprojlim_{k}\pi_n(X_{\leq k})$.

\begin{corollary}
Suppose $\tX_j$ is $(n-1)$ connected for all $j\in\bbn$. Then the canonical homomorphism $\Phi:\pi_n(X)\to \check{\pi}_n(X)$, $\Phi([\ell])=([r_k\circ \ell])$ to the $n$-th shape homotopy group is injective.
\end{corollary}

\begin{proof}
The homotopy equivalence $\zeta:Y\to X$ induces an isomorphism on $\pi_n$ and $\check{\pi}_n$. Since $\Phi$ is natural, it suffices to prove the result for $Y$.

Suppose $0\neq [\ell]\in \pi_n(Y)$. Let $\wt{\ell}:S^n\to \wt{Y}$ be the lift of $\ell$ and $\ell '=\wt{f}\circ\wt{\ell}$. Since $\tf$ and $p$ induce isomorphisms on $\pi_n$, we have $0\neq [\ell ']\in \pi_n(Z)$. By Theorem \ref{injectivetheorem}, there exists $j\in\bbn$ and $\alpha\in\ntij$ such that $0\neq [\mu_{j,\alpha}\circ \ell ']\in \pi_n(\mcc_{\infty,j,\alpha})$. Now $\sigma_j$ maps $\mcc_{\infty,j,\alpha}$ homeomorphically onto $\mcc_{j,j,\alpha '}$ for some $\alpha '\in\nt_{j,j}$ and so $0\neq [\sigma_j\circ \ell ']\in\pi_n(\mcc_{\infty,j,\alpha '})$. Since $\mcc_{\infty,j,\alpha '}$ is a retract of $Z_j$, we have $0\neq [\sigma_j\circ \ell ']\in \pi_n(Z_j)$. 
 \[\xymatrix{
 Y & \wt{Y} \ar[l]_-{p} \ar[d]^-{\wt{f}}  \ar[r]^-{\varrho_j} & \wt{Y}_{\leq j} \ar[d]^-{\wt{f}_j} \\
(I^n,\partial I^n) \ar[u]^-{\ell} \ar[ur]^-{\wt{\ell}} \ar[r]_-{\ell '} & Z \ar[r]_-{\sigma_j} & Z_j}\]
Thus $0\neq [\sigma_j\circ \ell ']=[\sigma_j\circ \wt{f}\circ \wt{\ell}]=[\wt{f}_j\circ \varrho_j\circ \wt{\ell}]$. Since $\wt{f}_j$ is a homotopy equivalence, $0\neq [\varrho_j\circ \wt{\ell}]$ in $\pi_n(\wt{Y}_{\leq j})$. Since $p_{\leq j}:\wt{Y}_{\leq j}\to Y_{\leq j}$ induces an isomorphism on $\pi_n$, we have $0\neq [p_{\leq j}\circ\varrho_j\circ \wt{\ell}]=[r_j\circ p\circ \wt{\ell}]=[r_j\circ\ell]$ in $\pi_n(Y_j)$. We conclude that there exists $j\in\bbn$ such that $[r_j\circ\ell]\neq 0$. Therefore, $\Phi([\ell])\neq 0$.
\end{proof}

\begin{remark}
Recall from the introduction that standard homotopy theory gives $\pi_n(X_{\leq k})\cong \bigoplus_{1\leq j\leq k}\bigoplus_{\pi_1(X_{\leq k})/\pi_1(X_j)}\pi_n(X_j)$ when each $\tX_j$ is $(n-1)$-connected. Under this hypothesis, the previous corollary provides a canonical injection of $\pi_n(X)$ into \[\varprojlim_{k\in\bbn} \left(\bigoplus_{1\leq j\leq k}\bigoplus_{\pi_1(X_{\leq k})/\pi_1(X_j)}\pi_n(X_j)\right).\]We point out that this inverse limit does not simply give the product over $k$ because the bonding maps are not product-projections. Rather, the bonding maps correspond to the induced homomorphisms $(\wt{q}_{k+1,k})_{\#}:\pi_n(\wt{X}_{\leq k+1})\to \pi_n(\tX_{\leq k})$.
\end{remark}

\subsection{The aspherical case}\label{sectionaspherical}

A path-connected space $X$ is \textit{aspherical} if $\pi_n(X)=0$ for all $n\geq 2$. If $X$ admits a generalized universal covering $\tX$, then $\tX$ has trivial homotopy groups. If $X$ is an aspherical CW-complex, then $\tX$ is contractible.

\begin{lemma}\label{zcontractiblelemma}
If $Y_j$ is aspherical for every $j\in\bbn$, then $\bfee_{\infty}$ is a deformation retract of $\tZ$. In particular, $\mu:\tZ\to\bfee_{\infty}$ is homotopic to $id_{\tZ}$.
\end{lemma}

\begin{proof}
Suppose $Y_j$ is aspherical for every $j\in\bbn$. Then each universal covering space $\tY_j$ is contractible. For all $j\in\bbn$ and $\alpha\in\ntij$, we have \[\mcc_{\infty,j,\alpha}=\mcb_{\infty,j,\alpha}/\mct_{\infty,j,\alpha}\cong \tY_j/\beta_{\alpha}T_j\simeq \tY_j,\]Thus $\mcc_{\infty,j,\alpha}$ is contractible. For each $\beta\in\pi_1(Y_j)$, fix a based contraction $K_{j,\beta}:C_{j,\beta}\times\ui\to C_{j,\beta}$, i.e. a based homotopy from $id_{C_{j,\beta}}$ to the constant map at $c_{j,\beta}$. Define contractions $L_{k}:Z_k\times \ui\to Z_k$ so that $L_k(z,t)=z$ for $(z,t)\in\bfee_k\times\ui$ and so that the restriction of $L_k$ to the subcomplex $\mcc_{k,j,\alpha}\times \ui$ is a map $\mcc_{k,j,\alpha}\times \ui\to \mcc_{k,j,\alpha}$ making the following square commute when $\Lambda_{k,j,\alpha}(\mct_{k,j,\alpha})=\beta T_j$.
\[\xymatrix{
\mcc_{k,j,\alpha}\times \ui \ar[r]^-{L_k} \ar[d]_-{\lambda_{k,j,\alpha}\times id} & \mcc_{k,j,\alpha} \ar[d]^-{\lambda_{k,j,\alpha}} \\
C_{j,\beta} \times\ui \ar[r]_-{K_{j,\beta}} & C_{j,\beta}
}\]
The maps $L_k$ are clearly well-defined and since $\tZ_k\times\ui$ is a CW-complex, $L_k$ is continuous. The same inductive argument used to construct the maps $\wt{g}_{k,j,\alpha}$ can be used to show that $L_{k}\circ (\wt{s}_{k+1,k}\times id)=\wt{s}_{k+1,k}\circ L_{k+1}$. Let $\wh{L}_0:\wh{Z}_0\times \ui\to \wh{Z}_0$ be the restriction of the inverse limit map $\wh{L}=\varprojlim_{k}L_k:\wh{Z}\times\ui\to \wh{Z}$.

Define $L:Z\times\ui\to Z$ as follows: $L(z,t)=z$ for $(z,t)\in \bfee_{\infty}\times\ui$. For each $j\in\bbn$ and $\alpha\in\ntij$, we define the restriction of $L$ to the subspace $\mcc_{\infty,j,\alpha}\times \ui$ to be the based contraction $\mcc_{\infty,j,\alpha}\times \ui\to \mcc_{\infty,j,\alpha}$ which makes the following square commute.
\[\xymatrix{
\mcc_{\infty,j,\alpha}\times \ui \ar[r]^-{L} \ar[d]_-{\lambda_{\infty,j,\alpha}\times id} & \mcc_{\infty,j,\alpha} \ar[d]^-{\lambda_{\infty,j,\alpha}} \\
C_{j,\beta_{\alpha}} \times\ui \ar[r]_-{K_{j,\beta_{\alpha}}} & C_{j,\beta_{\alpha}}
}\]
Clearly $L$ is well-defined, $L(z,0)=z$ and $L(z,1)=\mu(z)$. We have constructed $L$ so that $\sigma_k\circ L=L_k\circ (\sigma_k\times id)$ for all $k\in\bbn$ (details required to verify this are identical to previous arguments, e.g. the construction of $\wt{g}_k$). Since $Z\times\ui$ is locally path connected and the projections $\sigma_k\circ L$ are continuous, $L$ is continuous by Corollary \ref{continuitycor}.
\end{proof}

\begin{proof}[Proof of Theorem \ref{asphericalcorollary}]
Suppose each $X_j$ is aspherical. Then each $Y_j$ is aspherical. We have the following sequence of maps.
\[\xymatrix{
\tX  & \tY \ar[l]_-{\wt{\zeta}} \ar[r]^-{\wt{f}} & Z \ar[r]^-{\mu} & \bfee_{\infty} \ar[r]^-{\wt{g}|_{\bfee_{\infty}}} & \bft_{\infty}
}\]
All of these maps except for $\mu$ are always homotopy equivalences. Since each $Y_j$ is aspherical, Lemma \ref{zcontractiblelemma} implies that $\mu$ is a homotopy equivalence. Since every uniquely arcwise connected Hausdorff space is aspherical, $\bft_{\infty}$ is aspherical. Thus $\tX$ is aspherical. Moreover, if each $X_j$ is locally finite, Corollary \ref{tiscontractible} implies that $\bft_{\infty}$ is contractible. Hence, in this case, the sequence gives that $\tX$ is contractible.
\end{proof}

\section*{Acknowledgment}

The author would like to acknowledge the research support provided by Double Stuf Oreos, without which this project would not have been possible.

\end{document}